\documentclass{amsart}
\usepackage{amssymb}
\usepackage{amsmath,amsaddr}
\usepackage{mathrsfs}
\usepackage{amsfonts}
\allowdisplaybreaks[2]
\usepackage{mathtools}
\usepackage{ulem}
\usepackage{braket}
\usepackage{hyperref}
\usepackage{aliascnt}
\hypersetup{
    colorlinks=true,
    citecolor=blue,
    linkcolor=blue,
    urlcolor=blue,
}
\renewcommand{\eqref}[1]{\textcolor{blue}{(\ref{#1})}}
\usepackage{amscd}
\usepackage{tikz}
\usepackage{tikz-cd}
\usetikzlibrary{calc}
\usetikzlibrary{matrix,arrows}
\usetikzlibrary{decorations.markings}

\usepackage{amsthm}
\theoremstyle{plain}

\newtheorem{thm}{Theorem}[section]

\newaliascnt{prop}{thm}
\newtheorem{prop}[prop]{Proposition}
\aliascntresetthe{prop}

\newaliascnt{cor}{thm}
\newtheorem{cor}[cor]{Corollary}
\aliascntresetthe{cor}

\newaliascnt{lem}{thm}
\newtheorem{lem}[lem]{Lemma}
\aliascntresetthe{lem}

\newtheorem{thmA}{Theorem}[section]

\theoremstyle{definition}

\newaliascnt{defi}{thm}
\newtheorem{defi}[defi]{Definition}
\aliascntresetthe{defi}

\newaliascnt{rem}{thm}
\newtheorem{rem}[rem]{Remark}
\aliascntresetthe{rem}

\newaliascnt{reconstruction}{thm}
\newtheorem{reconstruction}[reconstruction]{Reconstruction}
\aliascntresetthe{reconstruction}

\newtheorem*{exam*}{Example}
\newtheorem*{rrem*}{Remark}
\newtheorem*{defi*}{Definition}

\usepackage[]{enumitem}

\makeatletter
\AddEnumerateCounter{\fnsymbol}{\c@fnsymbol}{9}
\makeatother

\usepackage{moreenum}

\usepackage{latexsym}
\DeclareMathOperator{\Hom}{Hom}
\DeclareMathOperator{\Isom}{Isom}
\DeclareMathOperator{\ISOM}{\mathbf{Isom}}
\DeclareMathOperator{\id}{\mathrm{id}}
\DeclareMathOperator{\im}{\mathrm{Im}}
\DeclareMathOperator{\Spec}{\mathrm{Spec}}
\DeclareMathOperator{\Aut}{\mathrm{Aut}}

\newcommand{\bbbullet}{\lozenge}
\newcommand{\bbullet}{\blacklozenge}

\DeclareMathOperator{\colim}{\mathrm{colim}}
\DeclareMathOperator{\plim}{\mathrm{lim}}
\newcommand{\bcolim}{\colim^\bbullet}
\newcommand{\bplim}{\plim^\bbullet}
\newcommand{\bcoprod}{\coprod^\bbullet}
\newcommand{\btimes}{\times^\bbullet}

\newcommand{\op}{\mathrm{op}}

\newcommand{\red}{\mathrm{red}}
\newcommand{\qcpt}{\mathrm{qcpt}}
\newcommand{\qsep}{\mathrm{qsep}}
\newcommand{\sep}{\mathrm{sep}}
\newcommand{\rqqs}{\left\{ \red , \qcpt , \qsep , \sep \right\}}

\newcommand{\rsa}{\rightsquigarrow}
\renewcommand{\emptyset}{\varnothing}

\newcommand{\dfn}{:\overset{\mathrm{def}}{=}}
\newcommand{\deff}{:\hspace{-3pt}\overset{\text{def}}{\iff}}

\newcommand{\univ}[1]{\mathbf{#1}}
\newcommand{\usm}{\(\univ{U}\)-small}
\newcommand{\vsm}{\(\univ{V}\)-small}

\DeclareMathOperator{\Sche}{\mathsf{Sch}}
\newcommand{\Sch}[1]{\Sche_{/#1}}
\newcommand{\Schb}[1]{\Sche_{\bbullet/#1}}
\newcommand{\Schc}[1]{\Sche_{\bbbullet/#1}}

\DeclareMathOperator{\sfSet}{\mathsf{Set}}
\newcommand{\SetU}{\sfSet}
\newcommand{\SetV}{\sfSet_{\univ{V}}}
\DeclareMathOperator{\sfTop}{\mathsf{Top}}
\newcommand{\TopU}{\sfTop}
\newcommand{\TopV}{\sfTop_{\univ{V}}}
\DeclareMathOperator{\sfSch}{\mathsf{Sch}}
\newcommand{\SchU}{\Sche}
\newcommand{\SchV}{\Sche_{\univ{V}}}

\newcommand{\Alg}[1]{{#1}\text{-}\mathsf{Alg}}
\DeclareMathOperator{\LAlg}{\mathsf{LAlg}}


\begin{document}


\begin{title}[Categories of Reduced Schemes]{Category-Theoretic Reconstruction of Schemes
  from Categories of Reduced Schemes}\end{title}

\author{Tomoki Yuji}
\date{\today}
\email{\href{mailto:math@yujitomo.com}{\texttt{math@yujitomo.com}}}


\begin{abstract}
  Let \(S\) be a locally Noetherian normal scheme
  and \(\bbullet/S\) a set of properties of \(S\)-schemes.
  Then we shall write \(\Schb{S}\) for
  the full subcategory of the category of \(S\)-schemes \(\Sch{S}\)
  determined by the objects \(X\in \Schb{S}\) that satisfy every property of \(\bbullet/S\).
  In the present paper, we shall mainly be concerned with the properties
  ``reduced'', ``quasi-compact over \(S\)'', ``quasi-separated over \(S\)'', and ``separated over \(S\)''.
  We give a functorial category-theoretic algorithm for reconstructing \(S\)
  from the intrinsic structure of the abstract category \(\Schb{S}\).
  This result is analogous to a result of Mochizuki \cite{Mzk04}
  and may be regarded as a partial generalization of a result of de Bruyn \cite{deBr19}
  in the case where \(S\) is a locally Noetherian normal scheme.
\end{abstract}

\maketitle

\setcounter{tocdepth}{1}
\tableofcontents


\section*{Introduction}

Let \(\univ{U}\) and \(\univ{V}\) be Grothendieck universes such that \(\univ{U} \in \univ{V}\).
Let \(S\) be a \usm \ scheme.
In the following, we shall use the term ``scheme'' to refer to a \usm \ scheme.
Let \(\bbullet/S\) be a (\vsm) set of properties of \(S\)-schemes.
We shall write
\[\Schb{S}\]
for the full subcategory of the (\vsm) category of \(S\)-schemes \(\Sch{S}\)
determined by the objects \(X\in \Schb{S}\) that satisfy every property of \(\bbullet/S\).
In the present paper, we shall mainly be concerned with the properties
\[\text{``red'', ``qcpt'', ``qsep'', ``sep''}\]
of \(S\)-schemes, i.e.,
``reduced'', ``quasi-compact over \(S\)'',
``quasi-separated over \(S\)'', and ``separated over \(S\)''.
If \(\bbullet/S = \emptyset\), then we simply write \(\Sch{S}\) for \(\Schb{S}\).

In the present paper,
we consider the problem of reconstructing the scheme \(S\)
from the intrinsic structure of the abstract category \(\Schb{S}\).
In \cite{Mzk04}, Mochizuki gave a solution to this problem
in the case where \(S\) is locally Noetherian,
and \(\bbullet/S = \text{``of finite type over \(S\)''}\).
In \cite{deBr19}, van Dobben de Bruyn gave a solution to this problem
in the case where \(S\) is an arbitrary scheme,
and \(\bbullet/S = \emptyset\).
The techniques applied in \cite{deBr19} make essential use of
the existence of \textit{non-reduced schemes} in \(\Sch{S}\).
By contrast, in the present paper, we focus on the problem of reconstructing the scheme \(S\)
from categories of \(S\)-schemes that only contain \textit{reduced \(S\)-schemes},
hence rely on techniques that differ essentially from the techniques applied in \cite{deBr19}.

If \(X,Y\) are objects of a (\vsm) category \(\mathcal{C}\),
then we shall write \(\Isom (X,Y)\) for the set of isomorphisms from \(X\) to \(Y\).
By a slight abuse of notation, we shall also regard this set
as a discrete category.
If \(\mathcal{C,D}\) are (\vsm) categories,
then we shall write \(\ISOM (\mathcal{C,D})\) for the (\vsm) category
of equivalences \(\mathcal{C} \xrightarrow{\sim} \mathcal{D}\) and natural isomorphisms.
If \(\mathcal{C}\) is a (\vsm) category, and \(X\) is an object of \(\mathcal{C}\),
then we shall write \(\mathcal{C}_{/X}\) for the slice category of objects and morphisms
equipped with a structure morphism to \(X\).
If \(f:X\to Y\) is a morphism in a (\vsm) category \(\mathcal{C}\)
which is closed under fiber products,
then we shall write \(f^* : \mathcal{C}_{/Y}\to \mathcal{C}_{/X}\)
for the functor induced by the
operation of base-change, via \(f\), from \(X\) to \(Y\).

Our main result is the following:

\begin{thmA}\label{thmA: main thm}
  \
  \begin{enumerate}
    \item (\autoref{reconstruction: Sch}) \\
    Let \(S\) be a locally Noetherian normal (\usm) scheme,
    \[\bbullet \subset \rqqs\] a [possibly empty] subset
    such that \(\left\{ \mathrm{qsep,sep}\right\} \not\subset \bbullet\).
    Then the following may be constructed category-theoretically from \(\Schb{S}\)
    by means of algorithms that are independent of the choice of the subset
    \(\bbullet \subset \{\mathrm{red, qcpt, qsep,}\) \(\mathrm{sep}\}\):
    \begin{enumerate}
      \item
      for each object \(T\) of \(\Schb{S}\),
      a (\vsm) scheme \(T_{\univ{V}}\)
      and an isomorphism of (\vsm) schemes \(\eta_{T}:T_{\univ{V}}\xrightarrow{\sim} T\)
      (where we note that a \usm \ scheme is, in particular, \vsm), and
      \item
      for each morphism \(f:T_1\to T_2\) of \(\Schb{S}\),
      a morphism of (\vsm) schemes \(f_{\univ{V}}:T_{1,\univ{V}} \to T_{2,\univ{V}}\)
      such that \(\eta_{T_2}\circ f_{\univ{V}} = f \circ \eta_{T_1}\).
    \end{enumerate}
    \item (\autoref{cor: bbullet equal circ}) \\
    Let \(S,T\) be quasi-separated (\usm) schemes,
    \[\bbullet , \bbbullet \subset \rqqs\]
    [possibly empty] subsets
    such that \(\left\{ \mathrm{qsep,sep}\right\} \not\subset \bbullet, \ \left\{ \mathrm{qsep,sep}\right\} \not\subset\bbbullet\).
    Then if the (\vsm) categories \(\Schb{S}\), \(\Schc{T}\) are equivalent,
    then \(\bbullet = \bbbullet\).
    \item (\autoref{thm: isom cat equiv}) \\
    Let \(S,T\) be locally Noetherian normal (\usm) schemes,
    \[\bbullet  \subset \rqqs\]
    a [possibly empty] subset
    such that \(\left\{ \mathrm{qsep,sep}\right\} \not\subset \bbullet\).
    Then the natural functor
    \begin{align*}
      \Isom (S,T) &\to \ISOM (\Schb{T},\Schb{S}) \\
      f &\mapsto f^*
    \end{align*}
    is an equivalence of (\vsm) categories.
  \end{enumerate}
\end{thmA}

Our proof of \autoref{thmA: main thm} proceeds by giving
category-theoretic characterizations of
various properties of schemes and morphisms of schemes as follows:
\begin{itemize}
  \item
  In \autoref{section: underlying set},
  we give a category-theoretic characterization of
  the objects of \(\Schb{S}\) whose underlying scheme is isomorphic to the spectrum of a field
  and use this characterization to
  give a functorial category-theoretic algorithm for reconstructing
  the underlying set of the objects of \(\Schb{S}\)
  from the intrinsic structure of the abstract category \(\Schb{S}\).
  \item
  In \autoref{section: regular mono},
  we consider various scheme-theoretic properties of regular monomorphisms in \(\Schb{S}\).
  These properties will play an important role in
  \autoref{section: local domains} and \autoref{section: undetlying top}.
  \item
  In \autoref{section: local domains},
  we give a category-theoretic characterization of the objects of
  \(\Schb{S}\) whose underlying scheme is isomorphic to the spectrum of a local domain.
  This characterization is closely related to the category-theoretic
  characterization of specialization and generization relations
  between points of a scheme and
  also yields a category-theoretic characterization of
  the objects of \(\Schb{S}\) whose underlying scheme is irreducible,
  as well as the objects of \(\Schb{S}\)
  whose underlying scheme is isomorphic to the spectrum of a local ring.
  \item
  In \autoref{section: undetlying top},
  by using the method of \cite[Lemma 3.7]{deBr19},
  we give a category-theoretic characterization of the morphisms of
  \(\Schb{S}\) whose underlying morphism of schemes is a closed immersion.
  We then use this characterization
  to give
  a functorial category-theoretic algorithm for
  reconstructing the underlying topological space of the objects of \(\Schb{S}\)
  from the intrinsic structure of the abstract category \(\Schb{S}\).
  \item
  In \autoref{section: loc of fp},
  we give a category-theoretic characterization of the morphisms of
  \(\Schb{S}\) whose underlying morphism of schemes is locally of finite presentation.
  This characterization also yields a category-theoretic characterization of
  the morphisms of \(\Schb{S}\)
  whose underlying morphism of schemes is proper and of finite presentation.
  These properties form the foundation for the discussion in \autoref{section: P1}.
  \item
  In \autoref{section: P1},
  we give a category-theoretic characterization of the objects of
  \(\Schb{S}\) whose underlying \(S\)-scheme is isomorphic to the projective line \(\mathbb{P}^1_S\).
  We then use this characterization
  to give a functorial category-theoretic algorithm for
  reconstructing the underlying schemes of the objects of \(\Schb{S}\)
  from the intrinsic structure of the abstract category \(\Schb{S}\).
  In addition,
  we discuss some rigidity results related to the various reconstruction algorithms developed
  in the present paper.
\end{itemize}

Finally, we remark that, in the above discussion,
as well as throughout the present paper,
\textit{category-theoretic reconstruction algorithms} are
to be understood as algorithms that are \textit{independent} of the choice
of the subset \(\bbullet \subset \rqqs\).


\subsection*{Acknowledgements}

I would like to thank my advisor S. Mochizuki for constant support and helpful conversations.
Also I would like to thank Professor Y. Hoshi and Professor S. Yasuda for giving me advice on this paper and my research.


\subsection*{Notations and Conventions}\label{notation}

\

Let \(\univ{U},\univ{V}\) be Grothendieck universes such that \(\univ{U}\in \univ{V}\).
The notation \(\mathbb{Z}\) will be used to denote the ring of integers.
We shall use the notation
\[\mathsf{Set}_\univ{V},\ \mathsf{Top}_\univ{V},\ \mathsf{Ring}_\univ{V},\ \mathsf{Sch}_\univ{V}\]
to denote, respectively, the categories of \vsm \ sets,
\vsm \ topological spaces, \vsm \ rings,
and \vsm \ schemes.
We shall use similar notation with the subscript \(\univ{V}\) omitted to denote
the corresponding category of \usm \ objects.
In the following, we shall use the term ``scheme'' to refer to a \usm \ scheme.

Let \(\mathcal{C}\) be a (\vsm) category.
We shall write \(\mathcal{C}^{\mathrm{op}}\) for
the opposite category associated to \(\mathcal{C}\).
If \(X\) is an object of \(\mathcal{C}\),
then we shall write \(\mathcal{C}_{/X}\) for the slice category of objects and morphisms
equipped with a structure morphism to \(X\).

Let \(\bbullet/S\) be a (\vsm) set of properties of \(S\)-schemes.
Then we shall write
\[\Schb{S}\]
for the full subcategory of the (\vsm) category of \(S\)-schemes \(\Sch{S}\)
determined by the objects \(X\) of \(\Schb{S}\) that satisfy every property of \(\bbullet/S\).
In the present paper, we shall mainly be concerned with the properties
\[\text{``red'', ``qcpt'', ``qsep'', ``sep''}\]
of \(S\)-schemes, i.e., ``reduced'',
``quasi-compact over \(S\)'', ``quasi-separated over \(S\)'', and ``separated over \(S\)''.
Thus, if \(\bbullet/S = \emptyset\), then \(\Sch{S}=\Schb{S}\).
In the remainder of the present paper,
\begin{quote}
  we \textit{fix} subsets \(\bbullet, \bbbullet \subset
  \left\{ \mathrm{red,qcpt,qsep,sep} \right\}\)
  such that \(\left\{ \mathrm{qsep,sep}\right\} \not\subset
  \bbullet, \ \left\{ \mathrm{qsep,sep}\right\} \not\subset\bbbullet\).
\end{quote}

We shall write \(\bplim, \bcolim, \btimes, \bcoprod \)
for the (inverse) limit, colimit, fiber product, and push-out in \(\Schb{S}\).
We shall write \(\plim, \colim, \times, \coprod \)
for the (inverse) limit, colimit, fiber product, and push-out in \(\mathsf{Sch}\).
By a slight abuse of notation,
we shall use the notation \(\emptyset\) to denote the initial object
(i.e., the empty scheme) of \(\Schb{S}\) or \(\Sch{S}\).

Let \(S\) be a scheme,
\(Z'\to Z\) a morphism in \(\Schb{S}\), and
\(f:X\to Y\) a morphism over \(Z\) in \(\Schb{S}\).
We shall write \(X_{\bbullet,Z'}\) for
the base-change (if it exists) of \(X\) by the morphism \(Z'\to Z\) in \(\Schb{S}\).
We shall write \(f_{\bbullet,Z'}\) for
the base-change (if it exists) of \(f\) by the morphism \(Z'\to Z\) in \(\Schb{S}\).
Let \(F \in \{X,f\}\).
If \(Z' = \Spec(A)\), then we shall also write \(F_{\bbullet,A}\) for \(F_{\bbullet,Z'}\).
If \(Z' = \Spec(k(z))\), where \(k(z)\) denotes the residue field of \(z\in Z\),
then we shall also write \(F_{\bbullet,z}\) for \(F_{\bbullet,Z'}\).
If \(\bbullet = \emptyset\),
then we shall simply write \(F_{Z'}\) for \(F_{\bbullet,Z'}\).

If \(f:X\to Y\) is a morphism of schemes, and
\(Z = \Spec(k(x))\) is the spectrum of
the residue field \(k(x)\) of \(X\) at a point \(x\in X\),
we shall write \(f|_x\) for the composite of
the natural morphism \(Z \to X\) with \(f\).
If \(f:\Spec(A)\to \Spec(B)\) is a morphism of affine schemes,
then we shall write \(f^{\#} : B\to A\) for the ring homomorphism induced by \(f\).
By a slight abuse of notation, if \(f^{\#}:B\to A\) is a ring homomorphism,
then we shall use the notation \(f\) to denote
the corresponding morphism of schemes \(\Spec(A)\to \Spec(B)\).

Let \(X\) be a scheme.
Then we shall write \(\mathcal{O}_X\) for the structure sheaf of \(X\).
We shall write \(|X|\) for the underlying topological space of \(X\).
If the underlying set of \(|X|\) is of cardinality \(1\),
then we shall write \(*_X\) for the unique element of the underlying set of \(X\).
If \(F\subset |X|\) is a closed subset,
then we shall write \(F_{\red}\) for
the reduced induced closed subscheme determined by \(F\) and \(X_{\red}\dfn |X|_{\red}\).
If \(f:X\to Y\) is a morphism of schemes,
then we shall write \(f_{\red}:X_{\red}\to Y_{\red}\) for the morphism induced by \(f\).
For points \(x,y\in X\),
if \(x\) is a specialization of \(y\)
(where we recall that this includes the case where \(x=y\)),
then we shall write \(y\rsa x\).

Let \(S\) be a scheme;
\(\bbullet \subset \left\{ \mathrm{red,qcpt,qsep,sep} \right\}\) a subset such that
\(\left\{ \mathrm{qsep,sep}\right\} \not\subset \bbullet\);
\(I,J\) (\usm) sets;
\(\{X_i\}_{i\in I}\) a (\usm) family of objects of \(\Schb{S}\) parametrized by \(I\);
\(\{f_j\}_{j\in J}\) a (\usm) family of morphisms in \(\Schb{S}\) parametrized by \(J\);
\(\mathcal{P}\) a property of schemes;
\(\mathcal{Q}\) a property of (\usm) families of schemes parametrized by \(I\) and
(\usm) families of morphisms of schemes parametrized by \(J\).
Assume that \(S\) satisfies property \(\mathcal{P}\).
Then we shall say that
\begin{quote}
  the property that
  \[(\{X_i\}_{i\in I}, \{f_j\}_{j\in J})\]
  satisfies \(\mathcal{Q}\)
  may be \textit{characterized category-theoretically}
  from the data
  \[(\Schb{S}, \{X_i\}_{i\in I}, \{f_j\}_{j\in J})\]
\end{quote}
if for any (\usm) family of objects \(\{Y_i\}_{i\in I}\) of \(\Schb{S}\) parametrized by \(I\),
any (\usm) family of morphisms \(\{g_j\}_{j\in J}\) of \(\Schb{S}\) parametrized by \(J\),
any scheme \(T\) which satisfies property \(\mathcal{P}\),
any subset \(\bbbullet \subset \left\{ \mathrm{red,qcpt,qsep,sep} \right\}\)
such that \(\left\{ \mathrm{qsep,sep}\right\} \not\subset \bbbullet\), and
any equivalence \(F:\Schb{S}\xrightarrow{\sim}\Schc{T}\),
\begin{center}
  \((\{Y_i\}_{i\in I}, \{f_j\}_{j\in J})\) satisfies \(\mathcal{Q}\)
  \(\iff\) \((\{F(Y_i)\}_{i\in I}, \{F(g_j)\}_{j\in J})\) satisfies \(\mathcal{Q}\).
\end{center}

\section{The Underlying Set}
\label{section: underlying set}

In this section,
we give a functorial category-theoretic algorithm for reconstructing
the underlying set of the objects of \(\Schb{S}\)
from the intrinsic structure of the abstract category \(\Schb{S}\).

First, we note the following properties of fiber products in \(\Schb{S}\):

\begin{lem}\label{lem: fiber product bbullet}
  Let \(S\) be a scheme.
  Let \(Y\to X\) and \(Z\to X\) be morphisms in \(\Schb{S}\).
  \begin{enumerate}
    \item \label{enumi: fiber product bbullet not qcpt and red}
    If \(\qcpt,\red \not\in \bbullet\),
    then \(Y\times_XZ\) belongs to \(\Schb{S}\).
    In particular,
    the fiber product \(Y\btimes_XZ\) exists in \(\Schb{S}\) and
    is naturally isomorphic to \(Y\times_XZ\).
    \item \label{enumi: fiber product bbullet not qcpt and in red}
    If \(\qcpt\not\in \bbullet\) and \(\red\in \bbullet\),
    then \((Y\times_XZ)_{\red}\) belongs to \(\Schb{S}\).
    In particular,
    the fiber product \(Y\btimes_XZ\) exists in \(\Schb{S}\) and
    is naturally isomorphic to \((Y\times_XZ)_{\red}\).
  \end{enumerate}
  Assume that either \(Y\to X\) or \(Z\to X\) is quasi-compact.
  Then the following hold:
  \begin{enumerate}[start=3]
    \item \label{enumi: fiber product bbullet not red}
    If \(\red \not\in \bbullet\),
    then \(Y\times_XZ\) belongs to \(\Schb{S}\).
    In particular,
    the fiber product \(Y\btimes_XZ\) exists in \(\Schb{S}\) and
    is naturally isomorphic to \(Y\times_XZ\).
    \item \label{enumi: fiber product bbullet red}
    If \(\red \in \bbullet\),
    then \((Y\times_XZ)_{\red}\) belongs to \(\Schb{S}\).
    In particular,
    the fiber product \(Y\btimes_XZ\) exists in \(\Schb{S}\) and
    is naturally isomorphic to \((Y\times_XZ)_{\red}\).
  \end{enumerate}
\end{lem}

\begin{proof}
  If \(\qcpt\in \bbullet\)
  and either \(Y\to X\) or \(Z\to X\) is quasi-compact,
  then it is immediate that \(Y\times_XZ\) is quasi-compact over \(S\).
  Recall that the natural morphism \(Y\times_XZ \to Y\times_SZ\) is an immersion,
  hence separated.
  Thus, if \(\sep \in \bbullet\) (respectively, \(\qsep\in \bbullet\)),
  then since \(Y\times_SZ\) is separated (respectively, quasi-separated) over \(S\),
  \(Y\times_XZ\) is separated (respectively, quasi-separated) over \(S\).

  Thus, if \(\red \not\in \bbullet\),
  then the fiber product \(Y\times_X Z\) belongs to \(\Schb{S}\),
  which implies that \(Y\btimes_XZ\) exists in \(\Schb{S}\) and
  is naturally isomorphic to \(Y\times_XZ\).
  This completes the proof of
  \ref{enumi: fiber product bbullet not qcpt and red} and
  \ref{enumi: fiber product bbullet not red}.

  Suppose that \(\red\in \bbullet\).
  Then any morphism \(f:W\to Y\times_XZ\) from a reduced scheme \(W\)
  factors uniquely through the closed subscheme
  \((Y\times_XZ)_{\red} \subset Y\times_XZ\).
  Thus, since \((Y\times_X Z)_{\red}\) belongs to \(\Schb{S}\),
  we conclude that
  \(Y\btimes_XZ\) exists in \(\Schb{S}\)
  and is naturally isomorphic to \((Y\times_X Z)_{\red}\).
  This completes the proof of
  \ref{enumi: fiber product bbullet not qcpt and in red} and
  \ref{enumi: fiber product bbullet red}.
\end{proof}

Next,
we consider a category-theoretic characterization of the objects of
\(\Schb{S}\) whose underlying set is of cardinality \(1\).

\begin{lem}\label{lem: 1pt}
  Let \(S\) be a scheme.
  Let \(X\) be a non-initial object of \(\Schb{S}\).
  Then \(|X|\) is of cardinality \(1\)
  if and only if the following condition does \textbf{not} hold:
  \begin{enumerate}[label=(\fnsymbol*),start=2]
    \item \label{enumi: 1pt fiber empty}
    there exist objects \(Y,Z \neq \emptyset\) in \(\Schb{X}\) such that \(Y\btimes_X Z = \emptyset\).
  \end{enumerate}
  In particular,
  the property that
  \(|X|\) is of cardinality \(1\)
  may be characterized category-theoretically from the data \((\Schb{S},X)\).
\end{lem}

\begin{proof}
  Assume that there exist two points \(x_1,x_2\in X\) such that \(x_1\neq x_2\).
  Then \(\Spec k(x_1) \btimes_X \Spec k(x_2) = \emptyset\) (where we note that both \(\Spec k(x_1)\) and \(\Spec k(x_2)\) belong to \(\Schb{S}\)).
  Hence \(X\) satisfies the condition \ref{enumi: 1pt fiber empty}.

  Assume that \(X\) satisfies the condition \ref{enumi: 1pt fiber empty},
  i.e., that there exist \(f:Y\to X, g:Z\to X\) such that
  \(Y\btimes_X Z = \emptyset\) and \(Y,Z \neq \emptyset\).
  Since \(Y,Z \neq \emptyset\),
  we may assume without loss of generality
  that \(|Y|,|Z|\) are of cardinality \(1\).
  To prove that \(X\) has two distinct points,
  it suffices to prove that \(f(Y) \cap g(Z) = \emptyset\).
  Observe that it follows from \autoref{lem: fiber product bbullet}
  that \((Y\times_XZ)_{\red} = (Y\btimes_XZ)_{\red} = \emptyset\).
  Hence for any point \(x\in X\),
  the equality
  \[(Y_x\times_{\Spec k(x)} Z_x)_{\red} = ((Y\times_X Z)_{\red})_{x,\red} = \emptyset\]
  holds.
  On the other hand, if \(f(Y) \cap g(Z) \neq \emptyset\),
  then for any point \(x\in f(Y)\cap g(Z)\),
  it holds that \((Y_x \times_{\Spec k(x)} Z_x)_{\red} \neq \emptyset\),
  in contradiction to the displayed equality.
\end{proof}

Next,
we consider a category-theoretic reconstruction of
the underlying sets of the objects of \(\Schb{S}\).

\begin{reconstruction}\label{reconstruction: Set}
  Let \(S\) be a scheme.
  Let \(X\) be an object of \(\Schb{S}\).
  We define a (\vsm) set
  \[\mathsf{Pt}_{\bbullet/S}(X) \dfn
  \Set{p_Z:Z\to X | \text{\(|Z|\) is of cardinality \(1\)}}/ \sim ,\]
  where
  \[(p_Z:Z\to X)\sim (p_{Z'}:Z'\to X) \ \deff \
  Z\btimes_{p_Z,X,p_{Z'}}Z' \neq \emptyset\]
  (this fiber product exists in \(\Schb{S}\), cf. \autoref{lem: fiber product bbullet} \ref{enumi: fiber product bbullet not red} \ref{enumi: fiber product bbullet red}).
  Then the set \(\mathsf{Pt}_{\bbullet/S}(X)\) is defined
  by means of properties that may be
  characterized category-theoretically
  (cf. \autoref{lem: 1pt}) from the data \((\Schb{S},X)\).
  Since \(|Z|\) is of cardinality \(1\),
  \(p_Z\) defines a point \(p_Z(*_Z)\in X\).
  Furthermore, it follows immediately from \autoref{lem: fiber product bbullet}
  that
  \begin{equation}
    \label{equiv: reconstruction Set equiv}
    p_Z \sim p_{Z'} \iff p_Z(*_Z)=p_{Z'} (*_{Z'}).
    \tag{\(\dagger\)}
  \end{equation}
  In particular, the map
  \begin{align*}
    \eta_X:\mathsf{Pt}_{\bbullet/S}(X) &\to |X| \\
    [p_Z] &\mapsto p_Z(*_Z)
  \end{align*}
  is well-defined and injective.
  On the other hand, for a point \(x\in X\),
  the equality \(\eta_X([\Spec k(x)\to X]) = x \) holds.
  Hence \(\eta_X\) is a bijection.

  Let \((f:X\to Y)\) be a morphism in \(\Schb{S}\).
  We define
  \begin{align*}
    \mathsf{Pt}_{\bbullet/S}(f) :
    \mathsf{Pt}_{\bbullet/S}(X) &\to \mathsf{Pt}_{\bbullet/S}(Y) \\
    [p_Z] &\mapsto [f\circ p_Z].
  \end{align*}
  It follows immediately from \eqref{equiv: reconstruction Set equiv} that
  the map \(\mathsf{Pt}_{\bbullet/S}(f)\) is well-defined.
  Hence we obtain a functor
  \(\mathsf{Pt}_{\bbullet/S}: \Schb{S} \to \SetV\).

  Let \(\bigstar \dfn \SetU\) or \(\SetV\).
  We define
  \begin{equation}
    \label{equation: underlying set functor}
    \begin{aligned}
      U_{\bbullet/S}^{\bigstar}: \Schb{S} &\to \bigstar \\
      X &\mapsto |X|.
    \end{aligned}
    \tag{\(\ddagger\)}
  \end{equation}
  Write \(i^{\sfSet}_{\univ{U}\in\univ{V}}: \SetU \to \SetV\)
  for the natural inclusion.
  Then the equality
  \(U_{\bbullet/S}^{\SetV} = i^{\sfSet}_{\univ{U}\in\univ{V}}\circ U_{\bbullet/S}^{\SetU}\) holds.
  It follows from the equivalence \eqref{equiv: reconstruction Set equiv}, together with
  the definitions of \(\eta_X\) and \(\mathsf{Pt}_{\bbullet/S}(f)\), that
  the following diagram commutes:
  \begin{equation}
    \label{equiv: reconstruction Set commute}
    \begin{CD}
      \mathsf{Pt}_{\bbullet/S}(X) @>\mathsf{Pt}_{\bbullet/S}(f)>>
      \mathsf{Pt}_{\bbullet/S}(Y) \\
      @V\eta_X VV @VV \eta_Y V \\
      |X| @>>> |Y|.
    \end{CD}
    \tag{\(\dagger\dagger\)}
  \end{equation}
  Thus we obtain an isomorphism of functors
  \(\eta:\mathsf{Pt}_{\bbullet/S} \xrightarrow{\sim} U_{\bbullet/S}^{\SetV} =
  i^{\sfSet}_{\univ{U}\in\univ{V}}\circ U_{\bbullet/S}^{\SetU}\).
\end{reconstruction}

Since the functor \(\mathsf{Pt}_{\bbullet/S}\) is defined
category-theoretically from the data \(\Schb{S}\),
the following lemma holds:

\begin{lem}\label{lem: Pt equiv commute}
  Let \(S,T\) be schemes and \(F:\Schb{S} \xrightarrow{\sim} \Schc{T}\) an equivalence.
  Then \(S,T,F\) determine an isomorphism \(\rho^{\mathsf{Pt}}: \mathsf{Pt}_{\bbullet/S} \xrightarrow{\sim} \mathsf{Pt}_{\bbbullet/S}\circ F\)
  between the two composite functors of the following diagram:
  \[
  \begin{CD}
    \Schb{S} @>F>> \Schc{T} \\
    @V\mathsf{Pt}_{\bbullet/S}VV @VV\mathsf{Pt}_{\bbbullet/T}V \\
    \SetV @= \SetV.
  \end{CD}
  \]
\end{lem}

\begin{proof}
  Let \(X\) be an object of \(\Schb{S}\).
  We define a map
  \begin{align*}
    \rho_X : \mathsf{Pt}_{\bbullet/S}(X) &\to \mathsf{Pt}_{\bbbullet/T}(F(X)) \\
    [p_Z:Z\to X] &\mapsto [F(p_Z): F(Z)\to F(X)].
  \end{align*}
  By \autoref{lem: 1pt},
  \(|F(Z)|\) is of cardinality \(1\).
  Since \(F\) is an equivalence,
  \[Z\btimes_XZ \neq \emptyset \iff F(Z) \times^{\bbbullet}_{F(X)}F(Z') \neq \emptyset.\]
  Hence \(\rho_X\) is well-defined and injective.
  If \(Z'\) is an object of \(\Schc{T}\) whose underlying set is of cardinality \(1\),
  \(Z'\to F(X)\) is a morphism of \(\Schc{T}\), and
  \(F^{-1}\) is a quasi-inverse of \(F\),
  then it follows immediately that \(\rho_X([F^{-1}(Z')\to X]) = [Z'\to F(X)]\).
  Hence \(\rho_X\) is a bijection.

  To complete the proof of \autoref{lem: Pt equiv commute},
  it suffices to prove that
  for any morphism \((f:X\to Y)\in \Schb{S}\),
  the following diagram commutes:
  \[
  \begin{CD}
    \mathsf{Pt}_{\bbullet/S}(X) @> \rho_X >> \mathsf{Pt}_{\bbbullet/T}(F(X)) \\
    @V\mathsf{Pt}_{\bbullet/S}(f)VV @VV\mathsf{Pt}_{\bbbullet/T}(F(f))V \\
    \mathsf{Pt}_{\bbullet/S}(Y) @>> \rho_Y > \mathsf{Pt}_{\bbbullet/T}(F(Y))
  \end{CD}
  \]
  Let \(p_Z:Z\to X\) be a morphism in \(\Schb{S}\) such that
  \(Z\) is of cardinality \(1\).
  Then
  \begin{align*}
    \rho_Y(\mathsf{Pt}_{\bbullet/S}(f)([p_Z])) &= \rho_Y([f\circ p_Z]))
    = [F(f\circ p_Z)] = [F(f) \circ F(p_Z)] \\
    &= \mathsf{Pt}_{\bbbullet/T}(F(f))([F(p_Z)])
    = \mathsf{Pt}_{\bbbullet/T}(F(f))(\rho_X([p_Z])).
  \end{align*}
  This completes the proof of \autoref{lem: Pt equiv commute}.
\end{proof}

\begin{cor}\label{cor: underlying set equiv commute}
  Let \(S,T\) be schemes and \(F:\Schb{S} \xrightarrow{\sim} \Schc{T}\) an equivalence.
  Then the following diagram commutes up to isomorphism:
  \[
  \begin{CD}
    \Schb{S} @>F>> \Schc{T} \\
    @VU_{\bbullet/S}^{\SetU}VV @VVU_{\bbbullet/T}^{\SetU}V \\
    \SetU @= \SetU.
  \end{CD}
  \]
\end{cor}

\begin{proof}
  It follows from \autoref{lem: Pt equiv commute} and \autoref{reconstruction: Set}
  that we have natural isomorphisms
  \begin{align*}
    i^{\sfSet}_{\univ{U}\in\univ{V}}\circ U_{\bbullet/S}^{\SetU}
    \xleftarrow{\sim} \mathsf{Pt}_{\bbullet/S}
    \xrightarrow{\sim} \mathsf{Pt}_{\bbbullet/T} \circ F
    \xrightarrow{\sim} i^{\sfSet}_{\univ{U}\in\univ{V}}\circ U_{\bbbullet/T}^{\SetU} \circ F.
  \end{align*}
  Since \(i^{\sfSet}_{\univ{U}\in\univ{V}}\) is fully faithful,
  the composite isomorphism of the above display determines a natural isomorphism
  \(U_{\bbullet/S}^{\SetU} \xrightarrow{\sim} U_{\bbbullet/T}^{\SetU} \circ F\).
\end{proof}

In the following discussion,
we shall use the isomorphism of functors \(\eta\) constructed in \autoref{reconstruction: Set}
to \textit{identify} (until further notice) \(\mathsf{Pt}_{\bbullet/S}(-)\) with \(|-|\).

\begin{cor}\label{cor: surj is cat}
  Let \(S\) be a scheme.
  Let \(f:Y\to X\) be a morphism of \(\Schb{S}\), \(x\in |X|\).
  Then:
  \begin{enumerate}
    \item \label{enumi: cor: surj is cat surj}
    The property that
    \(f\) is surjective
    may be characterized category-theoretically from the data \((\Schb{S},f:Y\to X)\).
    \item \label{enumi: cor: surj is cat x in f}
    The property that
    \(x\in \mathrm{Im}(f)\)
    may be characterized category-theoretically from the data \((\Schb{S},f:Y\to X,x)\).
  \end{enumerate}
\end{cor}

\begin{proof}
  \autoref{cor: surj is cat} follows immediately from \autoref{cor: underlying set equiv commute}.
\end{proof}

In the remainder of this section,
we consider a category-theoretic characterization of the morphisms of
\(\Schb{S}\) whose underlying morphism of schemes
is isomorphic to the natural morphism to a scheme
from the spectrum of the residue field at a point of the scheme.

\begin{lem}\label{lem: Spec field}
  Let \(S\) be a scheme.
  Let \(X\) be an object of \(\Schb{S}\).
  Then \(X\) is isomorphic to the spectrum of a field if and only if
  the following conditions hold:
  \begin{enumerate}
    \item \label{enumi: Spec field 1pt}
    \(|X|\) is of cardinality \(1\);
    \item \label{enumi: Spec field epi}
    every morphism \(f:Y\to X\), where \(Y\neq \emptyset\), is an epimorphism.
  \end{enumerate}
  In particular, the property that
  \(X\) is isomorphic to the spectrum of a field
  may be characterized category-theoretically
  (cf. \autoref{lem: 1pt}) from the data \((\Schb{S},X)\).
\end{lem}

\begin{proof}
  First, we prove necessity.
  Assume that \(X\) is isomorphic to the spectrum of a field.
  It follows immediately that condition \ref{enumi: Spec field 1pt} holds.
  Next, we prove that \(X\) satisfies condition \ref{enumi: Spec field epi}.
  Let \(f:Y\to X\) be a morphism, where \(Y\neq \emptyset\), and
  \(V=\Spec B\) an affine open subscheme of \(Y\).
  To prove that \(f\) is an epimorphism,
  it suffices to prove that \(f|_V\) is an epimorphism.
  Hence we may assume that \(Y = \Spec B\) is affine.
  Since \(k=\Gamma (X,\mathcal{O}_X)\) is a field,
  the ring morphism \(f^{\#} : k\to B\) is injective.
  Let \(Z\) be an object of \(\Schb{S}\) and
  \(g,h:X\rightrightarrows Z\) two morphisms such that \(g \circ f = h \circ f\).
  Then it suffices to prove that \(g=h\).
  Since \(Y\neq \emptyset\), it follows immediately that \(f\) is surjective on
  the underlying sets, hence that
  \(g\) and \(h\) induce the same morphism on the underlying sets.
  Since \(|X|\) is of cardinality \(1\),
  \(g\) and \(h\) factor through an affine open subscheme \(\Spec C\) of \(Z\).
  Hence we may assume that \(Z = \Spec C\) is affine.
  Observe that it follows from the equality
  \(g \circ f = h \circ f\) that \(f^{\#} \circ g^{\#} = f^{\#} \circ h^{\#}\).
  Thus, since \(f^{\#}\) is injective, we conclude that \(g^{\#} = h^{\#}\), i.e., that \(g=h\) holds.

  Next, we prove sufficiency.
  Assume that \(X\) satisfies conditions \ref{enumi: Spec field 1pt} and \ref{enumi: Spec field epi}.
  It follows immediately from condition \ref{enumi: Spec field 1pt}
  that \(X\) is an affine scheme.
  Assume that \(A=\Gamma (X,\mathcal{O}_X)\) is not a field.
  Then the natural surjection \(f^{\#}:A\to k\dfn A/\sqrt{0}_A\),
  where \(\sqrt{0}_A\) denotes the nilradical of \(A\), is
  not injective.
  Write \(q^{\#},r^{\#}:A' \dfn A\times _k A\rightrightarrows A\) for the two projections.
  Since \(f^{\#}\) is not injective, \(q^{\#}\neq r^{\#}\).
  Thus, since the equality \(f^{\#} \circ q^{\#} = f^{\#} \circ r^{\#}\) holds,
  we conclude that \(f\) is not an epimorphism,
  in contradiction to our assumption that \(X\) satisfies condition \ref{enumi: Spec field epi}.
  This contradiction implies that \(A\) is a field, as desired.
\end{proof}

\begin{lem}\label{lem: quotient field}
  Let \(S\) be a scheme.
  Let \(f:Y\to X\) be a morphism of \(\Schb{S}\).
  Then \(f\) is isomorphic as an object of \(\Sch{X}\)
  to the object of \(\Sch{X}\)
  that arises from the natural morphism to \(X\)
  from the spectrum of the residue field at a point of \(X\)
  if and only if
  the following conditions hold:
  \begin{enumerate}
    \item \label{enumi: quotient field field}
    \(Y\) is isomorphic to the spectrum of a field;
    \item \label{enumi: quotient field factor}
    if \(g: Z\to X\) is a morphism of \(\Schb{S}\) such that
    \(Z\) is isomorphic to the spectrum of a field and \(Y\btimes_X Z \neq \emptyset\),
    then there exists a unique morphism \(h:Z\to Y\) such that \(g = f \circ h\).
  \end{enumerate}
  In particular, the property that
  \(f:Y\to X\) is isomorphic as an object of \(\Sch{X}\)
  to the object of \(\Sch{X}\)
  that arises from the natural morphism to \(X\)
  from the spectrum of the residue field at a point of \(X\)
  may be characterized category-theoretically
  (cf. \autoref{lem: Spec field}) from the data \((\Schb{S},f:Y\to X)\).
\end{lem}

\begin{proof}
  First, we prove necessity.
  Assume that there exists \(x\in X\) such that
  \(f\) is isomorphic as an object of \(\Sch{X}\) to the object of \(\Sch{X}\)
  that arises from the natural morphism \(\Spec k(x) \to X\).
  Then \(Y\cong \Spec k(x)\) is isomorphic to the spectrum of a field.
  Hence condition \ref{enumi: quotient field field} holds.
  Next, we prove that \(f\) satisfies condition \ref{enumi: quotient field factor}.
  Let \(g: Z\to X\) be a morphism of \(\Schb{S}\) such that
  \(Z\) is isomorphic to the spectrum of a field and \(Y\btimes_X Z \neq \emptyset\).
  Write \(p_Y : Y\btimes_X Z \to Y\) and \(p_Z : Y\btimes_X Z \to Z\)
  for the natural projections.
  Since \(Y\btimes_X Z\neq \emptyset\), there exists a point \(w\in Y\btimes_X Z\).
  Then \(x = f(*_Y) = f(p_Y(w)) = g(p_Z(w)) = g(*_Z)\).
  Hence \(g\) factors uniquely through \(Y\).
  Thus \(f\) satisfies condition \ref{enumi: quotient field factor}.

  Next, we prove sufficiency.
  Assume that \(f:Y\to X\) satisfies conditions
  \ref{enumi: quotient field field} and \ref{enumi: quotient field factor}.
  By condition \ref{enumi: quotient field field},
  there exists a field \(K\) such that \(Y \cong \Spec K\).
  We set \(x \dfn f(*_Y)\).
  Since the natural morphism \(g:Z = \Spec k(x) \to X\) satisfies condition \ref{enumi: quotient field factor},
  there exists \(h:Y\to Z\) such that \(g \circ h = f\).
  Since \(f:Y \to X\) satisfies condition \ref{enumi: quotient field factor},
  there exists \(h':Z\to Y\) such that \(g = f \circ h'\).
  Hence \(f = f \circ h' \circ h\) and \(g = g \circ h \circ h'\).
  Thus, it follows from the uniqueness portion of condition \ref{enumi: quotient field factor}
  that \(h' \circ h = \id _Y\) and \(h \circ h' = \id _Z\).
  Hence \(f\) is isomorphic as an object of \(\Sch{X}\) to the object of \(\Sch{X}\)
  that arises from the natural morphism \(g:\Spec k(x) \to X\).
\end{proof}


\section{Regular Monomorphisms}
\label{section: regular mono}

In this section,
we study various scheme-theoretic properties of regular monomorphisms in \(\Schb{S}\)
(cf. \autoref{lem: reg mono}).

\begin{defi}
  Let \(f:X\to Y\) be a morphism in a category \(\mathcal{C}\).
  We shall say that \(f:X\to Y\) is a \textit{regular monomorphism}
  if there exist two morphisms \(g,h : Y\rightrightarrows Z\) in \(\mathcal{C}\)
  such that \(f\) is the equalizer of \(g\) and \(h\).
\end{defi}

If, for a morphism of schemes \(f:X\to Y\),
the push-out \(Y\coprod_XY\) exists,
then it is often the case that \(f\) is the equalizer of
the two natural morphisms \(Y \rightrightarrows Y\coprod_XY\).
Hence to study various properties of regular monomorphisms in \(\Schb{S}\),
it is important to study various properties of coproducts in \(\Schb{S}\).

First, we note the following properties concerning coproducts in \(\Schb{S}\):

\begin{lem}\label{lem: coprod lemma}
  Let \(S\) be a scheme.
  Let \(X,Y\) be objects of \(\Schb{S}\).
  Then \(X\coprod Y\) belongs to \(\Schb{S}\).
  In particular, the coproduct \(X\bcoprod Y\) exists in \(\Schb{S}\)
  and is naturally isomorphic to \(X\coprod Y\).
  In the following, we shall simply write \(\coprod\) for \(\bcoprod\).
\end{lem}

\begin{proof}
  \autoref{lem: coprod lemma} follows from the definition of the disjoint union of schemes.
\end{proof}

By applying the above lemma,
we obtain the following category-theoretic characterization of the objects of
\(\Schb{S}\) whose underlying scheme is connected.

\begin{cor}\label{lem: connected is cat}
  Let \(S\) be a scheme.
  Let \(X\) be an object of \(\Schb{S}\).
  Then \(X\) is connected
  if and only if, for any objects \(Y,Z\) of \(\Schb{S}\),
  \begin{center}
    \(X\cong Y\coprod Z \ \Rightarrow \ Y \cong \emptyset \ \lor \ Z \cong \emptyset.\)
  \end{center}
  In particular,
  the property that \(X\) is connected
  may be characterized category-theoretically from the data \((\Schb{S},X)\).
\end{cor}

\begin{proof}
  \autoref{lem: connected is cat} follows from the definition of connectedness.
\end{proof}

Next,
we study various properties of push-outs along open immersions.

\begin{lem}\label{lem: coprod property}
  Let \(S\) be a scheme.
  Let \(X\) be an \(S\)-scheme, \(U \subset X\) an open subscheme of \(X\).
  Then the following hold:
  \begin{enumerate}
    \item \label{enumi: coprod exists}
    The push-out \(X\coprod_U X\) exists in \(\Sch{S}\).
    \item \label{enumi: coprod red}
    If \(X\) is reduced, then \(X\coprod_U X\) is reduced.
    \item \label{enumi: coprod qsep}
    If \(X\) is quasi-separated over \(S\),
    and the inclusion \(U\subset X\) is a quasi-compact open immersion,
    then \(X\coprod_UX\) is quasi-separated over \(S\).
    \item \label{enumi: coprod qcpt}
    If \(X\) is quasi-compact over \(S\),
    then \(X\coprod_UX\) is quasi-compact over \(S\).
  \end{enumerate}
\end{lem}

\begin{proof}
  Since the morphism \(U\hookrightarrow X\) is an open immersion,
  one may construct the push-out \(X\coprod_U X\) in \(\Sch{S}\)
  by glueing two copies of the scheme \(X\) along the open subscheme \(U\subset X\).
  Assertions \ref{enumi: coprod exists} and \ref{enumi: coprod red}
  follow immediately.

  Next, we prove assertion \ref{enumi: coprod qsep}.
  Since the property ``quasi-separated over \(S\)'' is Zariski local on \(S\),
  we may assume that \(S\) is an affine scheme.
  By \cite[\href{https://stacks.math.columbia.edu/tag/01KO}{Tag 01KO}]{stacks-project},
  it suffices to prove that
  if \(U_1,U_2 \subset X\) are affine open subsets,
  then \(U_1\cap U_2\) and \(U_1\cap U_2\cap U\) are quasi-compact.
  Since \(X\) is quasi-separated, \(U_1\cap U_2\) is quasi-compact.
  Furthermore, since the inclusion \(U\subset X\) is quasi-compact,
  the intersection \(U_1\cap U_2\cap U\) is quasi-compact.
  Hence \(X\coprod_UX\) is quasi-separated.

  Finally, we prove assertion \ref{enumi: coprod qcpt}.
  Since the property ``quasi-compact over \(S\)'' is Zariski local on \(S\),
  we may assume that \(S\) is an affine scheme.
  Then \(X\) is quasi-compact.
  Write \(i_1,i_2:X\rightrightarrows X\coprod _UX\)
  for the two natural open immersions.
  Then \(X\coprod _UX\) is covered by
  the two quasi-compact open subschemes \(i_1(X)\) and \(i_2(X)\).
  Hence \(X\coprod_UX\) is quasi-compact.
\end{proof}

Next, we note the following property
concerning affine schemes of \(\Schb{S}\):

\begin{lem}\label{lem: affine over qsep is qcpt}
  Let \(S\) be a quasi-separated scheme.
  Let \(X\) be a scheme over \(S\).
  Then the following assertions hold:
  \begin{enumerate}
    \item \label{enumi: lem: aff over qsep is qcpt qc}
    If \(X\) is a quasi-compact scheme,
    then the structure morphism \(f:X\to S\) is quasi-compact.
    In particular,
    if \(X\) is an affine scheme,
    then the structure morphism \(f:X\to S\) is quasi-compact.
    \item \label{enumi: lem: aff over qsep is qcpt sep}
    If \(X\) is a separated (respectively, quasi-separated) scheme,
    then the structure morphism \(f:X\to S\)
    is separated (respectively, quasi-separated).
    In particular,
    if \(X\) is an affine scheme,
    then the structure morphism \(f:X\to S\) is separated,
    hence also quasi-separated.
  \end{enumerate}
\end{lem}

\begin{proof}
  Assertion \ref{enumi: lem: aff over qsep is qcpt qc} follows immediately from
  \cite[\href{https://stacks.math.columbia.edu/tag/03GI}{Tag 03GI}]{stacks-project}.
  Assertion \ref{enumi: lem: aff over qsep is qcpt sep} follows immediately from
  \cite[\href{https://stacks.math.columbia.edu/tag/01KV}{Tag 01KV}]{stacks-project}.
\end{proof}

By applying \autoref{lem: coprod property} and \autoref{lem: affine over qsep is qcpt},
we obtain the following property concerning the existence of push-outs in \(\Schb{S}\):

\begin{cor}\label{cor: coprod exists}
  Let \(S\) be a quasi-separated scheme.
  Let \(X\) be an object of \(\Schb{S}\) and
  \(U\) an affine open subscheme of \(X\).
  Suppose that
  \(\bbbullet = (\bbullet \cup \left\{ \qcpt, \sep \right\})\setminus \{\qsep\}\).
  Then the following assertions hold:
  \begin{enumerate}
    \item \label{enumi: cor: coprod exists aff}
    \(U\) belongs to \(\Schc{S}\),
    hence, in particular, to \(\Schb{S}\).
    \item \label{enumi: cor: coprod exists coprod}
    Assume that \(\sep\not\in\bbullet\).
    Then \(X\coprod_UX\)
    (cf. \autoref{lem: coprod property} \ref{enumi: coprod exists})
    belongs to \(\Schb{S}\).
    In particular, the push-out \(X\bcoprod_UX\) exists in \(\Schb{S}\) and
    is naturally isomorphic to \(X\coprod_UX\).
  \end{enumerate}
\end{cor}

\begin{proof}
  Assertion \ref{enumi: cor: coprod exists aff} follows immediately from
  \autoref{lem: affine over qsep is qcpt} \ref{enumi: lem: aff over qsep is qcpt qc}
  and \ref{enumi: lem: aff over qsep is qcpt sep}.
  To prove assertion \ref{enumi: cor: coprod exists coprod},
  it suffices to prove that \(X\coprod_UX\) belongs to the category \(\Schb{S}\).
  If \(\qsep\not\in \bbullet\),
  then assertion \ref{enumi: cor: coprod exists coprod} follows from
  \autoref{lem: coprod property} \ref{enumi: coprod red} and \ref{enumi: coprod qcpt}.
  Assume that \(\qsep\in \bbullet\).
  Since \(X\) is a quasi-separated scheme,
  and \(U\) is an affine scheme,
  the inclusion \(U\hookrightarrow X\) is a quasi-compact open immersion
  by \autoref{lem: affine over qsep is qcpt} \ref{enumi: lem: aff over qsep is qcpt qc}.
  Hence assertion \ref{enumi: cor: coprod exists coprod} follows from
  \autoref{lem: coprod property} \ref{enumi: coprod red},
  \ref{enumi: coprod qsep}, and \ref{enumi: coprod qcpt}.
\end{proof}

Next,
we study the operation ``\((-)_{\red}\)'' in the case of immersions.

\begin{lem}\label{lem: imm red}
  Let \(S\) be a scheme.
  Let \(f:X\to Y\) be a morphism of \(S\)-schemes.
  \begin{enumerate}
    \item \label{lem: imm red imm}
    If \(f\) is an immersion,
    then \(f_{\red}:X_{\red}\to Y_{\red}\) is also an immersion.
    \item \label{lem: imm red qcpt imm}
    If \(f\) is a quasi-compact immersion,
    then \(f_{\red}:X_{\red}\to Y_{\red}\) is also a quasi-compact immersion.
    \item \label{lem: imm red cl imm}
    If \(f\) is a closed immersion,
    then \(f_{\red}:X_{\red}\to Y_{\red}\) is also a closed immersion.
  \end{enumerate}
\end{lem}

\begin{proof}
  It suffices to prove that
  if \(f\) is quasi-compact, an open immersion or a closed immersion,
  then \(f_{\red}\) is also quasi-compact,
  an open immersion or a closed immersion, respectively.

  Since \(f_{\red}\) and \(f\) induce same morphism
  of the underlying topological spaces,
  if \(f\) is quasi-compact,
  then \(f_{\red}\) is also quasi-compact.

  Assume that \(f\) is an open immersion.
  Then \(X_{\red} = (|X|,\mathcal{O}_{Y_{\red}}|_{|X|})\).
  Hence \(f_{\red}\) is also an open immersion.

  Finally, if a ring homomorphism \(A\to B\) is surjective,
  then \(A/\sqrt{0} \to B/\sqrt{0}\) is also surjective.
  In particular,
  if \(f\) is a closed immersion,
  then \(f_{\red}\) is also a closed immersion.
\end{proof}

Next,
we study various scheme-theoretic properties of
regular monomorphisms in \(\Schb{S}\).

\begin{lem}\label{lem: reg mono}
  Let \(S\) be a quasi-separated scheme.
  Let \(f:X\to Y\) be a morphism of \(\Schb{S}\).
  \begin{enumerate}
    \item \label{enumi: closed imm is reg mono}
    If \(f\) is a closed immersion,
    then \(f\) is a regular monomorphism in \(\Schb{S}\).
    \item \label{enumi: reg mono is imm}
    If \(f\) is a regular monomorphism in \(\Schb{S}\),
    then \(f\) is an immersion.
    \item \label{enumi: reg mono in qsep is qc imm}
    Assume that \(\qsep\in \bbullet\).
    If \(f\) is a regular monomorphism in \(\Schb{S}\),
    then \(f\) is a quasi-compact immersion.
    \item \label{enumi: reg mono in sep is closed imm}
    Assume that \(\sep\in \bbullet\).
    If \(f\) is a regular monomorphism in \(\Schb{S}\),
    then \(f\) is a closed immersion.
  \end{enumerate}
\end{lem}

\begin{proof}
  First, we prove assertion \ref{enumi: closed imm is reg mono}.
  Assume that \(f\) is a closed immersion.
  Recall that by
  \cite[\href{https://stacks.math.columbia.edu/tag/0E25}{Tag 0E25}]{stacks-project},
  the scheme \(Y\coprod_XY\) belongs to \(\Schb{S}\),
  and \(f\) is an equalizer of \(Y\rightrightarrows Y\coprod _XY\) in \(\Schb{S}\).
  Hence \(f\) is a regular monomorphism in \(\Schb{S}\).
  This completes the proof of assertion \ref{enumi: closed imm is reg mono}.

  Next, we prove assertions \ref{enumi: reg mono is imm},
  \ref{enumi: reg mono in qsep is qc imm}, and
  \ref{enumi: reg mono in sep is closed imm}.
  Assume that there exist \(g,h :Y\rightrightarrows Z\) such that
  \(f\) is an equalizer of \((g,h)\) in the category \(\Schb{S}\).
  Suppose that
  \(\bbbullet = \bbullet \setminus \left\{\qcpt\right\}\).
  Then \(f\) is an equalizer of \((g,h)\)
  in the category \(\Schc{S}\).
  Indeed,
  suppose that \(f':X'\to Y\) is a morphism in the category \(\Schc{S}\)
  such that \(g\circ f'=h\circ f'\).
  Since \(\bbullet\subset \bbbullet\cup \left\{\qcpt,\sep\right\}\),
  by \autoref{cor: coprod exists} \ref{enumi: cor: coprod exists aff},
  any affine open subscheme of \(X'\) belongs to \(\Schb{S}\).
  Since \(f\) is an equalizer of \((g,h)\) in the category \(\Schb{S}\),
  for any open immersion \(i':U'\to X'\) such that \(U'\) is affine,
  there exists a unique morphism \(p:U'\to X\) such that
  \(f\circ p = f'\circ i'\).
  Thus, by allowing \(U'\) to vary over the affine open subschemes of \(X'\)
  and gluing together the resulting morphisms \(p:U'\to X\),
  we obtain a unique morphism \(q:X'\to X\) such that
  \(f\circ q = f'\).
  Hence \(f\) is an equalizer of \((g,h)\) in the category \(\Schc{S}\).

  By \autoref{lem: fiber product bbullet}
  \ref{enumi: fiber product bbullet not qcpt and red} and
  \ref{enumi: fiber product bbullet not qcpt and in red},
  the fiber product \(Z\times_S^\bbbullet Z\) exists
  in the category \(\Schc{S}\).
  Hence we may consider the diagonal morphism
  \(\Delta _Z :Z\to Z\times_S^\bbbullet Z\).
  By \autoref{lem: imm red} \ref{lem: imm red imm}
  (cf. also \autoref{lem: fiber product bbullet}
  \ref{enumi: fiber product bbullet not qcpt and red} and
  \ref{enumi: fiber product bbullet not qcpt and in red}),
  \(\Delta _Z\) is an immersion.
  Since \(f\) is the equalizer of \((g,h)\) in the category \(\Schc{S}\),
  \(f\) is the pull-back in the category \(\Schc{S}\)
  of \(\Delta_Z\) along \((g,h) :Y\to Z\times_S^\bbbullet Z\).
  Hence \(f\) is an immersion.
  If, moreover, \(\qsep\in \bbullet\) (respectively, \(\sep \in \bbullet\)),
  then by \autoref{lem: imm red} \ref{lem: imm red qcpt imm} and
  \ref{lem: imm red cl imm}
  (cf. also \autoref{lem: fiber product bbullet}
  \ref{enumi: fiber product bbullet not qcpt and red} and
  \ref{enumi: fiber product bbullet not qcpt and in red}),
  \(\Delta _Z\), hence also \(f\),
  is a quasi-compact (respectively, closed) immersion.
  This completes the proof of assertions \ref{enumi: reg mono is imm},
  \ref{enumi: reg mono in qsep is qc imm}, and
  \ref{enumi: reg mono in sep is closed imm}.
\end{proof}

Finally,
we apply the theory developed thus far
to conclude a category-theoretic characterization of
reduced schemes.

\begin{cor}\label{cor: red is cat}
  Let \(S\) be a quasi-separated scheme.
  Let \(X\) be an object of \(\Schb{S}\).
  Then \(X\) is reduced if and only if
  every surjective regular monomorphism \(Y\to X\)
  in \(\Schb{S}\) is an isomorphism.
  In particular, the property that
  \(X\) is reduced
  may be characterized category-theoretically
  (cf. \autoref{cor: surj is cat} \ref{enumi: cor: surj is cat surj})
  from the data \((\Schb{S},X)\).
\end{cor}

\begin{proof}
  First, we prove necessity.
  Assume that \(X\) is reduced, and
  \(f:Y\to X\) is a surjective regular monomorphism in \(\Schb{S}\).
  By \autoref{lem: reg mono} \ref{enumi: reg mono is imm},
  \(f\) is a surjective closed immersion.
  Thus, since \(X\) is reduced,
  \(f\) is an isomorphism.
  This completes the proof of necessity.

  Next, we prove sufficiency.
  Assume that \(X\) is not reduced.
  Then the natural morphism \(f:X_{\red} \to X\)
  is a surjective closed immersion, but not an isomorphism in \(\Schb{S}\).
  Thus, by \autoref{lem: reg mono} \ref{enumi: closed imm is reg mono},
  the closed immersion \(f\) is a surjective regular monomorphism,
  but not an isomorphism.
  This completes the proof of \autoref{cor: red is cat}.
\end{proof}


\begin{cor}\label{cor: reduction is cat}
  Let \(S\) be a quasi-separated scheme.
  Let \(f:Y\to X\) be a morphism in \(\Schb{S}\).
  Then \(f\) is isomorphic to the natural morphism \(X_{\red} \to X\)
  in \(\Schb{X}\)
  if and only if
  for any reduced object \(Z\) in \(\Schb{S}\) and any morphism \(g:Z\to X\) in \(\Schb{S}\),
  there exists a unique morphism \(h:Z\to Y\) in \(\Schb{S}\)
  such that \(f\circ h = g\).
  In particular, the property that
  \(f:Y\to X\) is isomorphic to the natural morphism \(X_{\red}\to X\) in \(\Schb{X}\)
  may be characterized category-theoretically
  (cf. \autoref{cor: red is cat}) from the data \((\Schb{S},f:Y\to X)\).
\end{cor}

\begin{proof}
  \autoref{cor: reduction is cat} follows immediately
  from \autoref{cor: red is cat}.
\end{proof}

\section{Local Domains}
\label{section: local domains}

In this section,
we give a category-theoretic characterization of the objects of
\(\Schb{S}\) whose underlying scheme
is isomorphic to the spectrum of a local domain
and a category-theoretic characterization of
specialization and generization relations
between points of a scheme.

\begin{defi}
  Let \(X\) be a scheme and \(x,y\) points of \(X\).
  \begin{itemize}
    \item[(1)]
    We shall say that \(x\) is a \textit{generic point} of \(X\)
    if, for any point \(x'\in X\) such that
    \(x'\rsa x\), it holds that \(x'=x\).
    \item[(2)]
    We shall say that \(x,y\) are \textit{comparable}
    if either \(x\rsa y\) or \(y\rsa x\).
    \item[(3)]
    We shall say that \(X\) is \textit{local}
    if there exists a point \(v\in X\) such that for any \(x'\in X\),
    \(x'\rsa v\).
  \end{itemize}
\end{defi}

\begin{rem}\label{rem: local is aff}
  \
  \begin{enumerate}
    \item
    It is immediate that if \(X\) is local,
    then \(X\) has a unique closed point.
    \item \label{enumi: rem: local is aff}
    If \(X\) is local,
    then every affine open neighborhood of the unique closed point of \(X\) is equal to \(X\).
    Thus \(X\) is local if and only if
    \(X\) is isomorphic to the spectrum of a local ring.
    \item
    It is immediate that a point \(x\in X\) is closed if and only if
    every point \(y\in X\) such that \(x \rsa y\) is equal to \(x\).
  \end{enumerate}
\end{rem}

\begin{defi}\label{defi: local morphism}
  Let \(f:X\to Y\) be a morphism of local schemes.
  Then we shall say that \(f\) is \textit{local} if
  \(f\) maps the unique closed point of \(X\)
  to the unique closed point of \(Y\).
\end{defi}

We will use the following property to give
a category-theoretic characterization of the property ``comparable''.

\begin{defi}\label{defi: str local defi}
  Let \(S\) be a scheme.
  Let \(X\) be an object of \(\Schb{S}\) and
  \(x_1,x_2\in X\).
  Write \(f_1:\Spec(k(x_1)) \to X, f_2:\Spec(k(x_2))\to X\) for
  the natural morphisms determined by the points \(x_1, x_2 \in X\).
  We shall say that the triple \((X,x_1,x_2)\) is
  \textit{strongly local} in \(\Schb{S}\)
  if the following conditions hold:
  \begin{enumerate}
    \item \label{enumi: str loc is connected}
    \(X\) is connected.
    \item \label{enumi: str loc reg mono isom}
    For any regular monomorphism \(f:Z\to X\),
    if \(x_1,x_2\in \mathrm{Im}(f)\), then \(f\) is an isomorphism.
    \item \label{enumi: str loc epi}
    The morphism
    \(f_1\coprod f_2: \Spec(k(x_1)) \coprod \Spec(k(x_2)) \to X\)
    determined by \(f_1\) and \(f_2\)
    is an epimorphism.
    \item \label{enumi: str loc closed pt}
    The morphism \(f_1\) is a regular monomorphism.
    \item \label{enumi: str loc not reg mono}
    For any regular monomorphism \(f:Z\to X\),
    if \(x_1\not\in \mathrm{Im}(f)\) and \(Z\neq \emptyset\),
    then the morphism \(f\coprod f_1: Z\coprod \Spec(k(x_1)) \to X\)
    determined by \(f\) and \(f_1\)
    is \textbf{not} a regular monomorphism.
  \end{enumerate}
  Thus the property that \((X,x_1,x_2)\) is strongly local
  is defined completely in terms of properties that may be
  characterized category-theoretically
  (cf. \autoref{reconstruction: Set},
  \autoref{cor: surj is cat} \ref{enumi: cor: surj is cat x in f},
  \autoref{lem: quotient field},
  \autoref{lem: connected is cat})
  from the data \((\Schb{S},(X,x_1,x_2))\)
  (where we note that by using the result of \autoref{reconstruction: Set}, the datum \((X,x_1,x_2)\) may be regarded as a category-theoretic datum).
\end{defi}

Next, we study basic properties of strongly local triples.

\begin{lem}\label{lem: str local is local domain}
  Let \(S\) be a quasi-separated scheme.
  Let \(X\) be an object of \(\Schb{S}\)
  and \(x_1,x_2\in X\).
  Suppose that \((X,x_1,x_2)\) is strongly local in \(\Schb{S}\).
  Then the following assertions hold:
  \begin{enumerate}
    \item \label{enumi: lem: str local is local domain}
    \(X\) is isomorphic to the spectrum of a local domain.
    \item \label{enumi: lem: str local cl or gen}
    One of \(x_1,x_2\) is the closed point, and
    the other is the generic point.
    In particular,
    \(x_1,x_2\) are comparable.
  \end{enumerate}
\end{lem}

\begin{proof}
  Write \(f_1:\Spec(k(x_1)) \to X, f_2:\Spec(k(x_2))\to X\) for
  the natural morphisms determined by the points \(x_1, x_2 \in X\).

  Suppose that \(x_1 = x_2\).
  Recall from \autoref{defi: str local defi} \ref{enumi: str loc closed pt} that
  the morphism \(f_1\) is a regular monomorphism.
  Since \(x_1=x_2\), it holds that \(x_2\in \im(f_1)\).
  Hence by \autoref{defi: str local defi} \ref{enumi: str loc reg mono isom},
  the morphism \(f_1\) is an isomorphism.
  Thus \(X\) is isomorphic to the spectrum of a field.
  In particular,
  assertions \ref{enumi: lem: str local is local domain}
  and \ref{enumi: lem: str local cl or gen} hold,
  so in the remainder of the proof of \autoref{lem: str local is local domain},
  we may suppose that \(x_1\neq x_2\).

  By \autoref{lem: reg mono} \ref{enumi: closed imm is reg mono},
  the natural closed immersion
  \(\overline{\left\{ x_1,x_2\right\}}_{\red} \to X\)
  is a regular monomorphism.
  Hence by \autoref{defi: str local defi} \ref{enumi: str loc reg mono isom},
  the regular monomorphism
  \(\overline{\left\{ x_1,x_2\right\}}_{\red} \to X\) is an isomorphism.
  Therefore, the following properties hold:
  \begin{enumerate}[label=(\alph*)]
    \item \label{enumi: pf str loc is loc dom X red}
    \(X\) is reduced.
    \item \label{enumi: pf str loc is loc dom X has 2 irred comp}
    \(|X| = \overline{\left\{ x_1\right\}} \cup \overline{\left\{ x_2\right\}}\).
  \end{enumerate}
  In particular, by property \ref{enumi: pf str loc is loc dom X has 2 irred comp},
  the number of irreducible components of \(X\) is at most two.


  Suppose that \(\sep\in \bbullet\).
  First, we verify that \(X\) is local, and
  \(x_1\) is the unique closed point of \(X\).
  By \autoref{defi: str local defi} \ref{enumi: str loc closed pt} and
  \autoref{lem: reg mono} \ref{enumi: reg mono in sep is closed imm},
  \(x_1\) is the unique closed point of \(X\).
  Assume that there exists a point \(x\in X\)
  such that \(x_1 \not\in \overline{\left\{ x\right\}}\).
  Then the morphisms \(f: \overline{\left\{ x\right\}}_{\red}\to X\) and
  \(f\coprod f_1 : \overline{\left\{ x\right\}}_{\red} \coprod \Spec (k(x_1))
  \to X\)
  are closed immersions.
  Hence by \autoref{lem: reg mono} \ref{enumi: closed imm is reg mono},
  \(f\) and \(f\coprod f_1\) are regular monomorphisms.
  This contradicts the fact that \((X,x_1,x_2)\) satisfies
  condition \ref{enumi: str loc not reg mono} of
  \autoref{defi: str local defi}.
  Thus \(x_1 \in \overline{\left\{ x\right\}}\) holds for any \(x\in X\),
  i.e., \(X\) is local, and
  \(x_1\) is the unique closed point of \(X\).

  Next, we complete the proof of \autoref{lem: str local is local domain}
  in the case where \(\sep\in \bbullet\).
  Since \(x_1\) is the unique closed point of \(X\),
  it holds that \(x_2 \rsa x_1\).
  Hence by property \ref{enumi: pf str loc is loc dom X has 2 irred comp},
  \(X\) has a unique generic point \(x_2\).
  In particular,
  assertion \ref{enumi: lem: str local cl or gen} holds.
  By property \ref{enumi: pf str loc is loc dom X red},
  this implies that \(X\) is integral.
  In particular,
  assertion \ref{enumi: lem: str local is local domain} holds.
  This completes the proof of
  \autoref{lem: str local is local domain} in the case where \(\sep\in \bbullet\).

  Suppose that \(\sep\not\in \bbullet\).
  First, we verify that \(X\) is irreducible.
  Assume that \(X\) is not irreducible.
  By property \ref{enumi: pf str loc is loc dom X has 2 irred comp},
  \(X\) has precisely two irreducible components,
  and \(x_1,x_2\) are the generic points of
  these two irreducible components.
  Let \(U_1\) be an affine open neighborhood of \(x_1\)
  such that \(U_1 \subset X \setminus \overline{\left\{ x_2 \right\} }\),
  and \(U_2\) an affine open neighborhood of \(x_2\)
  such that \(U_2 \subset X \setminus \overline{\left\{ x_1 \right\} }\).
  By property \ref{enumi: pf str loc is loc dom X has 2 irred comp},
  the equality
  \[U_1\cap U_2 \subset
  (X \setminus \overline{\left\{ x_2 \right\} }) \cap
  (X \setminus \overline{\left\{ x_1 \right\} })
  = X\setminus (\overline{\left\{ x_1 \right\} } \cup
  \overline{\left\{ x_2 \right\}})
  = \emptyset\]
  holds.
  Since \(U_1,U_2\) are affine,
  \(U \dfn U_1 \cup U_2\) is an affine open subset of \(X\).
  Hence by \autoref{cor: coprod exists} \ref{enumi: cor: coprod exists coprod},
  the push-out \(X\coprod_UX\) belongs to \(\Schb{S}\).
  Since \(U_1,U_2\neq \emptyset\) and \(U_1\cap U_2 = \emptyset\),
  \(U\) is not connected.
  By \autoref{defi: str local defi} \ref{enumi: str loc is connected},
  it holds that \(U\neq X\).
  Hence the two natural open immersions \(i_1,i_2:X\rightrightarrows X\coprod_UX\)
  are distinct.
  On the other hand,
  since \(x_1,x_2\) are contained in \(U\),
  the equality
  \(i_1\circ (f_1\coprod f_2) = i_2\circ (f_1\coprod f_2)\) holds,
  i.e., the morphism
  \(f_1\coprod f_2:\Spec (k(x_1))\coprod \Spec (k(x_2)) \to X\)
  is not an epimorphism.
  This contradicts the fact that \((X,x_1,x_2)\) satisfies
  condition \ref{enumi: str loc epi} of \autoref{defi: str local defi}.
  Thus \(X\) is irreducible.
  Write \(\eta\) for the generic point of \(X\).
  By property \ref{enumi: pf str loc is loc dom X has 2 irred comp},
  \(\eta\in \left\{ x_1,x_2\right\}\).

  Next, we complete the proof of \autoref{lem: str local is local domain}
  in the case where \(\sep\not\in \bbullet\).
  Let \(v\) be the unique point of \(X\) such that
  \(\left\{ v,\eta \right\} = \left\{ x_1,x_2\right\}\).
  By properties \ref{enumi: pf str loc is loc dom X red}
  and \ref{enumi: pf str loc is loc dom X has 2 irred comp},
  to complete the proof of \autoref{lem: str local is local domain},
  it suffices to prove that
  the relation \(x\rsa v\) holds for any \(x\in X\).
  Assume that there exists a point \(x\in X\)
  such that \(v\not\in \overline{\left\{ x\right\}}\).
  Then there exists an affine open neighborhood \(V\) of \(v\)
  such that \(x\not\in V\).
  By \autoref{cor: coprod exists} \ref{enumi: cor: coprod exists coprod},
  the push-out \(X\coprod_VX\) belongs to \(\Schb{S}\).
  Since \(V\neq X\),
  the two natural inclusions \(i_1',i_2':X\rightrightarrows X\coprod_VX\)
  are distinct.
  On the other hand,
  since \(\left\{ x_1,x_2\right\} = \left\{ v,\eta\right\} \subset V\),
  the equality \(i_1'\circ (f_1\coprod f_2) = i_2'\circ (f_1\coprod f_2)\) holds,
  i.e., the morphism
  \(f_1\coprod f_2:\Spec (k(x_1))\coprod \Spec (k(x_2)) \to X\)
  is not an epimorphism.
  This contradicts the fact that \((X,x_1,x_2)\) satisfies
  condition \ref{enumi: str loc epi} of \autoref{defi: str local defi}.
  Thus for any \(x\in X\), the relation \(x\rsa v\) holds.
  This completes the proof of \autoref{lem: str local is local domain}
  in the case where \(\sep\not\in \bbullet\).
\end{proof}


\begin{cor}\label{cor: cl pt or gen pt in str local is cat}
  Let \(S\) be a quasi-separated scheme.
  Let \(X\) be an object of \(\Schb{S}\)
  and \(x_1,x_2,x\in X\).
  Then the property that \((X,x_1,x_2)\) is strongly local in \(\Schb{S}\), and
  \(x\) is either the unique closed point or the unique generic point of \(X\)
  may be characterized category-theoretically
  from the data \((\Schb{S},X,x_1,x_2,x)\).
\end{cor}

\begin{proof}
  \autoref{cor: cl pt or gen pt in str local is cat} follows
  formally from
  \autoref{reconstruction: Set},
  \autoref{defi: str local defi}, and
  \autoref{lem: str local is local domain}.
\end{proof}

Next, we prove that the triple consisting of the spectrum of a valuation ring,
the unique closed point of the spectrum,
and the unique generic point of the spectrum
is strongly local.

\begin{lem}\label{lem: val is str local}
  Let \(S\) be a quasi-separated scheme.
  Let \(A\) be a valuation ring equipped with
  a structure morphism \(\Spec(A) \to S\),
  \(v\in \Spec(A)\) the unique closed point, and
  \(\eta\in \Spec(A)\) the unique generic point.
  Then \((\Spec(A),v,\eta)\) is strongly local in \(\Schb{S}\).
\end{lem}

\begin{proof}
  Since the spectrum of a local ring is connected,
  \((\Spec(A),v,\eta)\) satisfies condition
  \ref{enumi: str loc is connected} of
  \autoref{defi: str local defi}.

  Next, we verify that \((\Spec(A),v,\eta)\) satisfies condition
  \ref{enumi: str loc reg mono isom} of
  \autoref{defi: str local defi}.
  Let \(f:Z\to X\) be a regular monomorphism such that \(v,\eta \in f(Z)\).
  Then, by \autoref{lem: reg mono} \ref{enumi: reg mono is imm},
  \(f\) is an immersion, i.e.,
  \(f\) is the composite of a closed immersion with an open immersion.
  Moreover, since \(v,\eta\in f(Z)\),
  the immersion \(f\) is a homeomorphism,
  i.e., \(f\) is a surjective closed immersion.
  Furthermore,
  since \(A\) is reduced,
  the surjective closed immersion \(f\) is an isomorphism.
  This completes the proof that
  \((\Spec(A),v,\eta)\) satisfies condition
  \ref{enumi: str loc reg mono isom} of
  \autoref{defi: str local defi}.

  Next, we verify that \((\Spec(A),v,\eta )\) satisfies condition
  \ref{enumi: str loc epi} of
  \autoref{defi: str local defi}.
  Let \(f,g:X\rightrightarrows Y\) be two morphisms
  such that \(f|_v = g|_v , f|_{\eta} = g|_{\eta}\).
  It suffices to prove that \(f = g\).
  Let \(K\) be the field of fractions of \(A\) and
  \(i^{\#}:A\to K\) the natural injection.
  It follows from the definition of \(v\) that
  \(f\) and \(g\) factor through an affine open neighborhood
  \(V = \Spec(B)\) of \(f(v) = g(v)\).
  Hence, we may assume that \(Y = \Spec(B)\) is affine.
  Since \(f|_{\eta} = g|_{\eta}\),
  the equality \(i^{\#}\circ f^{\#} = i^{\#}\circ g^{\#}\) holds.
  Since \(i^{\#}\) is injective,
  we conclude that \(f^{\#} = g^{\#}\),
  hence that \(f=g\).
  This completes the proof that
  \((\Spec(A),v,\eta)\) satisfies condition
  \ref{enumi: str loc epi} of
  \autoref{defi: str local defi}.

  Next, by \autoref{lem: reg mono} \ref{enumi: closed imm is reg mono},
  the closed immersion \(\Spec(k(v)) \to \Spec(A)\) is a regular monomorphism,
  i.e., \((\Spec(A),v,\eta)\) satisfies condition
  \ref{enumi: str loc closed pt} of
  \autoref{defi: str local defi}.

  Finally, we verify that \((\Spec(A),v,\eta)\) satisfies condition
  \ref{enumi: str loc not reg mono} of
  \autoref{defi: str local defi}.
  Let \(f:Z\to X\) be a regular monomorphism such that
  \(v\not\in f(Z)\) and \(Z\neq \emptyset\).
  By \autoref{lem: reg mono} \ref{enumi: reg mono is imm},
  it suffices to prove that
  \(g:Z\coprod \Spec(k(v)) \to X\) is not an immersion.
  Since \(Z\neq \emptyset\) and \(X\) is local,
  any open neighborhood of \(v\) intersects \(f(Z)\).
  On the other hand,
  the open neighborhood \(\left\{*_{\Spec(k(v))}\right\}\)
  of the point \(*_{\Spec(k(v))}\) in \(Z\coprod \Spec(k(v))\)
  does not intersect \(Z\).
  Thus \(g\) is not an immersion.
  This completes the proof of \autoref{lem: val is str local}.
\end{proof}

By applying the theory of strongly local triples developed thus far,
we conclude a category-theoretic characterization
of comparable pairs of points of a scheme.

\begin{prop}\label{prop: comparable is cat}
  Let \(S\) be a quasi-separated scheme.
  Let \(X\) be an object of \(\Schb{S}\) and
  \(x_1,x_2\in X\).
  Then \(x_1,x_2\) are comparable if and only if
  there exist an object \(Z\),
  points \(z_1,z_2\in Z\),
  and a morphism \(f:Z\to X\)
  such that \((Z,z_1,z_2)\) is strongly local  in \(\Schb{S}\)
  and \(\left\{ f(z_1),f(z_2)\right\} = \left\{ x_1,x_2\right\}\).
  In particular,
  the property that \(x_1,x_2\in X\) are comparable
  may be characterized category-theoretically
  (cf. \autoref{reconstruction: Set},
  \autoref{defi: str local defi})
  from the data \((\Schb{S},X,x_1,x_2)\).
\end{prop}

\begin{proof}
  First, we prove necessity.
  By \autoref{lem: affine over qsep is qcpt}
  \ref{enumi: lem: aff over qsep is qcpt qc}
  \ref{enumi: lem: aff over qsep is qcpt sep},
  we may assume without loss of generality that
  \(X\) is the spectrum of a local domain \(A\)
  whose unique closed point is \(x_1\), and
  whose unique generic point is \(x_2\).
  Write \(K\) for the fraction field of \(A\).
  Let \(C\subset K\) be a valuation ring
  that dominates \(A\)
  (cf. \cite[\href{https://stacks.math.columbia.edu/tag/00IA}{Tag 00IA}]{stacks-project}),
  \(Z\dfn \Spec(C)\),
  \(z_1\in Z\) the unique closed point,
  \(z_2\in Z\) the unique generic point, and
  \(f:Z\to X\) the morphism induced by
  \(A \subset C\).
  By \autoref{lem: affine over qsep is qcpt}
  \ref{enumi: lem: aff over qsep is qcpt qc}
  \ref{enumi: lem: aff over qsep is qcpt sep},
  \(f\) belongs to \(\Schb{S}\).
  Moreover,
  by \autoref{lem: val is str local},
  \((Z,z_1,z_2)\) is strongly local.
  This completes the proof of necessity.

  Next, we prove sufficiency.
  Let \((Z,z_1,z_2)\) be strongly local in \(\Schb{S}\)
  and \(f:Z\to X\) a morphism such that
  \(\left\{ f(z_1),f(z_2)\right\} = \left\{ x_1,x_2 \right\}\).
  We may assume without loss of generality that
  \(f(z_1) = x_1, f(z_2) = x_2\).
  Suppose that \(z_2\rsa z_1\).
  Then
  \[x_1 = f(z_1) \in f(\overline{\left\{ z_2\right\}})
  \subset \overline{\left\{ f(z_2)\right\}} = \overline{\left\{ x_2\right\}}.\]
  Thus \(x_2 \rsa x_1\) holds.
  Suppose that \(z_1\rsa z_2\).
  Then
  \[x_2 = f(z_2) \in f(\overline{\left\{ z_1\right\}})
  \subset \overline{\left\{ f(z_1)\right\}} = \overline{\left\{ x_1\right\}}.\]
  Thus \(x_1 \rsa x_2\) holds.
  This completes the proof of \autoref{prop: comparable is cat}.
\end{proof}


To give a category-theoretic characterization of local objects of \(\Schb{S}\),
we define the following property:

\begin{defi}\label{defi: valuative}
  Let \(S\) be a quasi-separated scheme.
  Let \(X\) be an object of \(\Schb{S}\)
  and \(x_1,x_2\in X\).
  We shall say that
  the triple \((X,x_1,x_2)\) is \textit{valuative} in \(\Schb{S}\)
  if the following conditions hold:
  \begin{enumerate}
    \item \label{enumi: defi: valuative comparable}
    Any two points of \(X\) are comparable.
    \item \label{enumi: defi: valuative str loc}
    \((X,x_1,x_2)\) is strongly local in \(\Schb{S}\).
  \end{enumerate}
  Thus the property that \((X,x_1,x_2)\) is valuative in \(\Schb{S}\)
  is defined completely in terms of properties that may be
  characterized category-theoretically
  (cf. \autoref{defi: str local defi}, \autoref{prop: comparable is cat})
  from the data \((\Schb{S},X,x_1,x_2)\).
\end{defi}

\begin{rem}\label{rem: valuative remark}
  \
  \begin{enumerate}
    \item \label{enumi: rem: val rem cat char}
    By \autoref{cor: cl pt or gen pt in str local is cat},
    if the triple \((X,x_1,x_2)\) is valuative,
    then for \(x\in X\),
    the property that \(x\) is either the closed point or the generic point
    may be characterized category-theoretically
    from the data \((\Schb{S},X,x_1,x_2,x)\).
    \item \label{enumi: rem: val rem val is val}
    It follows immediately from \autoref{lem: val is str local} that
    the triple \((\Spec(A), v, \eta)\) consisting of
    the spectrum of a valuation ring \(A\),
    the unique closed point \(v\in \Spec(A)\), and
    the unique generic point \(\eta \in \Spec(A)\)
    is a valuative triple in \(\Schb{S}\).
  \end{enumerate}
\end{rem}

The next lemma will play an important role in
our category-theoretic characterization of the objects of \(\Schb{S}\)
whose underlying scheme is irreducible local
(cf. \autoref{prop: irred local is cat}).

\begin{lem}\label{lem: 3pt val}
  Let \(X\) be a scheme and
  \(x_1,x_2,x_3\in X\).
  Suppose that \(x_3 \rsa x_2 \rsa x_1\).
  Then there exist a triple \((\Spec(V), v, \eta)\)
  and a morphism \(f:\Spec(V) \to X\) such that
  \(\Spec(V)\) is the spectrum of a valuation ring \(V\),
  \(v\in \Spec(V)\) is the unique closed point,
  \(\eta \in \Spec(V)\) is the unique generic point,
  \(x_2\in f(\Spec(V)), f(\eta) = x_3\), and
  \(f(v) = x_1\).
\end{lem}

\begin{proof}
  To prove the assertion,
  we may assume without loss of generality that
  \(X = \Spec(A)\) is the spectrum of a local domain \(A\),
  \(x_1\in X\) is the unique closed point, and
  \(x_3\in X\) is the unique generic point.

  Let \(\mathfrak{p}\) be the prime ideal of \(A\)
  which corresponds to \(x_2 \in \Spec(A)\),
  \(V_1\) a valuation ring
  which dominates \(A_\mathfrak{p}\) in \(k(x_3)\),
  \(\mathfrak{n}_1\) the maximal ideal of \(V_1\), and
  \(p:V_1\to V_1 /\mathfrak{n}_1\) the natural surjection.
  Since \(V_1\) dominates \(A_\mathfrak{p}\) in \(k(x_3)\),
  the natural morphism \(j:A/\mathfrak{p} \to V_1/\mathfrak{n}_1\) is injective.
  Let \(V_2\) be a valuation ring
  which dominates \(A/\mathfrak{p}\) in the field \(V_1/\mathfrak{n}_1\).
  Then \(V \dfn p^{-1}(V_2)\) is the desired valuation ring.
  This completes the proof of \autoref{lem: 3pt val}.
\end{proof}

Next, we give a category-theoretic characterization of the objects of \(\Schb{S}\)
whose underlying scheme is irreducible local.

\begin{prop}\label{prop: irred local is cat}
  Let \(S\) be a quasi-separated scheme.
  Let \(X\) be an object of \(\Schb{S}\).
  Assume that \(|X|\) is not of cardinality \(1\).
  Then \(X\) is irreducible local if and only if
  there exist points \(x_1\neq x_2 \in X\) such that
  the triple \((X,x_1,x_2)\) satisfies
  the following conditions:
  \begin{enumerate}
    \item \label{enumi: irred local 2}
    For any point \(x\in X\), \(x,x_1\) and \(x,x_2\) are comparable.
    \item \label{enumi: irred local 3}
    For any valuative triple \((Y,y_1,y_2)\) and any morphism \(f:Y\to X\),
    if \(x_1\in f(Y)\), then \(x_1 \in \left\{ f(y_1),f(y_2)\right\}\).
    \item \label{enumi: irred local 4}
    For any valuative triple \((Y,y_1,y_2)\) and any morphism \(f:Y\to X\),
    if \(x_2 \in f(Y)\), then \(x_2 \in \left\{ f(y_1),f(y_2)\right\}\).
  \end{enumerate}
  In particular,
  the property that \(X\) is irreducible local
  may be characterized category-theoretically
  (cf. \autoref{lem: 1pt},
  \autoref{reconstruction: Set},
  \autoref{cor: surj is cat} \ref{enumi: cor: surj is cat x in f},
  \autoref{prop: comparable is cat})
  from the data \((\Schb{S},X)\).
  Furthermore, the property that
  \(X\) is irreducible local, and
  \(x\in X\) is either the closed point or the generic point holds
  if and only if
  there exist points \(x_1,x_2 \in X\) such that
  \(x\in \left\{ x_1,x_2\right\}\) and
  the triple \((X,x_1,x_2)\) satisfies the above conditions
  \ref{enumi: irred local 2},
  \ref{enumi: irred local 3},
  \ref{enumi: irred local 4}.
  In particular,
  the property that
  \(X\) is irreducible local, and
  \(x\in X\) is either the closed point or the generic point
  may be characterized category-theoretically
  (cf. \autoref{lem: 1pt},
  \autoref{reconstruction: Set},
  \autoref{cor: surj is cat} \ref{enumi: cor: surj is cat x in f},
  \autoref{prop: comparable is cat})
  from the data \((\Schb{S},X,x)\).
\end{prop}

\begin{proof}
  First, we prove necessity.
  Suppose that \(X\) is irreducible local, and
  \(|X|\) is not of cardinality \(1\).
  Let \(x_1\in X\) be the unique closed point and
  \(x_2\in X\) the unique generic point.
  Then it follows immediately that
  the triple \((X,x_1,x_2)\) satisfies condition \ref{enumi: irred local 2}.

  Next, we prove that \(X\) satisfies condition \ref{enumi: irred local 3}.
  Let \((Y,y_1,y_2)\) be a valuative triple and
  \(f:Y\to X\) a morphism such that \(x_1\in f(Y)\).
  Then there exists a point \(y\in Y\) such that \(f(y) = x_1\).
  Since either \(y_1\in Y\) or \(y_2 \in Y\)
  is the unique closed point of \(Y\)
  (cf. \autoref{lem: str local is local domain}
  \ref{enumi: lem: str local cl or gen}),
  either \(y\rsa y_1\) or \(y \rsa y_2\) holds.
  Hence either \(x_1 \rsa f(y_1)\) or \(x_1 \rsa f(y_2)\) holds.
  Since \(x_1\) is the unique closed point of \(X\),
  either \(f(y_1) = x_1\) or \(f(y_2) = x_1\) holds.
  Thus \(X\) satisfies condition \ref{enumi: irred local 3}.

  Next, we prove that \(X\) satisfies condition \ref{enumi: irred local 4}.
  Let \((Y,y_1,y_2)\) be a valuative triple and
  \(f:Y\to X\) a morphism such that \(x_2\in f(Y)\).
  Then there exists a point \(y\in Y\) such that \(f(y) = x_2\).
  Since either \(y_1\in Y\) or \(y_2 \in Y\)
  is the unique generic point of \(Y\)
  (cf. \autoref{lem: str local is local domain}
  \ref{enumi: lem: str local cl or gen}),
  either \(y_1\rsa y\) or \(y_2 \rsa y\) holds.
  Hence either \(f(y_1) \rsa x_2\) or \(f(y_2) \rsa x_2\) holds.
  Since \(x_2\) is the unique generic point of \(X\),
  either \(f(y_1) = x_2\) or \(f(y_2) = x_2\) holds.
  Thus \(X\) satisfies condition \ref{enumi: irred local 4}.
  This completes the proof of necessity.

  Next, we prove sufficiency.
  Suppose that
  \(|X|\) is not of cardinality \(1\), and
  there exist points \(x_1\neq x_2 \in X\) such that
  the triple \((X,x_1,x_2)\) satisfies conditions
  \ref{enumi: irred local 2},
  \ref{enumi: irred local 3},
  \ref{enumi: irred local 4}.
  Then since the triple \((X,x_2,x_1)\) satisfies conditions
  \ref{enumi: irred local 2},
  \ref{enumi: irred local 3},
  \ref{enumi: irred local 4},
  we may assume without loss of generality that \(x_2 \rsa x_1\).

  Since \(X\) satisfies condition \ref{enumi: irred local 2},
  if \(x_1\) is a closed point of \(X\),
  then any point of \(X\) is a generalization of \(x_1\).
  Hence to prove that \(X\) is local,
  it suffices to prove that \(x_1\) is a closed point of \(X\).
  Suppose that there exists a point \(x \in X\)
  such that \(x_2 \rsa x_1 \rsa x\).
  By \autoref{lem: 3pt val},
  there exists a triple \((\Spec(W),w,\zeta)\) consisting of
  the spectrum of a valuation ring \(W\),
  the unique closed point \(w\in \Spec(W)\), and
  the unique generic point \(\zeta \in \Spec(W)\),
  together with a morphism \(f:\Spec(W) \to X\),
  such that \(f(\zeta) = x_2, f(w) = x\), and \(x_1 \in f(\Spec V)\).
  By \autoref{lem: affine over qsep is qcpt}
  \ref{enumi: lem: aff over qsep is qcpt qc}
  \ref{enumi: lem: aff over qsep is qcpt sep},
  and \autoref{rem: valuative remark} \ref{enumi: rem: val rem val is val},
  the morphism \(f:\Spec(W)\to X\) belongs to \(\Schb{S}\), and
  the triple \((\Spec(W),w,\zeta)\) is valuative in \(\Schb{S}\).
  Since \(X\) satisfies condition \ref{enumi: irred local 3},
  it holds that
  \(x_1 \in \left\{ f(w),f(\zeta)\right\} = \left\{ x,x_2\right\}\).
  Since \(x_1\neq x_2\), this proves that \(x_1 = x\).
  Thus \(X\) is local, and
  \(x_1\) is the unique closed point of \(X\).

  Since \(X\) satisfies condition \ref{enumi: irred local 2},
  if \(x_2\) is a generic point of \(X\),
  then any point of \(X\) is a specialization of \(x_2\).
  Hence to prove that \(X\) is irreducible,
  it suffices to prove that \(x_2\) is a generic point of \(X\).
  Suppose that there exists a point \(\xi \in X\)
  such that \(\xi \rsa x_2 \rsa x_1\).
  By \autoref{lem: 3pt val},
  there exists a triple \((\Spec(W'),w',\zeta')\) consisting of
  the spectrum of a valuation ring \(W'\),
  the closed point \(w'\in \Spec(W')\), and
  the generic point \(\zeta' \in \Spec(W')\),
  together with a morphism \(f':\Spec(W') \to X\),
  such that \(f(\zeta') = \xi, f(w') = x_1\), and \(x_2 \in f(\Spec(W'))\).
  By \autoref{lem: affine over qsep is qcpt}
  \ref{enumi: lem: aff over qsep is qcpt qc}
  \ref{enumi: lem: aff over qsep is qcpt sep},
  and \autoref{rem: valuative remark} \ref{enumi: rem: val rem val is val},
  the morphism \(f':\Spec(W')\to X\) belongs to \(\Schb{S}\), and
  the triple \((\Spec(W'),w',\zeta')\) is valuative in \(\Schb{S}\).
  Since \(X\) satisfies condition \ref{enumi: irred local 4},
  it holds that
  \(x_2 \in \left\{ f(\zeta'),f(w')\right\} = \left\{ \xi,x_1 \right\}\).
  Since \(x_1\neq x_2\), this proves that \(x_2 = \xi\).
  Thus \(X\) is irreducible, and
  \(x_2\) is the unique generic point of \(X\).
  This completes the proof of sufficiency.

  The remaining assertions of \autoref{prop: irred local is cat} are immediate.
  This completes the proof of \autoref{prop: irred local is cat}.
\end{proof}


To give a category-theoretic characterization of
specialization and generization relations between points of a scheme,
we first give a category-theoretic characterization of
the generic point of an object
that appears as the first member of a valuative triple.

\begin{prop}\label{prop: val generic pt}
  Let \(S\) be a quasi-separated scheme.
  Let \((X,x_1,x_2)\) be a valuative triple in \(\Schb{S}\)
  and \(\eta\in X\).
  Then \(\eta\in X\) is the generic point if and only if
  the following conditions hold:
  \begin{enumerate}
    \item \label{enumi: val generic pt x or eta}
    \(\eta\in \left\{x_1,x_2\right\}\).
    \item \label{enumi: gen pt epi}
    For any irreducible local object \(Y\)
    and any morphisms \(f,g:X\rightrightarrows Y\),
    if \(f|_\eta = g|_\eta\) then \(f=g\).
  \end{enumerate}
  In particular,
  the property that \((X,x_1,x_2)\) is valuative in \(\Schb{S}\), and
  \(\eta\in X\) is the generic point,
  may be characterized category-theoretically
  (cf. \autoref{lem: quotient field}, \autoref{prop: irred local is cat})
  from the data \((\Schb{S},X,x_1,x_2,\eta)\).
\end{prop}

\begin{proof}
  We may assume without loss of generality that
  \(X\) is not isomorphic to the spectrum of a field.
  By \autoref{lem: str local is local domain}
  \ref{enumi: lem: str local is local domain} and
  \autoref{defi: valuative} \ref{enumi: defi: valuative str loc},
  \(X\) is affine, and
  \(A\dfn \Gamma(X,\mathcal{O}_X)\) is a local domain.
  Write \(K\) for the field of fractions of \(A\) and
  \(i^{\#}:A\to K\) for the natural inclusion.

  First, we prove necessity.
  Suppose that \(\eta\in X\) is the generic point.
  It follows from
  \autoref{lem: str local is local domain}
  \ref{enumi: lem: str local cl or gen} that
  \(\eta\in \left\{x_1,x_2\right\}\).
  Hence condition \ref{enumi: val generic pt x or eta} holds.
  Let \(Y=\Spec(B)\) be the spectrum of a local ring
  (cf. \autoref{rem: local is aff} \ref{enumi: rem: local is aff}) and
  \(f,g:X\rightrightarrows Y\) morphisms such that \(f|_\eta = g|_\eta\).
  Since \(\eta\) is the generic point,
  the equality \(i^{\#}\circ f^{\#}=i^{\#}\circ g^{\#}\) holds.
  Since \(i^{\#}\) is injective, the equality \(f^{\#}=g^{\#}\) holds.
  Hence the equality \(f=g\) holds.
  This completes the proof of necessity.

  Next, we prove sufficiency.
  Suppose that a point \(\eta\in X\) satisfies conditions
  \ref{enumi: val generic pt x or eta} and \ref{enumi: gen pt epi}.
  By \autoref{lem: str local is local domain} \ref{enumi: lem: str local cl or gen}
  and condition \ref{enumi: val generic pt x or eta},
  it suffices to prove that
  if \(\eta\) is the closed point,
  then \(\eta\) does \textbf{not} satisfy condition \ref{enumi: gen pt epi}.
  Suppose that \(\eta\in X\) is the closed point.
  Let \(\mathfrak{m}_\eta\) be
  the maximal ideal of \(A\) corresponding to \(\eta\).
  Since \(A\) is not a field, and
  \(\eta\) is a closed point,
  there exists a non-zero element \(a\in \mathfrak{m}_\eta\).
  Let \(B \dfn A[t]_{(\mathfrak{m}_\eta,t)}\) be the local domain
  obtained by localizing the polynomial ring \(A[t]\)
  at the maximal ideal \((\mathfrak{m}_\eta,t)\);
  \(f,g:\Spec(A) \to \Spec(B)\)
  the morphisms determined by the morphisms of \(A\)-algebras
  \(f^{\#},g^{\#}:B\rightrightarrows A\)
  such that \(f^{\#}(t)=0, g^{\#}(t)=a\).
  Then \(f|_\eta = g|_\eta\).
  On the other hand, since \(0\neq a\),
  it holds that \(f\neq g\).
  Thus \(\eta\) does not satisfy condition \ref{enumi: gen pt epi}.
  This completes the proof of \autoref{prop: val generic pt}.
\end{proof}

By a similar argument to the argument
given in the proof of \autoref{prop: val generic pt},
one may give
a category-theoretic characterization of
the closed point of an object
that appears as the first member of a valuative triple.

\begin{cor}\label{cor: y rsa x is cat}
  Let \(S\) be a quasi-separated scheme.
  Let \(X\) be an object of \(\Schb{S}\)
  and \(x,x'\in X\).
  Then \(x'\rsa x\) if and only if
  there exist a valuative triple \((V,v_1,v_2)\),
  \(\eta\in V\), and
  a morphism \(f:V\to X\) such that
  \(\eta\) is the unique generic point
  (cf. \autoref{lem: str local is local domain}
  \ref{enumi: lem: str local is local domain},
  \autoref{defi: valuative} \ref{enumi: defi: valuative str loc})
  of \(V\),
  \(f(\eta) = x'\), and \(x\in \im(f)\).
  In particular, the property that \(x'\rsa x\)
  may be characterized category-theoretically
  (cf. \autoref{defi: valuative}, \autoref{prop: val generic pt})
  from the data \((\Schb{S},X,x,x')\).
\end{cor}

\begin{proof}
  Sufficiency follows immediately.
  Thus it remains to prove necessity.
  Suppose that \(x'\rsa x\).
  By \autoref{lem: affine over qsep is qcpt}
  \ref{enumi: lem: aff over qsep is qcpt qc}
  \ref{enumi: lem: aff over qsep is qcpt sep},
  we may assume without loss of generality that
  \(X\) is the spectrum of a local domain,
  \(x\in X\) is the unique closed point, and
  \(x'\in X\) is the unique generic point.
  By \autoref{rem: valuative remark} \ref{enumi: rem: val rem val is val},
  \autoref{lem: 3pt val},
  there exist a valuative triple \((V,v_1,v_2)\),
  a point \(\eta \in V\),
  and a morphism \(f:V \to X\) such that
  \(V\) is the spectrum of a valuation ring,
  \(v_1\) is the unique closed point of \(V\),
  \(v_2 = \eta\) is the unique generic point of \(V\),
  \(x'=f(\eta)\), and \(x = f(v_1) \in V\).
  By \autoref{lem: affine over qsep is qcpt}
  \ref{enumi: lem: aff over qsep is qcpt qc}
  \ref{enumi: lem: aff over qsep is qcpt sep},
  \(f\) belongs to \(\Schb{S}\).
  This completes the proof of \autoref{cor: y rsa x is cat}.
\end{proof}

%

\begin{cor}\label{cor: irred, loc, are cat}
  Let \(S\) be a quasi-separated scheme.
  Let \(X\) be an object of \(\Schb{S}\)
  and \(x\in X\).
  Then the following assertions hold:
  \begin{enumerate}
    \item \label{enumi: cor: cl pt is cat}
    \(x\) is a closed point in \(X\) if and only if
    for any point \(x'\in X\) such that \(x\rsa x'\),
    it holds that \(x=x'\).
    In particular, the property that \(x\) is a closed point in \(X\)
    may be characterized category-theoretically
    (cf. \autoref{cor: y rsa x is cat}) from the data \((\Schb{S},X,x)\).
    \item \label{enumi: cor: gen pt is cat}
    \(x\) is a generic point in \(X\) if and only if
    for any point \(x'\in X\) such that \(x' \rsa x\),
    it holds that \(x=x'\).
    In particular, the property that \(x\) is a generic point in \(X\)
    may be characterized category-theoretically
    (cf. \autoref{cor: y rsa x is cat}) from the data \((\Schb{S},X,x)\).
    \item \label{enumi: cor: loc is cat}
    \(X\) is local if and only if
    there exists a point \(x_1\in X\) such that
    \(x'\rsa x_1\) holds for all \(x'\in X\).
    In particular, the property that \(X\) is local
    may be characterized category-theoretically
    (cf. \autoref{cor: y rsa x is cat}) from the data \((\Schb{S},X)\).
    \item \label{enumi: cor: irred is cat}
    \(X\) is irreducible if and only if
    there exists a point \(x_1\in X\) such that
    \(x_1\rsa x'\) holds for all \(x'\in X\).
    In particular, the property that \(X\) is irreducible
    may be characterized category-theoretically
    (cf. \autoref{cor: y rsa x is cat}) from the data \((\Schb{S},X)\).
  \end{enumerate}
\end{cor}

\begin{proof}
  These assertions follow from
  the definitions of the various properties under consideration.
\end{proof}

\begin{cor}\label{cor: int, int loc is cat}
  Let \(S\) be a quasi-separated scheme.
  Let \(X\) be an object of \(\Schb{S}\).
  Then the following assertions hold:
  \begin{enumerate}
    \item \label{enumi: cor: int is cat}
    \(X\) is integral if and only if \(X\) is irreducible and reduced.
    \item \label{enumi: cor: int loc is cat}
    \(X\) is isomorphic to the spectrum of a local domain if and only if
    \(X\) is irreducible, reduced and local.
  \end{enumerate}
  In particular, the properties of \ref{enumi: cor: int is cat}
  and \ref{enumi: cor: int loc is cat}
  may be characterized category-theoretically
  from the data \((\Schb{S},X)\)
  (cf. \autoref{cor: red is cat},
  \autoref{cor: irred, loc, are cat} \ref{enumi: cor: loc is cat}
  \ref{enumi: cor: irred is cat}).
\end{cor}

\begin{proof}
  These assertions follow from the definition of the notion of an integral scheme.
\end{proof}

Next, we apply the theory developed thus far to conclude
a category-theoretic characterization of the morphisms of \(\Schb{S}\)
whose underlying morphism of schemes is isomorphic to the natural morphism
to a scheme from the spectrum of the local ring at a point of the scheme.

\begin{defi}\label{defi: pair irred loc}
  Let \(X\) be a scheme; \(x,\eta\in X\).
  \begin{enumerate}
    \item
    We shall say that the pair \((X,\eta)\) is \textit{irreducible} if
    \(X\) is irreducible, and
    \(\eta\) is its unique generic point.
    \item \label{enumi: defi: pair loc}
    We shall say that the pair \((X,x)\) is \textit{local} if
    \(X\) is local, and
    \(x\) is its unique closed point.
    \item
    We shall say that the triple \((X,x,\eta)\) is \textit{irreducible local} if
    \((X,\eta)\) is irreducible, and
    \((X,x)\) is local.
  \end{enumerate}
  Let \(S\) be a quasi-separated scheme.
  Suppose that \(X\) is an object of \(\Schb{S}\).
  Then these properties may be characterized category-theoretically
  (cf. \autoref{cor: y rsa x is cat},
  \autoref{cor: irred, loc, are cat}
  \ref{enumi: cor: cl pt is cat}
  \ref{enumi: cor: gen pt is cat}
  \ref{enumi: cor: loc is cat}
  \ref{enumi: cor: irred is cat}) from the data \((\Schb{S},X,x,\eta)\).
\end{defi}

\begin{prop}\label{prop: local ring}
  Let \(S\) be a quasi-separated scheme.
  Let \(f:Y\to X\) be a morphism of \(\Schb{S}\)
  and \(x\in X\).
  Then \(f\) is isomorphic as an object of \(\Sch{X}\) to
  the object of \(\Sch{X}\)
  that arises from the natural morphism to \(X\) from
  the spectrum of the local ring at a point \(x\) of \(X\) if and only if
  the following conditions hold:
  \begin{enumerate}
    \item \label{enumi: local ring mono}
    \(f\) is a monomorphism in \(\Schb{S}\).
    \item \label{enumi: local ring closed pt}
    There exists a point \(y\in Y\)
    such that \((Y,y)\) is local, and \(f(y)=x\).
    \item \label{enumi: local ring factor}
    For any local pair \((Z,z)\) and any morphism \(g:Z\to X\)
    such that \(g(z) = x\),
    there exists a unique local morphism \(h:Z\to Y\) such that \(g = f\circ h\).
  \end{enumerate}
  In particular, the property that
  \(f\) is isomorphic as an object of \(\Sch{X}\) to the object of \(\Sch{X}\)
  that arises from the natural morphism to \(X\) from
  the spectrum of the local ring at a point \(x\) of \(X\)
  may be characterized category-theoretically
  (cf. \autoref{defi: pair irred loc} \ref{enumi: defi: pair loc})
  from the data \((\Schb{S}, f:Y\to X, x)\).
\end{prop}


\begin{proof}
  First, we prove necessity.
  Suppose that \(f:Y\to X\) is isomorphic to
  the natural morphism \(\Spec(\mathcal{O}_{X,x}) \to X\) as \(X\)-schemes.
  Write \(y\) for the unique closed point of \(Y\).
  Then the pair \((Y,y)\) is local, and \(f(y)=x\).
  By the universality of localization,
  condition \ref{enumi: local ring factor} is satisfied.
  By \cite[\href{https://stacks.math.columbia.edu/tag/01L9}{Tag 01L9}]{stacks-project},
  \(f\) is a monomorphism of \(\Sche\).
  Hence \(f\) satisfies condition \ref{enumi: local ring mono}.
  This completes the proof of necessity.

  Next, we prove sufficiency.
  Let \(g:Z=\Spec (\mathcal{O}_{X,x}) \to X\) be the natural localization morphism.
  Since \(f:Y\to X\) satisfies condition \ref{enumi: local ring factor},
  there exists a unique local morphism \(h:Z\to Y\) such that \(g=f\circ h\).
  By the necessity portion of \autoref{prop: local ring},
  \(g\) satisfies condition \ref{enumi: local ring factor}.
  Hence there exists a unique local morphism \(h':Y\to Z\) such that \(f=g\circ h'\).
  Thus the equalities \(f = f\circ h\circ h'\) and \(g = g\circ h'\circ h\) hold.
  Since \(f\) satisfies condition \ref{enumi: local ring mono},
  i.e., \(f\) is a monomorphism,
  the equality \(f = f\circ h\circ h'\) implies \(h\circ h' = \id_Y\).
  Since \(g\) satisfies condition \ref{enumi: local ring mono},
  i.e., \(g\) is a monomorphism,
  the equality \(g = g\circ h'\circ h\) implies \(h'\circ h = \id_Z\).
  Thus the equalities \(h\circ h' = \id_Y\) and \(h'\circ h = \id_Z\) imply that
  the morphisms \(h,h'\) are isomorphisms between \(Y,Z\).
  Since \(g=f\circ h\),
  \(f:Y\to X\) is isomorphic to \(g:Z=\Spec(\mathcal{O}_{X,x})\to X\)
  as an \(X\)-scheme.
  This completes the proof of \autoref{prop: local ring}.
\end{proof}

Finally, 
we give a category-theoretic characterization of the objects of \(\Schb{S}\)
whose underlying scheme is isomorphic to the spectrum of a valuation ring.

\begin{prop}\label{prop: Spec val is cat}
  Let \(S\) be a quasi-separated scheme.
  Let \(X\) be an object of \(\Schb{S}\).
  Then \(X\) is isomorphic to the spectrum of a valuation ring
  if and only if the following conditions hold:
  \begin{enumerate}
    \item \label{enumi: Spec val val}
    There exist points \(x_1,x_2\in X\) such that
    \((X,x_1,x_2)\) is valuative.
    \item \label{enumi: Spec val maximal}
    Write \(\eta\) for the unique generic point of \(X\).
    Then for any integral local object \(Y\) and
    any local morphism \(f:Y\to X\),
    if \(Y\btimes_X \Spec(k(\eta)) \cong \Spec(k(\eta))\),
    then \(f\) is an isomorphism.
  \end{enumerate}
  In particular, the property that
  \(X\) is isomorphic to the spectrum of a valuation ring
  may be characterized category-theoretically
  (cf. \autoref{lem: quotient field}, \autoref{prop: val generic pt},
  \autoref{cor: irred, loc, are cat} \ref{enumi: cor: cl pt is cat},
  \autoref{cor: int, int loc is cat} \ref{enumi: cor: int loc is cat})
  from the data \((\Schb{S},X)\).
\end{prop}

\begin{proof}
  First, we prove necessity.
  Suppose that \(X = \Spec(A)\) is the spectrum of a valuation ring \(A\).
  Write \(\eta\in X\) for the unique generic point
  and \(x\in X\) for the unique closed point.
  By \autoref{rem: valuative remark} \ref{enumi: rem: val rem val is val},
  the triple \((X,x,\eta)\) is valuative in \(\Schb{S}\).
  Hence \(X\) satisfies condition \ref{enumi: Spec val val}.
  In the remainder of the proof of necessity,
  we prove that \(X\) satisfies condition \ref{enumi: Spec val maximal}.
  Let \(K\) be the field of fractions of \(A\),
  \(Y = \Spec(B)\) the spectrum of a local domain,
  \(\xi\in Y\) the unique generic point of \(Y\),
  \(y\in Y\) the unique closed point of \(Y\), and
  \(f:Y\to X\) a local morphism such that
  \(Y\btimes_X \Spec(k(\eta)) \cong \Spec(k(\eta))\).
  Since \(A,B\) are local domains, and
  \(K\) is the field of fractions of \(A\),
  the tensor product \(K\otimes_AB\) is integral.
  Hence \(Y\times_X\Spec(k(\eta))\cong Y\btimes_X\Spec(k(\eta))\cong \Spec(k(\eta))\)
  (cf. \autoref{lem: fiber product bbullet}
  \ref{enumi: fiber product bbullet not red}
  \ref{enumi: fiber product bbullet red}),
  i.e., \(B\otimes_AK \cong K\).
  Since \(K\) is a localization of \(A\),
  the field of fractions of \(B\) is isomorphic to \(K\).
  Since \(f\) is a local morphism,
  the local domain \(B\) dominates the valuation ring \(A\)
  in the field \(K\).
  Since the valuation ring \(A\) is maximal for the relation
  of domination among the set of local subrings in \(K\)
  (\cite[\href{https://stacks.math.columbia.edu/tag/00I8}{Section 00I8}]{stacks-project}),
  \(f\) is an isomorphism.
  This completes the proof of necessity.

  Next, we prove sufficiency.
  Suppose that \(X\) satisfies conditions
  \ref{enumi: Spec val val} and \ref{enumi: Spec val maximal}.
  By \autoref{defi: valuative} \ref{enumi: defi: valuative str loc}
  and \autoref{lem: str local is local domain} \ref{enumi: lem: str local is local domain},
  \(X\) is isomorphic to the spectrum of a local domain.
  Write \(A\dfn \Gamma(X,\mathcal{O}_X)\) and
  \(K\) for the field of fractions of \(A\).
  Let \(B\) be a valuation ring which dominates \(A\) in \(K\)
  (cf. \cite[\href{https://stacks.math.columbia.edu/tag/00IA}{Tag 00IA}]{stacks-project}).
  Since \(X\) satisfies condition \ref{enumi: Spec val maximal},
  the morphism \(\Spec(B) \to \Spec(A)\) is an isomorphism
  (cf. \autoref{lem: fiber product bbullet}
  \ref{enumi: fiber product bbullet not red}
  \ref{enumi: fiber product bbullet red}).
  Hence \(A = B\), i.e., \(A\) is a valuation ring.
  This completes the proof of \autoref{prop: Spec val is cat}.
\end{proof}



\section{The Underlying Topological Space}
\label{section: undetlying top}

In this section,
we give a functorial category-theoretic algorithm for
reconstructing the underlying topological space of the objects of \(\Schb{S}\)
from the intrinsic structure of the abstract category \(\Schb{S}\).
First, 
we consider the underlying topological space of fiber products in \(\Schb{S}\).

\begin{lem}\label{lem: btimes to times}
  Let \(S\) be a quasi-separated scheme.
  Let \(f:X\to Z, g:Y\to Z\) be morphisms of \(\Schb{S}\).
  Suppose that the fiber product \(X\btimes_ZY\) exists in \(\Schb{S}\).
  Then the following assertions hold:
  \begin{enumerate}
    \item \label{enumi: lem: btimes isom to times}
    If \(\red\not\in \bbullet\),
    then the natural morphism of schemes
    \(v:X\btimes_ZY\to X\times_ZY\)
    is an isomorphism.
    \item \label{enumi: lem: btimes is red of times}
    If \(\red\in \bbullet\),
    then the natural morphism of schemes
    \(v:X\btimes_ZY\to X\times_ZY\)
    is isomorphic to the natural closed immersion
    \((X\times_ZY)_{\red} \to X\times_Z Y\) in \(\Sch{X\times_ZY}\).
  \end{enumerate}
\end{lem}

\begin{proof}
  First, we prove assertion \ref{enumi: lem: btimes isom to times}.
  Let \(U\) be an affine scheme and
  \(f:U\to X\times_ZY\) a morphism of schemes.
  Then by \autoref{lem: affine over qsep is qcpt} \ref{enumi: lem: aff over qsep is qcpt qc}
  \ref{enumi: lem: aff over qsep is qcpt sep},
  together with our assumption that \(\red\not\in \bbullet\),
  \(U\) belongs to \(\Schb{S}\).
  Hence we obtain the following natural bijections:
  \begin{align*}
    \Hom_{\Sch{S}}(U,X\btimes_ZY)
    &\xrightarrow{\sim} \Hom_{\Schb{S}}(U,X) \times_{\Hom_{\Schb{S}}(U,Z)}\Hom_{\Schb{S}}(U,Y) \\
    &= \Hom_{\Sch{S}}(U,X) \times_{\Hom_{\Sch{S}}(U,Z)}\Hom_{\Sch{S}}(U,Y) \\
    &\xleftarrow{\sim} \Hom_{\Sch{S}}(U,X\times_ZY).
  \end{align*}
  Therefore, for any affine scheme \(U\) and any morphism \(f:U\to X\times_ZY\),
  \(f\) factors uniquely through the natural morphism \(v:X\btimes_ZY\to X\times_ZY\).
  This implies that \(v\) is an isomorphism.
  This completes the proof of assertion \ref{enumi: lem: btimes isom to times}.

  Next, we prove assertion \ref{enumi: lem: btimes is red of times}.
  Let \(U\) be a reduced affine scheme and
  \(f: U\to X\times_ZY\) a morphism of schemes.
  Then by \autoref{lem: affine over qsep is qcpt} \ref{enumi: lem: aff over qsep is qcpt qc}
  \ref{enumi: lem: aff over qsep is qcpt sep},
  together with our assumption that \(U\) is reduced,
  \(U\) belongs to \(\Schb{S}\).
  Hence we obtain the following natural bijections:
  \begin{align*}
    \Hom_{\Sch{S}}(U,X\btimes_ZY)
    &\xrightarrow{\sim} \Hom_{\Schb{S}}(U,X) \times_{\Hom_{\Schb{S}}(U,Z)}\Hom_{\Schb{S}}(U,Y) \\
    &= \Hom_{\Sch{S}}(U,X) \times_{\Hom_{\Sch{S}}(U,Z)}\Hom_{\Sch{S}}(U,Y) \\
    &\xleftarrow{\sim} \Hom_{\Sch{S}}(U,X\times_ZY).
  \end{align*}
  Therefore, for any reduced affine scheme \(U\) and any morphism \(f:U\to X\times_ZY\),
  \(f\) factors uniquely through the natural morphism \(v:X\btimes_ZY\to X\times_ZY\).
  This implies that \(v\) is isomorphic to
  the natural closed immersion
  \((X\times_ZY)_{\red} \to X\times_ZY\) in \(\Sch{X\times_ZY}\).
  This completes the proof of \autoref{lem: btimes to times}.
\end{proof}

\begin{cor}\label{cor: btimes to times is surj cl imm}
  Let \(S\) be a quasi-separated scheme.
  Let \(f:X\to Z, g:Y\to Z\) be morphisms of \(\Schb{S}\).
  Suppose that the fiber product \(X\btimes_ZY\) exists in \(\Schb{S}\).
  Then the natural morphism \(v: X\btimes_ZY \to X\times_ZY\)
  is a surjective closed immersion.
  In particular, \(v\) is a homeomorphism.
\end{cor}

\begin{proof}
  \autoref{cor: btimes to times is surj cl imm} follows immediately from
  \autoref{lem: btimes to times} \ref{enumi: lem: btimes isom to times}
  \ref{enumi: lem: btimes is red of times}.
\end{proof}

Next, we give a category-theoretic characterization of the morphisms of
\(\Schb{S}\) whose underlying morphism of schemes is a closed immersion.

The following result was motivated by \cite[Lemma 3.7]{deBr19}.

\begin{prop}\label{prop: closed imm is cat}
  Let \(S\) be a quasi-separated scheme.
  Let \(f:X\to Y\) be a morphism of \(\Schb{S}\).
  Then \(f\) is a closed immersion if and only if
  the following conditions hold:
  \begin{enumerate}
    \item \label{enumi: cl imm cat reg mono}
    \(f\) is a regular monomorphism in \(\Schb{S}\).
    \item \label{enumi: cl imm cat bc exists}
    For any morphism \(T\to Y\),
    the fiber product \(X_{\bbullet,T}\in \Schb{S}\) exists. 
    \item \label{enumi: cl imm cat bc closed pt}
    For any morphism \(T\to Y\) and any closed point \(t\in T\)
    such that \(t\not\in \mathrm{Im}(f_{\bbullet,T}:X_{\bbullet,T}\to T)\),
    the morphism \(X_{\bbullet,T} \coprod \Spec (k(t)) \to T\)
    is a regular monomorphism in \(\Schb{S}\).
  \end{enumerate}
  In particular, the property that \(f\) is a closed immersion
  may be characterized category-theoretically
  (cf. \autoref{lem: quotient field},
  \autoref{cor: irred, loc, are cat} \ref{enumi: cor: cl pt is cat})
  from the data \((\Schb{S},f:Y\to X)\).
\end{prop}

\begin{proof}
  First, we prove necessity.
  Assume that \(f:Y\to X\) is a closed immersion.
  Then, by \autoref{lem: reg mono} \ref{enumi: closed imm is reg mono},
  \(f\) is a regular monomorphism in \(\Schb{S}\),
  i.e., \(f\) satisfies condition \ref{enumi: cl imm cat reg mono}.
  Let \(T\to Y\) be a morphism.
  Since the closed immersion \(f\) is quasi-compact,
  it follows from \autoref{lem: fiber product bbullet}
  \ref{enumi: fiber product bbullet not red}
  \ref{enumi: fiber product bbullet red}
  that the fiber product \(X_{\bbullet,T}\) exists, i.e.,
  \(f\) satisfies condition \ref{enumi: cl imm cat bc exists}.
  Let \(t\in T\setminus \mathrm{Im}(f_{\bbullet,T})\) be a closed point.
  Since \(f_T\) is a closed immersion,
  and \(t\not\in \mathrm{Im}(f_{\bbullet,T})\) is a closed point of \(T\),
  the morphism \(X_{\bbullet,T} \coprod \Spec (k(t)) \to T\) is a closed immersion.
  Hence by \autoref{lem: reg mono} \ref{enumi: closed imm is reg mono},
  the morphism \(X_{\bbullet,T} \coprod \Spec (k(t)) \to T\)
  is a regular monomorphism in \(\Schb{S}\).
  Thus \(f\) satisfies condition \ref{enumi: cl imm cat bc closed pt}.
  This completes the proof of necessity.

  Next, we prove sufficiency.
  Assume that \(f\) satisfies conditions
  \ref{enumi: cl imm cat reg mono}, \ref{enumi: cl imm cat bc exists}, and
  \ref{enumi: cl imm cat bc closed pt}.
  Then, by \autoref{lem: reg mono} \ref{enumi: reg mono is imm},
  \(f\) is an immersion.
  Hence it suffices to prove that \(\mathrm{Im}(f) \subset Y\) is closed.
  By condition \ref{enumi: cl imm cat bc exists}
  and \autoref{cor: btimes to times is surj cl imm},
  to prove that \(\mathrm{Im}(f) \subset Y\) is closed,
  it suffices to prove that
  for any affine open subscheme \(V\subset Y\),
  \(\mathrm{Im}(f_{\bbullet,V}) \subset V\) is closed,
  i.e., \(\overline{\mathrm{Im}(f_{\bbullet,V})} = \mathrm{Im}(f_{\bbullet,V})\).
  Let \(V\) be an affine open subscheme of \(Y\)
  and \(v\in V\) a closed point such that \(v\not\in \mathrm{Im}(f_{\bbullet,V})\).
  Since \(f\) satisfies condition \ref{enumi: cl imm cat bc closed pt},
  the morphism \(X_{\bbullet,V} \coprod \Spec (k(v)) \to V\)
  is a regular monomorphism in \(\Schb{S}\).
  Hence by \autoref{lem: reg mono} \ref{enumi: reg mono is imm},
  the morphism \(X_{\bbullet,V} \coprod \Spec (k(v)) \to V\) is an immersion,
  i.e., \(v\) is an isolated point of
  the topological subspace \(\mathrm{Im}(f_{\bbullet,V}) \cup \left\{ v\right\} \subset V\).
  Thus there exists an open neighborhood \(v\in V'\subset V\)
  such that \(\mathrm{Im}(f_{\bbullet,V})\cap V' = \emptyset\),
  i.e., \(v\not\in \overline{\mathrm{Im}(f_{\bbullet,V})}\).
  This implies that the subset
  \(\overline{\mathrm{Im}(f_{\bbullet,V})}\setminus \mathrm{Im}(f_{\bbullet,V}) \subset V\)
  does not contain any closed point of \(V\).
  Since \(V\) is an affine scheme,
  to prove that \(\overline{\mathrm{Im}(f_{\bbullet,V})} = \mathrm{Im}(f_{\bbullet,V})\),
  it suffices to prove that
  \(\overline{\mathrm{Im}(f_{\bbullet,V})}\setminus \mathrm{Im}(f_{\bbullet,V})\subset V\)
  is closed.

  By \autoref{cor: btimes to times is surj cl imm},
  the natural morphism \(X_{\bbullet,V}\to X_V\) is a surjective closed immersion.
  Since \(f_{\bbullet,V}:X_{\bbullet,V}\to V\) is the composite of
  the natural surjective closed immersion \(X_{\bbullet,V} \to X_V\) and
  the immersion \(f_V: X_V\to V\),
  \(f_{\bbullet,V}\) is an immersion.
  Hence \(f_{\bbullet,V}\)
  admits a factorization \(X_{\bbullet,V}\xrightarrow{g} U \subset V\)
  such that \(g:X_{\bbullet,V}\to U\) is a closed immersion, and
  \(U\subset V\) is an open immersion.
  Let \(p\in \overline{\mathrm{Im}(f_{\bbullet,V})}\setminus \mathrm{Im}(f_{\bbullet,V})\)
  be a point.
  Suppose that \(p\in U\).
  Then, since \(p\not\in \mathrm{Im}(f_{\bbullet,V}) = \mathrm{Im}(g)\),
  the open neighborhood \(p\in U\setminus \mathrm{Im}(g)\) does not intersect
  \(\mathrm{Im}(f_{\bbullet,V})\).
  This contradicts the fact that \(p\in \overline{\mathrm{Im}(f_{\bbullet,V})}\).
  Hence \(p\not\in U\).
  This implies that
  \(\overline{\mathrm{Im}(f_{\bbullet,V})}\setminus \mathrm{Im}(f_{\bbullet,V})
  \subset \overline{\mathrm{Im}(f_{\bbullet,V})}\setminus U\).
  Since \(\mathrm{Im}(f_{\bbullet,V}) \subset U\),
  the inclusion \(\overline{\mathrm{Im}(f_{\bbullet,V})}\setminus U\subset
  \overline{\mathrm{Im}(f_{\bbullet,V})}\setminus \mathrm{Im}(f_{\bbullet,V})\) holds.
  Thus it holds that
  \(\overline{\mathrm{Im}(f_{\bbullet,V})}\setminus \mathrm{Im}(f_{\bbullet,V})
  = \overline{\mathrm{Im}(f_{\bbullet,V})}\setminus U\).
  Since \(\overline{\mathrm{Im}(f_{\bbullet,V})}\subset V\) is closed, and
  \(U\subset V\) is open,
  we thus conclude that the subset
  \(\overline{\mathrm{Im}(f_{\bbullet,V})}\setminus \mathrm{Im}(f_{\bbullet,V})\subset V\)
  is closed.
  This completes the proof of \autoref{prop: closed imm is cat}.
\end{proof}

\begin{cor}\label{cor: open imm is cat}
  Let \(S\) be a quasi-separated scheme.
  Let \(X\) be an object of \(\Schb{S}\) and
  \(f:U\to X\) a morphism of \(\Schb{S}\).
  Then \(f\) is an open immersion if and only if
  there exists a (necessarily quasi-compact)
  closed immersion \(i:F\to X\) in \(\Schb{S}\)
  such that the following conditions hold:
  \begin{enumerate}
    \item \label{enumi: open imm empty}
    \(F\btimes_XU = \emptyset\) (cf. \autoref{lem: fiber product bbullet}
    \ref{enumi: fiber product bbullet not red}
    \ref{enumi: fiber product bbullet red}).
    \item \label{enumi: open imm univ}
    If a morphism \(g:Y\to X\) in \(\Schb{S}\) satisfies \(Y\btimes_X F = \emptyset\)
    (cf. \autoref{lem: fiber product bbullet}
    \ref{enumi: fiber product bbullet not red}
    \ref{enumi: fiber product bbullet red}),
    then \(g\) factors uniquely through \(f:U\to X\).
  \end{enumerate}
  In particular, the property that \(f\) is an open immersion
  may be characterized category-theoretically
  (cf. \autoref{prop: closed imm is cat})
  from the data \((\Schb{S},f:Y\to X)\).
\end{cor}

\begin{proof}
  First, we prove necessity.
  Assume that \(f:U\to X\) is an open immersion.
  Write \(i: F\dfn \left( X\setminus \mathrm{Im}(f) \right)_{\red} \to X\)
  for the natural closed immersion.
  Since \(i\) is quasi-compact and separated,
  \(i\) belongs to \(\Schb{S}\).
  Since \(F\times_X U = \emptyset\),
  condition \ref{enumi: open imm empty} holds.
  Let \(g:Y\to X\) be a morphism in \(\Schb{S}\) such that \(Y\btimes_X F = \emptyset\).
  Then the underlying morphism of topological spaces
  \(|g| : |Y| \to |X|\) factors uniquely through
  the open immersion of topological spaces \(|f| : |U|\to |X|\).
  Thus \(g : Y\to X\) factors uniquely through \(f:U\to X\).
  This completes the proof of necessity.

  Next, we prove sufficiency.
  Assume that there exists a closed immersion \(i:F\to X\) in \(\Schb{S}\) such that
  \(f:U\to X\) and \(i:F\to X\) satisfy conditions
  \ref{enumi: open imm empty} and \ref{enumi: open imm univ}.
  Write \(j: V = X\setminus \mathrm{Im}(i)\to X\) for the open immersion
  determined by the open subset \(X\setminus \mathrm{Im}(i)\subset X\).
  By the necessity portion of \autoref{cor: open imm is cat},
  the open immersion \(j: V\to X\) satisfies conditions
  \ref{enumi: open imm empty} and \ref{enumi: open imm univ}.
  Hence there exist morphisms \(g:U\to V, h:V\to U\) such that
  \(j\circ g = f, f\circ h = j\).
  By the uniqueness portion of condition \ref{enumi: open imm univ},
  the equalities \(h\circ g = \id_U, g\circ h = \id_V\) hold.
  This proves that \(g\) is an isomorphism,
  i.e., \(f = j\circ g\) is an open immersion.
  This completes the proof of sufficiency.
\end{proof}

Next, we give a category-theoretic reconstruction of
the underlying topological space of an object of \(\Schb{S}\).

\begin{reconstruction}\label{reconstruction: Top}
  Let \(S\) be a quasi-separated scheme.
  Let \(X\) be an object of \(\Schb{S}\).
  We define a (\vsm) topological space
  \begin{align*}
    \mathsf{Op}_{\bbullet/S}(X) &\dfn
    \Set{U\subset \mathsf{Pt}_{\bbullet/S}(X) |
    \begin{array}{l}
      \text{there exists a collection of open} \\
      \text{immersions
      \(\left\{i_\lambda:U_\lambda\to X\right\}_{\lambda\in \Lambda}\) in \(\Schb{S}\)} \\
      \text{such that \(U = \bigcup_{\lambda\in\Lambda}\im(\mathsf{Pt}_{\bbullet/S}(i_\lambda))\)}
    \end{array}
    }, \\
    \mathsf{Sp}_{\bbullet/S}(X) &\dfn (\mathsf{Pt}_{\bbullet/S}(X),\mathsf{Op}_{\bbullet/S}(X)).
  \end{align*}
  These (\vsm) sets are defined
  completely in terms of properties that may be
  characterized category-theoretically (cf. \autoref{cor: open imm is cat})
  from the data \((\Schb{S},X)\). 
  Since any open subscheme \(U\subset X\) admits a covering by affine open subschemes,
  it follows from \autoref{lem: affine over qsep is qcpt}
  \ref{enumi: lem: aff over qsep is qcpt qc}
  \ref{enumi: lem: aff over qsep is qcpt sep} that
  the bijection \(\eta_X:\mathsf{Sp}_{\bbullet/S}(X)\to |X|\)
  (cf. \autoref{reconstruction: Set}) induces
  a natural bijection between \(\mathsf{Op}_{\bbullet/S}(X)\) and
  the set of open subsets of \(X\).
  Hence \(\mathsf{Sp}_{\bbullet/S}(X)\) is a topological space, and
  the bijection \(\eta_X\) may be regarded
  (by a slight abuse of notation) as a homeomorphism.
  Since for any morphism \(f:X\to Y\) of \(\Schb{S}\),
  the diagram
  \begin{equation}
    \label{equiv: reconstruction Top commute}
    \begin{CD}
      \mathsf{Sp}_{\bbullet/S}(X) @> \eta_X > \sim > |X| \\
      @V \mathsf{Pt}_{\bbullet/S}(f) VV @VV |f| V \\
      \mathsf{Sp}_{\bbullet/S}(Y) @> \eta_Y > \sim > |Y|
    \end{CD}
    \tag{\(\dagger\)}
  \end{equation}
  (cf. \autoref{reconstruction: Set}) commutes,
  the map
  \(\mathsf{Pt}_{\bbullet/S}(f):\mathsf{Sp}_{\bbullet/S}(X)\to \mathsf{Sp}_{\bbullet/S}(Y)\)
  is continuous.
  We shall write \(\mathsf{Sp}_{\bbullet/S}(f)\) for \(\mathsf{Pt}_{\bbullet/S}(f)\).
  Thus we obtain a functor
  \(\mathsf{Sp}_{\bbullet/S}:\Schb{S} \to \TopV\).

  Let \(\bigstar \dfn \TopU\) or \(\TopV\).
  We define
  \begin{equation}
    \label{equation: underlying top functor}
    \begin{aligned}
      U_{\bbullet/S}^{\bigstar}: \Schb{S} &\to \bigstar \\
      X &\mapsto |X|.
    \end{aligned}
    \tag{\(\ddagger\)}
  \end{equation}
  Let \(i^{\sfTop}_{\univ{U}\in\univ{V}}: \TopU \to \TopV\)
  be the inclusion,
  then the equality
  \(U_{\bbullet/S}^{\TopV} = i^{\sfTop}_{\univ{U}\in\univ{V}}\circ U_{\bbullet/S}^{\TopU}\) holds.
  Since \(\eta_X\) is a homeomorphism,
  \(\eta\) may be regarded (by a slight abuse of notation) as a natural isomorphism
  \(\eta: \mathsf{Sp}_{\bbullet/S} \xrightarrow{\sim}
  U_{\bbullet/S}^{\TopV} =
  i^{\sfTop}_{\univ{U}\in\univ{V}}\circ U_{\bbullet/S}^{\TopU}\).
\end{reconstruction}

Since the functor \(\mathsf{Sp}_{\bbullet/S}\) is defined
category-theoretically from the data \(\Schb{S}\),
the following lemma holds:

\begin{lem}\label{lem: Sp equiv commute}
  Let \(S,T\) be quasi-separated schemes and
  \(F:\Schb{S} \xrightarrow{\sim} \Schc{T}\) an equivalence.
  Then \(S,T,F\) determine an isomorphism \(\rho^{\mathsf{Sp}}\) between
  the two composite functors of the following diagram:
  \[
  \begin{CD}
    \Schb{S} @>F>> \Schc{T} \\
    @V\mathsf{Sp}_{\bbullet/S}VV @VV\mathsf{Sp}_{\bbbullet/T}V \\
    \TopV @= \TopV.
  \end{CD}
  \]
\end{lem}

\begin{proof}
  Let \(X\) be an object of \(\Schb{S}\).
  By \autoref{lem: Pt equiv commute},
  the map \(\rho_X : \mathsf{Pt}_{\bbullet/S}(X) \to \mathsf{Pt}_{\bbbullet/T}(F(X))\)
  constructed in the proof of \autoref{lem: Pt equiv commute} is bijective.
  Hence to prove \autoref{lem: Sp equiv commute},
  it suffices to prove that
  the map \(\rho_X : \mathsf{Sp}_{\bbullet/S}(X) \to \mathsf{Sp}_{\bbbullet/T}(F(X))\)
  is continuous and open.

  By the definitions of the topologies on
  \(\mathsf{Sp}_{\bbullet/S}(X)\) and \(\mathsf{Sp}_{\bbbullet/T}(F(X))\),
  to prove that \(\rho_X\) is continuous,
  it suffices to prove that
  for any open immersion \(j:V\to F(X)\) in \(\Schc{T}\),
  \(\rho_X^{-1}(\im(\mathsf{Sp}_{\bbbullet/T}(j)))\)
  is an open subset of \(\mathsf{Sp}_{\bbullet/S}(X)\).
  Let \(j:V\to F(X)\) be an open immersion.
  If \(F^{-1}\) is a quasi-inverse of \(F\),
  then by \autoref{cor: open imm is cat},
  the morphism \(j':F^{-1}(V) \to X\) corresponding to \(j:V\to F(X)\)
  is an open immersion.
  Moreover, by \autoref{lem: Pt equiv commute}, the equality
  \(\im(\mathsf{Sp}_{\bbullet/S}(j')) = \rho_X^{-1}(\im(\mathsf{Sp}_{\bbbullet/T}(j)))\) holds.
  Thus \(\rho_X\) is continuous.

  By the definitions of the topologies on
  \(\mathsf{Sp}_{\bbullet/S}(X)\) and \(\mathsf{Sp}_{\bbbullet/T}(F(X))\),
  to prove that \(\rho_X\) is open,
  it suffices to prove that
  for any open immersion \(i:V\to X\) in \(\Schb{S}\),
  \(\rho_X(\im(\mathsf{Sp}_{\bbullet/S}(i)))\)
  is an open subset of \(\mathsf{Sp}_{\bbbullet/T}(X)\).
  Let \(i:U\to X\) be an open immersion.
  By \autoref{cor: open imm is cat},
  \(F(i):F(U) \to F(X)\) is an open immersion.
  Moreover, by \autoref{lem: Pt equiv commute}, the equality
  \(\rho_X(\im(\mathsf{Sp}_{\bbullet/S}(i))) = \im(\mathsf{Sp}_{\bbbullet/T}(F(i)))\) holds.
  Thus \(\rho_X\) is an open map.
  This completes the proof of \autoref{lem: Sp equiv commute}.
\end{proof}

\begin{cor}\label{cor: underlying top equiv commute}
  Let \(S,T\) be schemes and
  \(F:\Schb{S} \xrightarrow{\sim} \Schc{T}\) an equivalence.
  Then the following diagram commutes up to natural isomorphism:
  \[
  \begin{CD}
    \Schb{S} @>F>> \Schc{T} \\
    @VU_{\bbullet/S}^{\TopU}VV @VVU_{\bbbullet/T}^{\TopU}V \\
    \TopU @= \TopU.
  \end{CD}
  \]
\end{cor}

\begin{proof}
  It follows from \autoref{lem: Sp equiv commute} and \autoref{reconstruction: Top}
  that we obtain the following natural isomorphisms:
  \begin{align*}
    i^{\sfTop}_{\univ{U}\in\univ{V}}\circ U_{\bbullet/S}^{\TopU}
    \xleftarrow{\sim} \mathsf{Sp}_{\bbullet/S}
    \xrightarrow{\sim} \mathsf{Sp}_{\bbbullet/T} \circ F
    \xrightarrow{\sim} i^{\sfTop}_{\univ{U}\in\univ{V}}\circ U_{\bbbullet/T}^{\TopU} \circ F.
  \end{align*}
  Since \(i^{\sfTop}_{\univ{U}\in\univ{V}}\) is fully faithful,
  the composite isomorphism of the above display determines
  a natural isomorphism
  \(U_{\bbullet/S}^{\TopU} \xrightarrow{\sim} U_{\bbbullet/T}^{\TopU} \circ F\).
\end{proof}

By \autoref{reconstruction: Top},
topological properties of schemes or morphisms may be
characterized category-theoretically.
For example, the following assertion holds:

\begin{cor}\label{cor: qc is cat}
  Let \(S\) be a quasi-separated scheme.
  \begin{enumerate}
    \item \label{enumi: qc sch is cat}
    Let \(X\) be an object of \(\Schb{S}\).
    Then the property that \(X\) is quasi-compact
    may be characterized category-theoretically
    from the data \((\Schb{S},X)\).
    \item \label{enumi: qc morph is cat}
    Let \(f:Y\to Z\) be a morphism of \(\Schb{S}\).
    Then the property that \(f\) is quasi-compact
    may be characterized category-theoretically
    from the data \((\Schb{S},f)\).
  \end{enumerate}
\end{cor}

\begin{proof}
  \autoref{cor: qc is cat}
  follow immediately from \autoref{reconstruction: Top}.
\end{proof}

\begin{cor}\label{cor: qs sep are cat}
  Let \(S\) be a quasi-separated scheme.
  Let \(X\) be an object of \(\Schb{S}\)
  and \(f:Y\to Z\) a morphism of \(\Schb{S}\).
  Then the following assertions hold:
  \begin{enumerate}
    \item \label{enumi: qs sch is cat}
    \(X\) is quasi-separated if and only if
    for any two open immersions
    \(U_1\to X, U_2\to X\) in \(\Schb{S}\) such that
    \(U_1,U_2\) are quasi-compact,
    \(U_1\btimes_X U_2\) exists and is quasi-compact.
    In particular,
    the property that \(X\) is quasi-separated
    may be characterized category-theoretically
    (cf. \autoref{cor: qc is cat})
    from the data \((\Schb{S},X)\).
    \item \label{enumi: qs morph is cat}
    \(f\) is quasi-separated if and only if
    for any three open immersions
    \(W\to Z, V_1\to Y, V_2 \to Y\) in \(\Schb{S}\) such that
    \(W,V_1,V_2\) are quasi-compact and quasi-separated, and
    \(V_1\to Y\to Z, V_2\to Y\to Z\) factor (necessarily uniquely) through \(W\to Z\),
    \(V_1\btimes_Y V_2\) exists and is quasi-compact.
    In particular,
    the property that \(f\) is quasi-separated
    may be characterized category-theoretically
    (cf. \autoref{cor: qc is cat}, \autoref{cor: qs sep are cat} \ref{enumi: qs sch is cat})
    from the data \((\Schb{S},f)\).
    \item \label{enumi: sep over S is cat}
    \(X\) is separated \textbf{over \(S\)} if and only if
    the image of the diagonal morphism \(\Delta_X^{\bbullet}:X\to X\btimes X\)
    (where we note that \(X\btimes X = X\btimes_SX\) exists in \(\Schb{S}\)
    by \autoref{lem: fiber product bbullet}
    \ref{enumi: fiber product bbullet not qcpt and red}
    \ref{enumi: fiber product bbullet not qcpt and in red}
    \ref{enumi: fiber product bbullet not red}
    \ref{enumi: fiber product bbullet red})
    is closed.
    In particular,
    the property that \(X\) is separated over \(S\)
    may be characterized category-theoretically
    from the data \((\Schb{S},X)\).
  \end{enumerate}
\end{cor}

\begin{proof}
  Assertions \ref{enumi: qs sch is cat} and
  \ref{enumi: qs morph is cat} follow formally from
  \autoref{lem: fiber product bbullet}
  \ref{enumi: fiber product bbullet not red}
  \ref{enumi: fiber product bbullet red},
  \autoref{lem: affine over qsep is qcpt}
  \ref{enumi: lem: aff over qsep is qcpt qc}
  \ref{enumi: lem: aff over qsep is qcpt sep},
  \autoref{cor: btimes to times is surj cl imm}, and
  \autoref{reconstruction: Top}, together with
  the definition of the notion of a quasi-separated morphism.
  Assertion \ref{enumi: sep over S is cat} follows immediately from
  \autoref{cor: btimes to times is surj cl imm} and \autoref{reconstruction: Top},
  together with
  the definition of the notion of a separated morphism.
\end{proof}

\begin{cor}\label{cor: sep, univ cl are cat}
  Let \(S\) be a quasi-separated scheme.
  Let \(f:X\to Y\) be a morphism of \(\Schb{S}\).
  Then the following assertions hold:
  \begin{enumerate}
    \item \label{enumi: sep is cat}
    \(f\) is separated if and only if
    \(f\) is quasi-separated, and
    for any commutative diagram
    \[
    \begin{tikzpicture}[auto]
      \node (A) at (0,1.5) {\(U\)};
      \node (A') at (0,0) {\(T\)};
      \node (B) at (4,1.5) {\(X\)};
      \node (B') at (4,0) {\(Y\)};
      \node (C) at (0.2,0.15) {\(\)};
      \node (C') at (3.8,1.35) {\(\)};
      \draw[->] (A) to node {\(\scriptstyle g'\)} (B);
      \draw[->] (A) to node[swap] {\(\scriptstyle f'\)} (A');
      \draw[->] (B) to node {\(\scriptstyle f\)} (B');
      \draw[->] (A') to node[swap] {\(\scriptstyle g\)} (B');
      \draw[->, transform canvas={xshift=1.5pt, yshift=-1.5pt}] (C) to node {\(\scriptstyle h_1\)} (C');
      \draw[->, transform canvas={xshift=-1.5pt, yshift=1.5pt}] (C) to node[swap] {\(\scriptstyle h_2\)} (C');
    \end{tikzpicture}
    \]
    in \(\Schb{S}\) such that
    \begin{itemize}
      \item[\(\bullet\)]
      \(T\) is isomorphic to the spectrum of a valuation ring, and
      \item[\(\bullet\)]
      \(f':U\to T\) is isomorphic as an object of \(\Sch{T}\)
      to the object of \(\Sch{T}\) that arises from
      the natural morphism to \(T\) from the spectrum of
      the residue field at the generic point of \(T\),
    \end{itemize}
    it holds that \(h_1 = h_2\).
    In particular,
    the property that \(f\) is separated
    may be characterized category-theoretically
    (cf. \autoref{lem: quotient field},
    \autoref{cor: irred, loc, are cat}
    \ref{enumi: irred local 2},
    \autoref{prop: Spec val is cat},
    \autoref{cor: qs sep are cat} \ref{enumi: qs morph is cat})
    from the data \((\Schb{S},f)\).
    \item \label{enumi: qc + univ cl is cat}
    \(f\) is quasi-compact and universally closed if and only if
    \(f\) is quasi-compact, and
    for any diagram
    \[
    \begin{CD}
      U @> g' >> X \\
      @V f' VV @VV f V \\
      T @> g >> Y
    \end{CD}
    \]
    in \(\Schb{S}\) such that
    \begin{itemize}
      \item[\(\bullet\)]
      \(T\) is isomorphic to the spectrum of a valuation ring, and
      \item[\(\bullet\)]
      \(f':U\to T\) is isomorphic as an object of \(\Sch{T}\)
      to the object of \(\Sch{T}\) that arises from
      the natural morphism to \(T\) from the spectrum of
      the residue field at the generic point of \(T\),
    \end{itemize}
    there exists a morphism \(h: T\to X\) such that
    \(g' = h\circ f', g = f\circ h\).
    In particular,
    the property that \(f\) is quasi-compact and universally closed
    may be characterized category-theoretically
    (cf. \autoref{lem: quotient field},
    \autoref{cor: irred, loc, are cat}
    \ref{enumi: irred local 2},
    \autoref{prop: Spec val is cat},
    \autoref{cor: qc is cat} \ref{enumi: qc morph is cat})
    from the data \((\Schb{S},f)\).
  \end{enumerate}
\end{cor}

\begin{proof}
  These assertions follow from \autoref{lem: affine over qsep is qcpt}
  \ref{enumi: lem: aff over qsep is qcpt qc}
  \ref{enumi: lem: aff over qsep is qcpt sep},
  together with well-known valuative criteria
  (cf. \cite[\href{https://stacks.math.columbia.edu/tag/01KD}{Tag 01KD}]{stacks-project}
  \cite[\href{https://stacks.math.columbia.edu/tag/01KZ}{Tag 01KZ}]{stacks-project},
  \cite[\href{https://stacks.math.columbia.edu/tag/01L0}{Tag 01L0}]{stacks-project},
  \cite[\href{https://stacks.math.columbia.edu/tag/01KF}{Tag 01KF}]{stacks-project}).
\end{proof}

Finally,
we conclude one of the main results of the present paper as follows:

\begin{cor}\label{cor: bbullet equal circ}
  Let \(S,T\) be quasi-separated schemes.
  If the (\vsm) categories \(\Schb{S}\), \(\Schc{T}\) are equivalent,
  then the [possibly empty] subsets
  \(\bbullet , \bbbullet \subset \left\{ \mathrm{red,qcpt,qsep,sep} \right\}\)
  (such that \(\left\{\mathrm{qsep,sep}\right\}\not\subset \bbullet,\bbbullet\))
  coincide, i.e., \(\bbullet = \bbbullet\).
\end{cor}


\begin{proof}
  Let \(F:\Schb{S}\xrightarrow{\sim} \Schc{T}\) be an equivalence.
  It suffices to prove that \(\bbbullet \subset \bbullet\).
  Assume that there exists an element \(\star\in \bbbullet\setminus \bbullet\).
  Write
  \begin{align*}
    Z &\dfn \Spec(\mathbb{Z}[\varepsilon, X_1,X_2,\cdots,X_n,\cdots]/(\varepsilon^2)), \\
    W &\dfn \Spec(\mathbb{Z}[\varepsilon, X_1,X_2,\cdots,X_n,\cdots]/(\varepsilon^2))\setminus\{(0,0,\cdots)\},
  \end{align*}
  where \(n\) ranges over the positive integers. Then
  \begin{itemize}
    \item
    \(S\times Z\) satisfies \(\mathrm{qcpt}/S, \mathrm{sep}/S\), but not \(\red\),
    \item
    \((S\times W)_{\red}\) satisfies \(\mathrm{red,sep}/S\), but not \(\qcpt/S\),
    \item
    \((S\times (Z\coprod_WZ))_{\red}\)
    satisfies \(\mathrm{red,qcpt}/S\), but not \(\qsep/S\), and
    \item
    \(( \mathbb{A}^1_S\coprod_{\mathbb{A}^1_S\setminus\{0\}} \mathbb{A}^1_S)_{\red}\)
    satisfies \(\mathrm{red,qcpt}/S, \mathrm{qsep}/S\), but not \(\sep/S\).
  \end{itemize}
  Since \(\star\not\in \bbullet\), and
  \(\left\{\mathrm{qsep,sep}\right\}\not\subset \bbullet,\bbbullet\),
  it follows immediately that
  there exists an object \(X\) of \(\Schb{S}\)
  (which may in fact be taken to be one of the four schemes in the above display)
  such that \(X\) does not satisfy \(\star\).
  Then by \autoref{cor: red is cat},
  \autoref{cor: qc is cat} \ref{enumi: qc morph is cat},
  and \autoref{cor: qs sep are cat}
  \ref{enumi: qs morph is cat} \ref{enumi: sep over S is cat},
  \(F(X)\) does not satisfy property \(\star\).
  This contradicts the fact that \(\star\in\bbbullet\).
  Thus it holds that \(\bbbullet \subset \bbullet\).
\end{proof}


\section{Locally of Finite Presentation Morphisms}
\label{section: loc of fp}

In this section,
we give a category-theoretic characterization of the morphisms of
\(\Schb{S}\) whose underlying morphism of schemes is locally of finite presentation.
To give a category-theoretic characterization of such morphisms of \(\Schb{S}\),
we define the following properties:

\begin{defi} \label{defi: ess fp stwise fp}
  \begin{enumerate}
    \item \label{enumi: defi: ess fp}
    Let \(f:A\to B\) be a local homomorphism between local rings.
    We shall say that \(f\) is \textit{essentially of finite presentation}
    (respectively, \textit{essentially of finite type}) if
    \(f\) admits a factorization \(A\to C\to B\) such that
    \(A\to C\) is a ring homomorphism of finite presentation (respectively, of finite type), and
    \(C\to B\) is the localization morphism \textbf{at a prime ideal} of \(C\).
    \item
    Let \(f:X\to Y\) be a morphism of schemes.
    We shall say that \(f\) is \textit{stalkwise of finite presentation}
    (respectively, \textit{stalkwise of finite type}) at a point \(x\in X\) if
    the morphism of local rings \(\mathcal{O}_{Y,f(x)}\to \mathcal{O}_{X,x}\)
    is essentially of finite presentation
    (respectively, essentially of finite type).
    We shall say that \(f\) is \textit{stalkwise of finite presentation}
    (respectively, \textit{stalkwise of finite type}) if
    \(f\) is stalkwise of finite presentation at every point of \(X\).
  \end{enumerate}
\end{defi}

\begin{rem}\label{rem: ess.f.p. stwise.f.p.}
  \
  \begin{enumerate}
    \item \label{enumi: rem: ess.f.p. stwise.f.p. loc f.p. is stwise f.p.}
    If a morphism of schemes is locally of finite presentation,
    then it is stalkwise of finite presentation.
    \item
    If a morphism of rings is essentially of finite presentation,
    then it is essentially of finite type.
    \item
    For any local ring \(A\),
    the identity morphism \(\id_A\) is essentially of finite presentation.
    \item \label{enumi: rem: composite ess fp}
    By \cite[\href{https://stacks.math.columbia.edu/tag/07DS}{Tag 07DS}]{stacks-project},
    the composite of two morphisms of local rings which are essentially of finite presentation
    (respectively, essentially of finite type)
    is essentially of finite presentation (respectively, essentially of finite type).
  \end{enumerate}
\end{rem}

\begin{lem}\label{lem: fp lemma}
  Let \(p:R\to A\) and \(f:A\to B\) be ring homomorphisms.
  Then the following assertion holds:
  \begin{enumerate}
    \item \label{enumi: fp lemma ess ft}
    Assume that \(R,A,B\) are local rings, and
    \(p,f\) are local homomorphisms.
    If \(f\circ p\) is essentially of finite type,
    then \(f\) is also essentially of finite type.
  \end{enumerate}
  Assume that there exists a ring homomorphism \(s:B\to A\).
  Then the following assertions hold:
  \begin{enumerate}[start=2]
    \item \label{enumi: fp lemma ft fp}
    If \(f\) is of finite type, and \(s\circ f = \id_A\),
    then \(s\) is of finite presentation.
    \item \label{enumi: fp lemma fp}
    If \(f\circ p\) is of finite presentation, and \(s\circ f = \id_A\),
    then \(p\) is of finite presentation.
    \item \label{enumi: fp lemma ess ft fp}
    Assume that \(A,B\) are local rings, and
    \(f,s\) are local homomorphisms.
    If \(f\) is essentially of finite type, and \(s\circ f = \id_A\),
    then \(s\) is essentially of finite presentation.
    \item \label{enumi: fp lemma ess}
    Assume that \(R,A,B\) are local rings, and
    \(p,f,s\) are local homomorphisms.
    If \(f\circ p\) is essentially of finite presentation, and \(s\circ f = \id_A\),
    then \(p\) is essentially of finite presentation.
  \end{enumerate}
\end{lem}

\begin{proof}
  First, we verify assertion \ref{enumi: fp lemma ess ft}.
  Assume that \(f\circ p\) is essentially of finite type.
  Then \(f\circ p\) admits a factorization
  \(R\xrightarrow{g} R'\xrightarrow{i} B\) such that
  \(g:R\to R'\) is a ring homomorphism of finite type, and
  \(i:R'\to B\) is the localization morphism at a prime ideal of \(R'\).
  Hence the morphism \(g'\dfn \id_A\otimes g: A\to A\otimes_R R'\) is
  a ring homomorphism of finite type.
  Since \(i:R'\to B\) is a localization morphism at a prime ideal of \(R'\),
  if we write \(A''\) for the image of
  the natural morphism \(A\otimes_R R' \to B\) determined by
  the equality \(f\circ p = i\circ g\),
  then the natural inclusion
  \(j:A''\hookrightarrow B\) is the localization morphism at the prime ideal \(A''\cap \mathfrak{m}_B\),
  where \(\mathfrak{m}_B\) is the unique maximal ideal of \(B\).
  Thus, if we write \(h:A\otimes_R R' \to A\) for the surjection, then
  the composite \(h\circ g':A\to A\otimes_R R'\to A''\) is
  a ring homomorphism of finite type, and
  we obtain a factorization
  \(A\xrightarrow{h\circ g'} A''\xrightarrow{j} B\) of \(f\) such that
  \(h\circ g':A\to A''\) is a ring homomorphism of finite type, and
  \(j:A''\to B\) is the localization morphism at a prime ideal of \(A''\).
  This proves that \(f\) is essentially of finite type.
  This completes the proof of assertion \ref{enumi: fp lemma ess ft}.

  Next, we verify assertion \ref{enumi: fp lemma ft fp}.
  Assume that \(f\) is of finite type, and \(s\circ f = \id_A\).
  Then \(s\) is surjective.
  Hence it suffices to prove that \(\ker(s)\) is finitely generated.
  Since \(f\) is of finite type,
  \(f\) admits a factorization
  \(A\xrightarrow{i} A[x_1,\cdots,x_r]\xrightarrow{q} B\) such that
  \(i:A\to A[x_1,\cdots,x_r]\) is the natural inclusion, and
  \(q\) is a surjective ring homomorphism.
  Thus we obtain the equality \(\ker(s\circ q) = (x_1-i(s(q(x_1))),\cdots,x_r-i(s(q(x_r))))\).
  In particular, \(\ker(s\circ q)\) is a finitely generated ideal of \(A[x_1,\cdots,x_r]\).
  Since \(q\) is a surjective ring homomorphism,
  we obtain the equality \(\ker(s) = q(\ker (s\circ q))\).
  Thus \(\ker(s)\) is a finitely generated ideal of \(B\).
  This completes the proof of assertion \ref{enumi: fp lemma ft fp}.

  Next, we verify assertion \ref{enumi: fp lemma fp}.
  Assume that \(f\circ p : R\to B\) is of finite presentation, and \(s\circ f = \id_A\).
  Then \(f\) is of finite type.
  By \ref{enumi: fp lemma ft fp},
  \(s\) is of finite presentation.
  Hence \(p = s\circ f\circ p\) is of finite presentation.
  This completes the proof of assertion \ref{enumi: fp lemma fp}.

  Next, we verify assertion \ref{enumi: fp lemma ess ft fp}.
  Assume that \(f\) is essentially of finite type, and \(s\circ f = \id_A\).
  Then \(f\) admits a factorization \(A\xrightarrow{f'} B'\xrightarrow{q} B\) such that
  \(f'\) is a ring homomorphism of finite type, and
  \(q\) is the localization morphism at a prime ideal \(\mathfrak{q}'\subset B'\).
  It follows from the equality \(s\circ f=\id_A\) that
  \(s\) is a surjective local homomorphism.
  Hence to prove assertion \ref{enumi: fp lemma ess ft fp},
  it suffices to prove that \(\ker(s)\) is a finitely generated ideal.
  Since \(f'\) is a ring homomorphism of finite type, and \(s\circ q\circ f' = \id_A\),
  it follows from \ref{enumi: fp lemma ft fp} that
  \(s\circ q: B'\to A\) is of finite presentation.
  Hence \(\ker(s\circ q)\) is a finitely generated ideal.
  Since \(s\) is a local homomorphism, \(\ker(s\circ q) \subset \mathfrak{q}'\subset B'\).
  Since \(q\) is the localization morphism at \(\mathfrak{q}'\subset B'\),
  \(\ker(s)\) is the localization of \(\ker(s\circ q)\) at \(\mathfrak{q}'\subset B'\).
  Thus \(\ker(s)\) is a finitely generated ideal of \(B\).
  This completes the proof of \ref{enumi: fp lemma ess ft fp}.

  Finally, we verify assertion \ref{enumi: fp lemma ess}.
  Assume that \(f\circ p\) is essentially of finite presentation,
  and \(s\circ f = \id_A\).
  By \ref{enumi: fp lemma ess ft},
  \(f\) is essentially of finite type.
  Since \(f\) is essentially of finite type, and \(s\circ f = \id_A\),
  it follows from \ref{enumi: fp lemma ess ft fp} that
  \(s\) is essentially of finite presentation.
  Thus \(p = s\circ (f\circ p)\) is essentially of finite presentation.
  This completes the proof of \autoref{lem: fp lemma}.
\end{proof}

\begin{cor}\label{cor: fp lemma cor}
  Let \(S,X,Y\) be schemes
  and \(p:X\to S, f:Y\to X, s:X\to Y\) morphisms.
  Assume that \(f\circ s = \id _X\), and
  \(p\circ f:Y\to S\) is locally of finite presentation
  (respectively, stalkwise of finite presentation).
  Then \(p\) is locally of finite presentation
  (respectively, stalkwise of finite presentation).
\end{cor}

\begin{proof}
  \autoref{cor: fp lemma cor} follows immediately
  from \autoref{lem: fp lemma} \ref{enumi: fp lemma fp} \ref{enumi: fp lemma ess}.
\end{proof}

The following lemma will play an important role
in the category-theoretic characterization of
stalkwise of finite presentation morphisms that
we give in \autoref{cor: stwise fp is cat} below.

\begin{lem}\label{lem: limit bbullet lemma}
  Let \(S\) be a quasi-separated scheme.
  Let \((V_\lambda, \lambda\in \Lambda)\) be a diagram in \(\Schb{S}\)
  indexed by a cofiltered category \(\Lambda\).
  Assume that there exists a \usm \
  cofinal (hence cofiltered) subcategory
  \(\Lambda'\subset \Lambda\) such that
  \(V_{\lambda}\) is affine for every \(\lambda \in \Lambda'\).
  Then the limits \(\plim_{\lambda\in\Lambda} V_\lambda\)
  and \(\bplim_{\lambda\in\Lambda} V_\lambda\) exist
  in \(\Sch{S}\) and \(\Schb{S}\), respectively,
  and are naturally isomorphic in \(\Sch{S}\) (\(\supset \Schb{S}\)).
  Moreover, \(\bplim_{\lambda\in\Lambda} V_\lambda\) may be taken to be
  \(\Spec(\colim_{\lambda\in\Lambda'^{\mathrm{op}}}
  \Gamma(V_\lambda,\mathcal{O}_{V_\lambda}))\).
\end{lem}

\begin{proof}
  It suffices to prove that
  the limit \(\plim_{\lambda\in\Lambda} V_\lambda\) exists in \(\Sch{S}\)
  and in fact belongs to \(\Schb{S}\).
  Since \(\Lambda'\subset \Lambda\) is cofinal,
  we may assume without loss of generality that \(\Lambda'=\Lambda\).
  Then we may take \(\plim_{\lambda\in\Lambda} V_\lambda\) to be
  \(\Spec(\colim_{\lambda\in\Lambda} \Gamma(V_\lambda,\mathcal{O}_{V_\lambda}))\).
  In particular, \(\plim_{\lambda\in\Lambda} V_\lambda\) exists in \(\Sch{S}\)
  and is affine.
  By
  \autoref{lem: affine over qsep is qcpt}
  \ref{enumi: lem: aff over qsep is qcpt qc}
  \ref{enumi: lem: aff over qsep is qcpt sep},
  \(\plim_{\lambda\in\Lambda} V_\lambda\) is separated and quasi-compact over \(S\).
  If each \(V_\lambda\) is reduced,
  then \(\plim_{\lambda\in\Lambda} V_\lambda\) is also reduced.
  Thus \(\plim_{\lambda\in\Lambda} V_\lambda\) belongs to \(\Schb{S}\).
  This completes the proof of \autoref{lem: limit bbullet lemma}.
\end{proof}

\begin{lem}\label{lem: filtered colim of loc is loc}
  Let \((C_\lambda,p_{\lambda\mu})_{\lambda\in\Lambda}\) be
  a filtered diagram of local rings and local homomorphisms.
  Then \(\colim_{\lambda\in\Lambda} C_\lambda\) is local, and
  the natural morphism
  \(i_\lambda: C_\lambda \to \colim_{\lambda\in\Lambda} C_\lambda\)
  is local.
\end{lem}

\begin{proof}
  Let \(I\subsetneq C\dfn \colim_{\lambda\in\Lambda} C_\lambda\) be an ideal.
  Write \(\mathfrak{m}_\lambda\) for the maximal ideal of \(C_\lambda\),
  \(\mathfrak{m}\dfn \colim_{\lambda\in\Lambda} \mathfrak{m}_\lambda\), and
  \(I_\lambda \dfn i_\lambda^{-1}(I)\subsetneq C_\lambda\)
  for the inverse image of the ideal \(I\).
  Since \(C_\lambda\) is local,
  \(I_\lambda \subset \mathfrak{m}_\lambda\).
  Since the natural morphism
  \(\colim_{\lambda\in\Lambda} I_\lambda \to I\) is an isomorphism,
  it holds that \(I\subset \mathfrak{m}\).
  This implies that
  \(\mathfrak{m}\) is a unique maximal ideal of \(C\),
  i.e., that \(C\) is local,
  and \(i_\lambda\) is local.
  This completes the proof of \autoref{lem: filtered colim of loc is loc}.
\end{proof}

\begin{cor}\label{cor: limit bullet exists}
  Let \(S\) be a quasi-separated scheme.
  Then the following assertions hold:
  \begin{enumerate}
    \item \label{enumi: cor: limit bullet exists loc}
    Let \((V_\lambda,p_{\lambda\mu})_{\lambda\in\Lambda}\) be
    a \usm \ cofiltered diagram in \(\Schb{S}\) such that
    each \(V_\lambda\) is local, and each \(p_{\lambda\mu}\) is local.
    Then \(\bplim_{\lambda\in \Lambda}V_\lambda\) exists.
    Moreover, \(\bplim_{\lambda\in \Lambda}V_\lambda\) is local, and
    each projection \(\bplim_{\lambda\in \Lambda}V_\lambda \to V_\lambda\) is local.
    \item \label{enumi: cor: limit bullet exists open imm}
    Let \(X\in \Schb{S}\) be an object and \(x\in X\) a point.
    Write \(I_X(x)\) for the full subcategory of \(\Schb{X}\) consisting of
    all open immersions \(U\to X\) in \(\Schb{X}\) such that
    \(x\) belongs to the image of \(U\to X\).
    (Thus, by a slight abuse of notation,
    \(I_X(x)\) may also be regarded,
    by applying the natural functor \(\Schb{X}\to \Schb{S}\),
    as a diagram indexed by \(I_X(x)\).)
    Then \(\bplim_{U\in I_X(x)}U\) exists, and
    may be taken to be \(\Spec(\mathcal{O}_{X,x})\).
  \end{enumerate}
\end{cor}

\begin{proof}
  First, we prove assertion \ref{enumi: cor: limit bullet exists loc}.
  Since a local object in \(\Schb{S}\) is affine,
  it follows from \autoref{lem: limit bbullet lemma} that
  \(\bplim_{\lambda\in \Lambda}V_\lambda\) exists, and
  \[\bplim_{\lambda\in \Lambda}V_\lambda \cong
  \Spec(\colim_{\lambda\in\Lambda^{\mathrm{op}}}\Gamma(V_\lambda,\mathcal{O}_{V_\lambda})). \]
  By \autoref{lem: filtered colim of loc is loc},
  \(\bplim_{\lambda\in\Lambda}V_\lambda\) is local, and
  each projection \(\bplim_{\lambda\in \Lambda}V_\lambda \to V_\lambda\)
  is local.
  Thus assertion \ref{enumi: cor: limit bullet exists loc} holds.

  Next, we prove assertion \ref{enumi: cor: limit bullet exists open imm}.
  Since the \usm \ subcategory
  \[
  \Set{(U\subset X)\in I_X(x) |
  \text{\(U\) is an affine open neighborhood of \(x\)}} \subset I_X(x)
  \]
  is cofinal,
  assertion \ref{enumi: cor: limit bullet exists open imm} follows from
  \autoref{lem: limit bbullet lemma}.
\end{proof}

\begin{lem}\label{lem: ess fp colim}
  Let \(f:R\to A\) be a local homomorphism between \usm \ local rings.
  Write \(\LAlg_R\) for the category of \usm \ local \(R\)-algebras and
  local homomorphisms of \(R\)-algebras.
  Then the following assertions hold:
  \begin{enumerate}
    \item \label{enumi: ess fp neccesity}
    If \(f\) is essentially of finite presentation,
    then for any \usm \ filtered diagram
    \((C_\lambda,p_{\lambda\mu})_{\lambda\in\Lambda}\) in \(\LAlg_R\),
    the natural morphism
    \[
    \varphi : \colim_{\lambda\in\Lambda} \Hom_{\LAlg_R}(A,C_\lambda)
    \to \Hom_{\LAlg_R}(A,\colim_{\lambda\in\Lambda} C_\lambda)
    \]
    (where we note that by \autoref{lem: filtered colim of loc is loc},
    \(\colim_{\lambda\in\Lambda} C_\lambda\) is local)
    is surjective.
    \item \label{enumi: ess fp diagram exists general}
    There exists a \usm \ filtered diagram
    \((C_\lambda,p_{\lambda\mu})_{\lambda\in\Lambda}\) in \(\LAlg_R\)
    such that each \(R\)-algebra \(C_\lambda\) is
    essentially of finite presentation, and
    \(\colim_{\lambda\in\Lambda}C_\lambda \cong A\).
    \item \label{enumi: ess fp diagram exists red}
    Assume that \(R,A\) are reduced.
    Then there exists a \usm \ filtered diagram
    \((C_\lambda,p_{\lambda\mu})_{\lambda\in\Lambda}\) in \(\LAlg_R\)
    such that each \(R\)-algebra \(C_\lambda\) is reduced and
    essentially of finite type, and
    \(\colim_{\lambda\in\Lambda}C_\lambda \cong A\).
  \end{enumerate}
\end{lem}

\begin{proof}
  First, we prove assertion \ref{enumi: ess fp neccesity}.
  Let \((C_\lambda,p_{\lambda\mu})_{\lambda\in\Lambda}\)
  be a \usm \ filtered diagram in \(\LAlg_R\).
  Write \(C \dfn \colim_{\lambda\in\Lambda}C_\lambda\) and
  \(p_\lambda: C_\lambda \to C\) for the natural morphism.
  By \autoref{lem: filtered colim of loc is loc},
  \(C\) is local, and each \(p_\lambda\) is local.
  Let \(g:A\to C\) be a local homomorphism of \(R\)-algebras.
  Since \(f\) is essentially of finite presentation,
  the morphism of local rings
  \(f:R\to A\) admits a factorization
  \(R \xrightarrow{h} B \xrightarrow{j} A\)
  such that \(h:R \to B\) is of finite presentation, and
  \(j:B\to A\) is the localization morphism
  at a prime ideal of \(B\).
  Since \(B\) is an \(R\)-algebra of finite presentation,
  it follows from
  \cite[\href{https://stacks.math.columbia.edu/tag/00QO}{Tag 00QO}]{stacks-project}
  that there exist an index \(\lambda\in \Lambda\) and
  a morphism \(B\xrightarrow{f_{\lambda}} C_{\lambda}\) of \(R\)-algebras
  such that the following diagram commutes:
  \[
  \begin{tikzpicture}[auto]
    \node (A) at (0,1.5) {\(B\)};
    \node (A') at (0,0) {\(C_\lambda\)};
    \node (B) at (5,1.5) {\(A\)};
    \node (B') at (5,0) {\(C\).};
    \draw[->] (A) to node[swap] {\(\scriptstyle f_{\lambda}\)} (A');
    \draw[->] (B) to node {\(\scriptstyle g\)} (B');
    \draw[->] (A) to node {\(\scriptstyle j\)} (B);
    \draw[->] (A') to node[swap] {\(\scriptstyle p_{\lambda}\)} (B');
  \end{tikzpicture}
  \]
  Since \(j\) is the localization morphism at a prime ideal of \(B\),
  and \(p_\lambda,g\) are local homomorphisms of local rings,
  there exists a local morphism \(f_{\lambda}':A \to C_\lambda\)
  such that \(f_\lambda = f_{\lambda}'\circ j\):
  \begin{equation*}
    \begin{tikzpicture}[auto]
      \node (A) at (0,1.5) {\(B\)};
      \node (A') at (0,0) {\(C_\lambda\)};
      \node (B) at (5,1.5) {\(A\)};
      \node (B') at (5,0) {\(C\).};
      \draw[->] (A) to node[swap] {\(\scriptstyle f_{\lambda}\)} (A');
      \draw[->] (B) to node {\(\scriptstyle g\)} (B');
      \draw[->] (A) to node {\(\scriptstyle j\)} (B);
      \draw[->] (A') to node[swap] {\(\scriptstyle p_{\lambda}\)} (B');
      \draw[->] (B) to node[swap] {\(\scriptstyle f_{\lambda}'\)} (A');
    \end{tikzpicture}
  \end{equation*}
  Since \(g\circ j = p_{\lambda} \circ f_{\lambda}
  = p_{\lambda} \circ f_{\lambda}'\circ j\), and
  \(j\) is the localization morphism at a prime ideal of \(B\),
  the equality \(g = p_{\lambda} \circ f'_{\lambda}\) holds.
  This implies that
  \[\varphi([f'_{\lambda}]) = p_\lambda \circ f'_{\lambda} = g,\]
  where \([f'_{\lambda}]\)
  denotes the element \(\in\colim_{\lambda\in\Lambda} \Hom_{\LAlg_R}(A,C_\lambda)\)
  determined by \(f'_\lambda \in \Hom_{\LAlg_R}(A,C_\lambda)\).
  Thus \(\varphi\) is surjective.
  This completes the proof of assertion \ref{enumi: ess fp neccesity}.

  Next, we prove assertion \ref{enumi: ess fp diagram exists general}.
  By \cite[\href{https://stacks.math.columbia.edu/tag/0BUF}{Tag 0BUF}]{stacks-project},
  there exists a \usm \ filtered diagram
  \((B_\lambda,q_{\lambda\mu})_{\lambda\in\Lambda}\)
  of \(R\)-algebras of finite presentation and homomorphisms of \(R\)-algebras
  such that \(\colim_{\lambda\in\Lambda}B_\lambda \cong A\).
  Write \(q_\lambda:B_\lambda \to A\) for the natural morphism,
  \(\mathfrak{m}\) for the maximal ideal of \(A\),
  \(\mathfrak{q}_\lambda \dfn q_\lambda^{-1}(\mathfrak{m})\),
  \(C_\lambda \dfn B_{\lambda, \mathfrak{q}_\lambda}\) for the localization
  of \(B_\lambda\) at \(\mathfrak{q}_{\lambda}\).
  Then the morphism \(q_\lambda:B_\lambda \to A\) induces
  the local homomorphism of local \(R\)-algebras
  \(p_\lambda:C_\lambda\to A\), and
  any isomorphism \(\colim_{\lambda\in\Lambda}B_\lambda \xrightarrow{\sim} A\)
  induces an isomorphism \(\colim_{\lambda\in\Lambda}C_\lambda \xrightarrow{\sim} A\).
  Since \(C_\lambda\) is essentially of finite presentation over \(R\),
  this completes the proof of assertion \ref{enumi: ess fp diagram exists general}.

  Finally, we prove assertion \ref{enumi: ess fp diagram exists red}.
  By assertion \ref{enumi: ess fp diagram exists general},
  there exists a \usm \ filtered diagram
  \((B_\lambda,q_{\lambda\mu})_{\lambda\in\Lambda}\)
  of essentially of finite type local \(R\)-algebras
  and local homomorphisms such that \(\colim_{\lambda\in\Lambda}B_\lambda\cong A\).
  Write \(q_\lambda:B_\lambda\to A\) for the natural morphism.
  Since \(A\) is reduced,
  \(\im(q_\lambda)\subset A\) is a reduced local \(R\)-algebra.
  Then it holds that \(A = \bigcup_{\lambda\in\Lambda}\im(q_\lambda)\).
  This implies that \(\colim_{\lambda\in\Lambda}\im(q_\lambda)\cong A\).
  Write \(C_{\lambda} \dfn \im(q_\lambda)\) and
  \(p_{\lambda\mu}: C_{\lambda}\to C_\mu\) for the morphism
  induced by \(q_{\lambda\mu}\).
  Since \(q_\lambda\) is local,
  \((C_\lambda,p_{\lambda\mu})_{\lambda\in\Lambda}\) is
  a \usm \ filtered diagram in \(\LAlg_R\).
  This completes the proof of \autoref{lem: ess fp colim}.
\end{proof}

\begin{cor}\label{cor: ess fp colim}
  Let \(S\) be a quasi-separated scheme.
  Let \(X,Y\) be local objects of \(\Schb{S}\) and
  \(f:X\to Y\) a local morphism in \(\Schb{S}\).
  Write \((\Schb{Y})_{\mathrm{loc}}\) for
  the full subcategory of \(\Schb{Y}\)
  determined by the local objects and the local morphisms.
  Consider the following condition on \(f\):
  \begin{enumerate}[label=(\fnsymbol*),start=2]
    \item \label{enumi: ess fp condition}
    For any \usm \ cofiltered diagram
    \((V_\lambda,p_{\lambda\mu})_{\lambda\in\Lambda}\)
    of \((\Schb{Y})_{\mathrm{loc}}\),
    the natural morphism
    \[
    \varphi : \colim_{\lambda\in\Lambda^{\mathrm{op}}}
    \Hom_{(\Schb{Y})_{\mathrm{loc}}}(V_\lambda ,X)
    \to \Hom_{(\Schb{Y})_{\mathrm{loc}}}(\bplim_{\lambda\in\Lambda} V_\lambda ,X)
    \]
    is surjective
    (where we note that
    by \autoref{cor: limit bullet exists}
    \ref{enumi: cor: limit bullet exists loc},
    \(\bplim_{\lambda\in\Lambda} V_\lambda\) exists in \((\Schb{Y})_{\mathrm{loc}}\)).
  \end{enumerate}
  Then the following assertions hold:
  \begin{enumerate}
    \item \label{enumi: loc fp implies condition}
    If the morphism of local rings
    \(f^{\#}:\Gamma(Y,\mathcal{O}_Y)\to \Gamma(X,\mathcal{O}_X)\)
    is essentially of finite presentation,
    then condition \ref{enumi: ess fp condition} holds.
    \item \label{enumi: condition implies loc ft}
    If condition \ref{enumi: ess fp condition} holds,
    then the morphism of local rings
    \(f^{\#}:\Gamma(Y,\mathcal{O}_Y)\to \Gamma(X,\mathcal{O}_X)\)
    is essentially of finite type.
  \end{enumerate}
\end{cor}

\begin{proof}
  Assertion \ref{enumi: loc fp implies condition}
  follows immediately from
  \autoref{lem: ess fp colim} \ref{enumi: ess fp neccesity}.
  In the remainder of the proof of \autoref{cor: ess fp colim},
  we verify assertion \ref{enumi: condition implies loc ft}.
  Assume that \(f\) satisfies condition \ref{enumi: ess fp condition}.
  By \autoref{lem: affine over qsep is qcpt}
  \ref{enumi: lem: aff over qsep is qcpt qc}
  \ref{enumi: lem: aff over qsep is qcpt sep},
  \autoref{lem: limit bbullet lemma}, and
  \autoref{lem: ess fp colim} \ref{enumi: ess fp diagram exists general}
  \ref{enumi: ess fp diagram exists red},
  there exists a cofiltered diagram
  \((V_\lambda,p_{\lambda\mu})_{\lambda\in\Lambda}\)
  of \((\Schb{Y})_{\mathrm{loc}}\) such that
  \(\bplim_{\lambda\in \Lambda}V_\lambda \cong X\), and
  for any \(\lambda\in \Lambda\),
  \(\Gamma(V_\lambda,\mathcal{O}_{V_\lambda})\) is
  essentially of finite type over \(\Gamma(Y,\mathcal{O}_Y)\).
  Let \(g\in \Hom_{(\Schb{Y})_{\mathrm{loc}}}(\bplim_{\lambda\in\Lambda} V_\lambda ,X)\)
  be an isomorphism.
  Write
  \(p_\lambda:\bplim_{\lambda\in\Lambda}V_\lambda\to V_\lambda\)
  for the natural projection.
  Since \(f\) satisfies condition \ref{enumi: ess fp condition},
  there exist an index \(\lambda\in\Lambda\) and
  a morphism \(g':V_\lambda\to X\) in \((\Schb{Y})_{\mathrm{loc}}\)
  such that \(g'\circ p_\lambda = g\).
  Since \(g\) is an isomorphism,
  it follows from \autoref{lem: fp lemma} \ref{enumi: fp lemma ess ft} that
  \(p_{\lambda}^{\#}\) is essentially of finite type.
  Since
  \(g'^{\#}\circ f^{\#}:\Gamma(Y,\mathcal{O}_Y)\to \Gamma(V_\lambda,\mathcal{O}_{V_\lambda})\)
  is essentially of finite type,
  it follows from
  \autoref{rem: ess.f.p. stwise.f.p.} \ref{enumi: rem: composite ess fp} that
  \(f^{\#} = (g^{-1})^{\#}\circ p_\lambda^{\#}\circ ({g'}^{\#}\circ f^{\#})\)
  is essentially of finite type.
  This completes the proof of \autoref{cor: ess fp colim}.
\end{proof}

\begin{defi}\label{defi: cat st ft}
  Let \(S\) be a quasi-separated scheme.
  Let \(f:X\to Y\) be a morphism of \(\Schb{S}\).
  We shall say that
  \(f\) is \textit{category-theoretically stalkwise of finite presentation}
  if \(f\) satisfies following condition:
  \begin{itemize}
    \item[ \ ] 
    For any \usm \ cofiltered diagram
    \((V_\lambda,p_{\lambda\mu})_{\lambda\in\Lambda}\) of \(\Schb{Y}\) such that
    each \(V_\lambda\) is local, and each \(p_{\lambda\mu}\) is local,
    the natural morphism
    \[
    \varphi : \colim_{\lambda\in\Lambda^{\mathrm{op}}} \Hom_{\Schb{Y}}(V_\lambda ,X)
    \to \Hom_{\Schb{Y}}(\bplim_{\lambda\in\Lambda} V_\lambda ,X)
    \]
    is surjective
    (where we note that
    by \autoref{cor: limit bullet exists}
    \ref{enumi: cor: limit bullet exists loc},
    \(\bplim_{\lambda\in\Lambda} V_\lambda\) exists in \(\Schb{Y}\)).
  \end{itemize}
  Thus the property that
  \(f\) is category-theoretically stalkwise of finite presentation
  is defined completely in terms of properties that may be
  characterized category-theoretically
  (cf. \autoref{cor: irred, loc, are cat}
  \ref{enumi: cor: cl pt is cat} \ref{enumi: cor: loc is cat})
  from the data \((\Schb{S},f)\).
\end{defi}

\begin{cor}\label{cor: stwise fp is cat}
  Let \(S\) be a quasi-separated scheme.
  Let \(f:X\to Y\) be a morphism of \(\Schb{S}\).
  Then the following assertions hold:
  \begin{enumerate}
    \item \label{enumi: st fp is cat st ft}
    If \(f\) is stalkwise of finite presentation,
    then \(f\) is category-theoretically stalkwise of finite presentation.
    \item \label{enumi: cat st ft is st ft}
    If \(f\) is category-theoretically stalkwise of finite presentation,
    then \(f\) is stalkwise of finite type.
  \end{enumerate}
\end{cor}

\begin{proof}
  Assertion \ref{enumi: st fp is cat st ft} follows immediately from
  \autoref{cor: ess fp colim} \ref{enumi: loc fp implies condition}.
  Assertion \ref{enumi: cat st ft is st ft} follows immediately from
  \autoref{cor: ess fp colim} \ref{enumi: condition implies loc ft}.
\end{proof}

\begin{defi}
  Let \(X\) be a scheme.
  We shall say that \(X\) is \textit{stalkwise Noetherian}
  if for any point \(x\in X\),
  \(\mathcal{O}_{X,x}\) is Noetherian.
\end{defi}

\begin{cor}\label{cor: st Noeth st fp}
  Let \(S\) be a quasi-separated scheme.
  Let \(f:X\to Y\) be a morphism of \(\Schb{S}\).
  Assume that \(Y\) is stalkwise Noetherian.
  Then \(f\) is stalkwise of finite presentation if and only if
  \(f\) is category-theoretically stalkwise of finite presentation.
\end{cor}

\begin{proof}
  Necessity follows immediately from
  \autoref{cor: stwise fp is cat} \ref{enumi: st fp is cat st ft}.
  Sufficiency follows immediately from
  \autoref{cor: stwise fp is cat} \ref{enumi: cat st ft is st ft}.
\end{proof}

\begin{rem}\label{rem: A_f}
  Let \(A\) be a ring and \(f\in A\) an element.
  Then it holds that \(A_f \cong A[x]/(xf-1)\).
  Thus \(A_f\) is of finite presentation over \(A\).
\end{rem}

\begin{lem}\label{lem: ess fp is loc at prime}
  Let \(f:B\to A\) be a ring homomorphism and
  \(\mathfrak{p}\in \Spec(A)\) a prime ideal.
  Write \(\mathfrak{q} \dfn f^{-1}(\mathfrak{p})\) and
  \(f_{\mathfrak{p}}:B_{\mathfrak{q}}\to A_{\mathfrak{p}}\)
  for the morphism induced by localizing at \(\mathfrak{q}\) and \(\mathfrak{p}\),
  respectively.
  Assume that \(f_{\mathfrak{p}}\) is essentially of finite presentation.
  Then there exist a \(B\)-algebra \(C\) of finite presentation
  and a prime ideal \(\mathfrak{r}\) of \(C\)
  such that \(A_{\mathfrak{p}}\cong C_{\mathfrak{r}}\)
  as \(B\)-algebras, hence also as \(B_{\mathfrak{q}}\)-algebras.
\end{lem}

\begin{proof}
  By \autoref{defi: ess fp stwise fp} \ref{enumi: defi: ess fp},
  the desired \(B\)-algebra \(C\) may be constructed by taking the quotient of a polynomial ring with finite variation over \(B\) by a finitely generated ideal and localizing it with one element.
  Thus \autoref{lem: ess fp is loc at prime} follows immediately from \autoref{rem: A_f}.
\end{proof}

\begin{defi}
  Let \(X\) be a scheme, and \(x\in X\) a point of \(X\).
  We shall write \(I_X(x)\) for
  the full subcategory of \(\Schb{X}\) consisting of
  all open immersions \(U\to X\) in \(\Schb{X}\) such that
  \(x\) belongs to the image of \(U\to X\).
  (Thus, by a slight abuse of notation,
  if \(X\) is an \(S\)-scheme,
  then \(I_X(x)\) may also be regarded,
  by applying the natural functor \(\Schb{X}\to \Schb{S}\),
  as a diagram indexed by \(I_X(x)\).)

  Let \(A\) be a ring and \(\mathfrak{p}\subset A\) a prime ideal of \(A\).
  We shall define \(J_A(\mathfrak{p})\) for
  the full subcategory of \(\Alg{A}\) determined by
  the set of \(A\)-algebras
  \(\Set{ A_f | f\in A\setminus \mathfrak{p}}\).
  (Thus, by a slight abuse of notation,
  if \(A\) is a \(B\)-algebra,
  then \(J_A(\mathfrak{p})\) may also be regarded,
  by applying the natural functor \(\Alg{A}\to \Alg{B}\),
  as a diagram indexed by \(J_A(\mathfrak{p})\).)
  Then it follows immediately that
  \(\colim_{A_f\in J_A(\mathfrak{p})}A_f \cong A_{\mathfrak{p}}\).
\end{defi}

\begin{lem}\label{lem: fp neighborhood}
  Let \(\xi:B\to A\) be a ring homomorphism and
  \(\mathfrak{p}\in \Spec(A)\) a prime ideal.
  Write \(\mathfrak{p}_B\dfn \xi^{-1}(\mathfrak{p})\),
  \(\xi_{\mathfrak{p}}:B_{\mathfrak{p}_B}\to A_{\mathfrak{p}}\)
  for the morphism induced by localizing at
  \(\mathfrak{p}_B\) and \(\mathfrak{p}\), respectively.
  Assume that the following conditions hold:
  \begin{enumerate}
    \item \label{enumi: fp is ess fp lem}
    \(\xi_{\mathfrak{p}}\) is essentially of finite presentation.
    \item \label{enumi: fp colim lem}
    For any \(B\)-algebra \(R\) and
    any prime ideal \(\mathfrak{r}\subset R\),
    the natural map
    \[
    \colim_{R_r\in J_R(\mathfrak{r})} \Hom_{\Alg{B}}(A,R_r)
    \to \Hom_{\Alg{B}}(A,\colim_{R_r\in J_R(\mathfrak{r})} R_r)
    \]
    is a bijection.
  \end{enumerate}
  Then there exists an element \(f\in A\setminus \mathfrak{p}\) such that
  the composite of \(\xi\) and the localization morphism \(A \to A_f\)
  is of finite presentation.
\end{lem}

\begin{proof}
  By \autoref{lem: ess fp is loc at prime}
  and condition \ref{enumi: fp is ess fp lem},
  there exist a \(B\)-algebra \(C\) of finite presentation,
  a prime ideal \(\mathfrak{q}\subset C\),
  and a morphism of \(B\)-algebras \(\psi: A \to C_{\mathfrak{q}}\)
  such that \(\mathfrak{p} = \psi^{-1}(\mathfrak{q})\), and the morphism
  \(\psi_{\mathfrak{p}}: A_{\mathfrak{p}} \xrightarrow{\sim} C_{\mathfrak{q}}\)
  induced by localizing at \(\mathfrak{p}\) is an isomorphism.
  If \(h\in C\setminus\mathfrak{q}\),
  then write \(\gamma_h:C_h\to C_{\mathfrak{q}}\) for the localization morphism.
  By condition \ref{enumi: fp colim lem},
  there exist an element \(h\in C\) and
  a morphism of \(B\)-algebras \(\psi_h:A\to C_h\)
  such that \[\psi = \gamma_h\circ \psi_h.\]
  By \autoref{rem: A_f},
  \(C_h\) is of finite presentation over \(C\),
  hence over \(B\).
  Thus, by \cite[\href{https://stacks.math.columbia.edu/tag/00QO}{Tag 00QO}]{stacks-project},
  the natural map
  \begin{equation}
    \label{eq: 1 lem: fp neighborhood}
    \colim_{A_f\in J_A(\mathfrak{p})} \Hom_{\Alg{B}}(C_h,A_f)
    \to \Hom_{\Alg{B}}(C_h,\colim_{A_f\in J_A(\mathfrak{p})} A_f)
    \tag{\(\dagger\)}
  \end{equation}
  is surjective.

  For any elements \(f_1,f_2\in A\) and
  any prime ideal \(\mathfrak{p}_1\subset A\) such that \(f_1\not\in \mathfrak{p}_1\),
  we shall write
  \[
  \alpha_{f_1}:A\to A_{f_1} \ , \
  \alpha_{f_1,f_2}:A_{f_1}\to A_{f_1f_2} \ , \
  \alpha_{f_1,\mathfrak{p}_1}:A_{f_1} \to A_{\mathfrak{p}_1}
  \]
  for the respective localization morphisms.
  Since the map \eqref{eq: 1 lem: fp neighborhood} is surjective,
  there exist an element \(f_1\in A\setminus \mathfrak{p}\) and
  a morphism of \(B\)-algebras \(\phi_{f_1}:C_h\to A_{f_1}\) such that
  \(\psi_{\mathfrak{p}}^{-1}\circ \gamma_h
  = \alpha_{f_1,\mathfrak{p}}\circ \phi_{f_1}: C_h \to A_{\mathfrak{p}}\).
  Since \(\psi_{\mathfrak{p}}^{-1}\circ \gamma_h\circ \psi_h
  = \psi_{\mathfrak{p}}^{-1} \circ \psi = \alpha_{1,\mathfrak{p}}\circ \alpha_1\),
  it holds that \(\alpha_{1,\mathfrak{p}}\circ \alpha_1
  = \alpha_{f_1,\mathfrak{p}}\circ \phi_{f_1}\circ \psi_h : A\to A_{\mathfrak{p}}\):
  \[
  \begin{CD}
    A @> \alpha_1 > \sim > A_1 \\
    @V \phi_{f_1}\circ \psi_h VV @VV \alpha_{1,\mathfrak{p}} V \\
    A_{f_1} @> \alpha_{f_1,\mathfrak{p}} >> A_{\mathfrak{p}}.
  \end{CD}
  \]
  Thus, by applying the injectivity portion of condition \ref{enumi: fp colim lem}
  in the case where \(R=A\), and \(\mathfrak{r}=\mathfrak{p}\),
  we conclude that there exists an element \(f_2\in A\setminus \mathfrak{p}\) such that
  \(\alpha_{f_1f_2} = \alpha_{f_1,f_2}\circ \phi_{f_1}\circ \psi_h : A\to A_{f_1f_2}\).

  Write \(f\dfn f_1f_2\in A\setminus \mathfrak{p}\),
  \(D\dfn C_{h\psi_h(f)}\cong C_h\otimes_A A_f\),
  \(\phi_f\dfn \alpha_{f_1,f_2}\circ \phi_{f_1}\), and
  \(\psi_h':A_f\to D\) for the morphism of \(B\)-algebras obtained
  by base changing \(\psi_h:A\to C_h\) by \(\alpha_f:A\to A_f\) (cf. the diagram \eqref{eq: diagram 5.17} displayed in the following).
  Since \(\alpha_f = \alpha_{f_1f_2}
  = \alpha_{f_1,f_2}\circ \phi_{f_1}\circ \psi_h = \phi_f \circ \psi_h\),
  it follows from the universal property of the tensor product that
  there exists a unique morphism \(\phi_f':D\to A_f\)
  such that the following diagram commutes:
  \begin{equation}
    \label{eq: diagram 5.17}
    \begin{aligned}
    \begin{tikzpicture}[auto]
      \node (A) at (0,1) {\(A\)};
      \node (B) at (2,2) {\(A_f\)};
      \node (B') at (2,0) {\(C_h\)};
      \node (C) at (4,1) {\(D\)};
      \node (D) at (6,1) {\(A_f.\)};
      \draw[->] (A) -- node {\(\scriptstyle \alpha_f\)} (B);
      \draw[->] (A) -- node[swap] {\(\scriptstyle \psi_h\)} (B');
      \draw[->] (B) -- node[swap] {\(\scriptstyle \psi_h'\)} (C);
      \draw[->] (B') -- (C);
      \draw[->] (B) to[bend left=20] node  {\(\scriptstyle \id\)} (D);
      \draw[->] (B') to[bend right=12] (D);
      \node (E) at (3.9,0.38) {\(\scriptstyle \phi_f\)};
      \draw[->] (C) -- node {\(\scriptstyle \phi_f'\)} (D);
      \node (F) at (6,0) {\(A_{f_1}\)};
      \draw[->] (B') -- node[swap] {\(\scriptstyle \phi_{f_1}\)} (F);
      \draw[->] (F) -- node[swap] {\(\scriptstyle \alpha_{f_1,f_2}\)} (D);
    \end{tikzpicture}
  \end{aligned}
    \tag{\(\ddagger\)}
  \end{equation}
  By \autoref{rem: A_f},
  \(D\) is of finite presentation over \(C\),
  hence over \(B\).
  Thus, it follows from \autoref{lem: fp lemma} \ref{enumi: fp lemma fp},
  together with the commutativity of the following diagram, that
  the composite \(B\xrightarrow{\xi} A \to A_f\)
  is of finite presentation:
  \[
  \begin{tikzpicture}[auto]
    \node (A) at (0,1.5) {\(B\)};
    \node (A') at (0,0) {\(C\)};
    \node (D) at (3,1.5) {\(A\)};
    \node (D') at (3,0) {\(C_h\)};
    \node (B) at (6,1.5) {\(A_f\)};
    \node (B') at (6,0) {\(D\)};
    \node (F) at (9,0) {\(A_f.\)};
    \draw[->] (A) -- node[swap] {\scriptsize f.p.} (A');
    \draw[->] (B) -- node[swap] {\(\scriptstyle \psi'_h\)} (B');
    \draw[->] (D) -- node[swap] {\(\scriptstyle \psi_h\)} (D');
    \draw[->] (A) -- node {\(\scriptstyle \xi\)} (D);
    \draw[->] (D) -- node {\(\scriptstyle \alpha_f\)} (B);
    \draw[->] (A') -- node {\scriptsize f.p.} (D');
    \draw[->] (D') -- node {\scriptsize f.p.} (B');
    \draw[->] (B') -- node {\(\scriptstyle \phi_f'\)} (F);
    \draw[->] (B) -- node {\(\scriptstyle \id\)} (F);
  \end{tikzpicture}
  \]
  This completes the proof of \autoref{lem: fp neighborhood}.
\end{proof}

\begin{defi}\label{defi: cat ft}
  Let \(S\) be a quasi-separated scheme.
  Let \(f:X\to Y\) be a morphism of \(\Schb{S}\).
  We shall say that \(f\) is
  \textit{category-theoretically of finite presentation} if
  \(f\) satisfies following conditions:
  \begin{enumerate}
    \item \label{enumi: condi 1 defi: cat ft}
    \(f\) is category-theoretically stalkwise of finite presentation.
    \item \label{enumi: condi defi: cat ft}
    For any morphism \(Z\to Y\) in \(\Schb{S}\) and any point \(z\in Z\),
    the natural map
    \[
    \varphi_{z,X}: \colim_{W\in I_Z(z)^{\mathrm{op}}} \Hom_{\Schb{Y}}(W ,X)
    \to \Hom_{\Schb{Y}}(\bplim_{W\in I_Z(z)} W ,X)
    \]
    is a bijection
    (where we note that
    by \autoref{cor: limit bullet exists}
    \ref{enumi: cor: limit bullet exists open imm},
    \(\bplim_{W\in I_Z(z)} W\) exists in \(\Schb{S}\)).
  \end{enumerate}
  Thus the property that
  \(f\) is category-theoretically of finite presentation
  is defined completely in terms of properties that may be
  characterized category-theoretically
  (cf. \autoref{reconstruction: Set},
  \autoref{cor: surj is cat} \ref{enumi: cor: surj is cat x in f},
  \autoref{cor: open imm is cat}, \autoref{defi: cat st ft})
  from the data \((\Schb{S},f)\).
\end{defi}

\begin{cor}\label{cor: loc of fp and cat ft}
  Let \(S\) be a quasi-separated scheme.
  Let \(f:X\to Y\) be a morphism of \(\Schb{S}\).
  Then the following assertions hold:
  \begin{enumerate}
    \item \label{enumi: loc fp cat}
    If \(f\) is locally of finite presentation,
    then \(f\) is category-theoretically of finite presentation.
    \item \label{enumi: cat ft loc fp}
    If \(f\) is category-theoretically of finite presentation,
    and \(Y\) is stalkwise Noetherian,
    then \(f\) is locally of finite presentation.
  \end{enumerate}
\end{cor}

\begin{proof}
  Assertion \ref{enumi: cat ft loc fp} follows immediately from
  \autoref{cor: st Noeth st fp} and \autoref{lem: fp neighborhood}.
  In the remainder of the proof of \autoref{cor: loc of fp and cat ft},
  we prove assertion \ref{enumi: loc fp cat}.
  Assume that \(f\) is locally of finite presentation.
  By \autoref{rem: ess.f.p. stwise.f.p.}
  \ref{enumi: rem: ess.f.p. stwise.f.p. loc f.p. is stwise f.p.} and
  \autoref{cor: stwise fp is cat} \ref{enumi: st fp is cat st ft},
  \(f\) is category-theoretically stalkwise of finite presentation.
  Thus \(f\) satisfies condition \ref{enumi: condi 1 defi: cat ft} of
  \autoref{defi: cat ft}.
  Let \(Z\to Y\) be a morphism and \(z\in Z\) a point of \(Z\).
  Since the full subcategory
  \[I^{\text{aff}}_Z(z) \dfn \Set{ W\subset Z |
  \text{\(W\) is an affine open neighborhood of \(z\)}} \subset I_Z(z)\]
  is cofinal, and \(\Schb{Y}\subset \Sch{Y}\) is a full subcategory,
  the natural morphism
  \[
  \colim_{W\in I^{\text{aff}}_Z(z)^{\mathrm{op}}} \Hom_{\Schb{Y}}(W ,X)
  \xrightarrow{\sim}
  \colim_{W\in I_Z(z)^{\mathrm{op}}} \Hom_{\Sch{Y}}(W ,X)
  \]
  is an isomorphism.
  By \autoref{lem: limit bbullet lemma},
  the natural morphism
  \[
  \bplim_{W\in I_Z(z)} W
  \xrightarrow{\sim} \plim_{W\in I^{\text{aff}}_Z(z)} W
  \]
  is an isomorphism.
  Since each \(W\in I^{\text{aff}}_Z(z)\) is affine, and
  \(f\) is locally of finite presentation,
  it follows from
  \cite[\href{https://stacks.math.columbia.edu/tag/01ZC}{Tag 01ZC}]{stacks-project}
  that \(f\) satisfies condition \ref{enumi: condi defi: cat ft} of
  \autoref{defi: cat ft}.
  This completes the proof of \autoref{cor: loc of fp and cat ft}.
\end{proof}

\begin{cor}\label{cor: loc of fp is cat}
  Let \(S\) be a quasi-separated scheme.
  Let \(f:X\to Y\) be a morphism of \(\Schb{S}\).
  Assume that \(Y\) is stalkwise Noetherian.
  Then \(f\) is locally of finite presentation if and only if
  \(f\) is category-theoretically of finite presentation.
  In particular, if \(S\) is stalkwise Noetherian, then
  the property that \(X\) is locally of finite presentation over \(S\)
  may be characterized category-theoretically
  (cf. \autoref{defi: cat ft})
  from the data \((\Schb{S}, X)\).
\end{cor}

\begin{proof}
  \autoref{cor: loc of fp is cat} follows immediately from
  \autoref{cor: loc of fp and cat ft} \ref{enumi: loc fp cat}
  \ref{enumi: cat ft loc fp}.
\end{proof}

\begin{cor}\label{cor: proper is cat}
  Let \(S\) be a quasi-separated scheme.
  Let \(f:X\to Y\) be a morphism of \(\Schb{S}\).
  Assume that \(Y\) is stalkwise Noetherian.
  Then \(f\) is proper and of finite presentation if and only if
  \(f\) is quasi-compact, separated, universally closed, and
  category-theoretically of finite presentation.
  In particular, if \(S\) is stalkwise Noetherian, then
  the property that \(X\) is proper and of finite presentation over \(S\)
  may be characterized category-theoretically
  (cf. \autoref{cor: sep, univ cl are cat} \ref{enumi: sep is cat}
  \ref{enumi: qc + univ cl is cat}, \autoref{defi: cat ft})
  from the data \((\Schb{S},X)\).
\end{cor}

\begin{proof}
  \autoref{cor: proper is cat} follows immediately from
  \autoref{cor: loc of fp is cat}.
\end{proof}


\section{The Projective Line}
\label{section: P1}

In this section,
we give a category-theoretic characterization of the objects of
\(\Schb{S}\) whose underlying \(S\)-scheme is isomorphic to
the projective line \(\mathbb{P}^1_S\).
We then use this characterization
to give a functorial category-theoretic algorithm for
reconstructing the underlying schemes of the objects of \(\Schb{S}\)
from the intrinsic structure of the abstract category \(\Schb{S}\).
In addition,
we discuss some results related to the various reconstruction algorithms developed
in the present paper.

First, we give a category-theoretic characterization of morphisms
which are isomorphic to the projection \(\mathbb{P}^1_k\to \Spec(k)\),
where \(k\) is a field.

\begin{lem}\label{lem: tr deg 1}
  Let \(S\) be a quasi-separated scheme.
  Let \(f:X\to Y\) be a morphism of \(\Schb{S}\) such that
  \(X,Y\) are isomorphic to the spectrum of a field.
  Write \(K\dfn \Gamma(X,\mathcal{O}_X), k\dfn \Gamma(Y,\mathcal{O}_Y)\).
  Then \(f^{\#}:k\to K\) is a purely transcendental extension of degree \(1\)
  if and only if the following conditions hold:
  \begin{enumerate}
    \item \label{enumi: tr deg 1 not finite}
    There exists a morphism \(g:X\to X\) in \(\Schb{Y}\) such that
    \(g\) is not an isomorphism. 
    \item \label{enumi: tr deg 1 exists isom}
    If \(f\) admits a factorization \(X\to Z\to Y\) in \(\Schb{S}\) such that
    \(Z\) is isomorphic to the spectrum of a field,
    and \(Z\to Y\) is not an isomorphism,
    then there exists an isomorphism \(Z\xrightarrow{\sim} X\) in \(\Schb{Y}\).
  \end{enumerate}
  In particular,
  the property that
  \begin{quote}
    \(X,Y\) are isomorphic to the spectrum of a field, and
    \(f:X\to Y\) is isomorphic as an object of \(\Sch{Y}\)
    to the object of \(\Sch{Y}\)
    that arises from
    a purely transcendental field extension of degree \(1\)
  \end{quote}
  may be characterized category-theoretically
  (cf. \autoref{lem: Spec field})
  from the data \((\Schb{S},\)\space\(f:X\to Y)\).
\end{lem}

\begin{proof}
  First, we prove necessity.
  Assume that \(f^{\#}: k\to K\) is a purely transcendental extension of degree \(1\).
  Then there exists a transcendental element \(t\in K\) such that \(K = k(t)\).
  By considering the \(k\)-algebra morphism \(k(t)\to k(t), t\mapsto t^2\),
  we conclude that \(f\) satisfies condition \ref{enumi: tr deg 1 not finite}.
  Let \(k\to L\to k(t)\) be a factorization of \(f^{\#}\) such that
  \(L\) is a field, and \(k\to L\) is not an isomorphism.
  Then by L\"{u}roth's theorem (cf., e.g., \cite[Chapter IV, Example 2.5.5]{Ha}),
  there exists an element \(u \in L\) such that \(L = k(u)\).
  Hence we obtain an isomorphism of \(k\)-algebras
  \(K\xrightarrow{\sim} L, t\mapsto u\).
  Thus \(f\) satisfies condition \ref{enumi: tr deg 1 exists isom}.
  This completes the proof of the necessity.

  Next, we prove sufficiency.
  Assume that \(k\dfn \Gamma(Y,\mathcal{O}_Y)\) and \(K\dfn \Gamma(X,\mathcal{O}_X)\)
  are fields, and \(f:X\to Y\) satisfies conditions
  \ref{enumi: tr deg 1 not finite} and \ref{enumi: tr deg 1 exists isom}.
  Since \(f\) satisfies condition \ref{enumi: tr deg 1 not finite},
  it follows from
  \cite[\href{https://stacks.math.columbia.edu/tag/0BMD}{Tag 0BMD}]{stacks-project}
  that \(K/k\) is not an algebraic extension.
  Hence there exists a transcendental element \(t\in K\) over \(k\).
  Then, by condition \ref{enumi: tr deg 1 exists isom},
  there exists an isomorphism of \(k\)-algebras \(k(t) \cong K\),
  which implies that \(K\) is a purely transcendental extension of degree \(1\).
  This completes the proof of \autoref{lem: tr deg 1}.
\end{proof}

\begin{lem}\label{lem: P1 over field}
  Let \(S\) be a quasi-separated scheme.
  Let \(X,Y\) be objects of \(\Schb{S}\) such that
  \(Y\) is isomorphic to the spectrum of a field, and
  \(f:X\to Y\) a morphism of \(\Schb{S}\).
  Then \(f:X\to Y\) is isomorphic as an object of \(\Schb{Y}\)
  to the object of \(\Schb{Y}\)
  that arises from the natural projection \(\mathbb{P}^1_Y\to Y\)
  if and only if the following conditions hold:
  \begin{enumerate}
    \item \label{enumi: P1 fp proper}
    \(f\) is proper and of finite presentation.
    \item \label{enumi: P1 int}
    \(X\) is integral.
    \item \label{enumi: P1 tr deg 1}
    If \(\eta\in X\) is the generic point,
    then the composite \(\Spec(k(\eta))\to X\to Y\) is isomorphic
    as a \(Y\)-scheme to
    the spectrum of a purely transcendental field extension of degree \(1\).
    \item \label{enumi: P1 reg}
    For every closed point \(x\in X\),
    the spectrum of the local ring \(\Spec(\mathcal{O}_{X,x})\)
    is isomorphic to the spectrum of a valuation ring.
  \end{enumerate}
  In particular, the property that
  \(Y\) is isomorphic to the spectrum of a field, and
  \(f:X\to Y\) is isomorphic as an object of \(\Schb{Y}\)
  to the object of \(\Schb{Y}\)
  that arises from the natural projection \(\mathbb{P}^1_Y\to Y\)
  may be characterized category-theoretically
  (cf. \autoref{cor: irred, loc, are cat} \ref{enumi: cor: gen pt is cat},
  \autoref{cor: int, int loc is cat} \ref{enumi: cor: int is cat},
  \autoref{prop: local ring}, \autoref{prop: Spec val is cat},
  \autoref{cor: proper is cat}, \autoref{lem: tr deg 1})
  from the data \((\Schb{S},f:X\to Y)\).
\end{lem}

\begin{proof}
  \autoref{lem: P1 over field} follows immediately from
  well-known properties of schemes and valuation rings.
\end{proof}

\begin{lem}\label{lem: ring str A^1 unique}
  Let \(T\) be a reduced scheme.
  Write \(0_T,1_T:T\to \mathbb{A}^1_T\) for the morphisms
  obtained by base-changing the sections
  \(0,1:\Spec(\mathbb{Z})\to \mathbb{A}^1_{\mathbb{Z}}\) and
  \(\mathbb{G}_{m,T}\dfn \mathbb{A}^1_T\setminus \im(0_T)\).
  Then the following assertions hold:
  \begin{enumerate}
    \item \label{enumi: any G_m lem: ring str A^1 unique}
    Let \(G\subset \mathbb{G}_{m,T}\) be an open subscheme
    equipped with a group scheme structure over \(T\)
    whose identity section \(T\to G\)
    (\(\subset \mathbb{G}_{m,T}\subset \mathbb{A}^1_T\)) is \(1_T\).
    Then the open immersion \(G\subset \mathbb{G}_{m,T}\)
    is an isomorphism of group schemes over \(T\).
    \item \label{enumi: any A^1 lem: ring str A^1 unique}
    The set of ring scheme structures on \(\mathbb{A}^1_T\) over \(T\)
    whose additive and multiplicative identity sections
    are \(0_T\) and \(1_T\), respectively,
    is of cardinality one.
  \end{enumerate}
\end{lem}

\begin{proof}
  First, we prove assertion
  \ref{enumi: any G_m lem: ring str A^1 unique} in the case where \(T\) is isomorphic to the spectrum of a separably closed field.
  Write \(C\dfn \mathbb{P}^1_T\setminus G \ (\supset\{0,\infty\})\),
  where we regard \(G\) as an open subscheme of \(\mathbb{P}^1_T\)
  by means of the open immersions \(G\subset \mathbb{G}_{m,T}\subset \mathbb{P}^1_T\);
  \(m:G\times_T G\to G\) for the multiplication morphism.
  For any closed point \(g\in G\),
  the automorphism \(\rho_g:\mathbb{P}^1_T\xrightarrow{\sim}\mathbb{P}^1_T\)
  given by \(m(g,-):G\xrightarrow{\sim}G\) satisfies the property \(\rho_g(C) = C\).
  Since the assignment \(g\mapsto \rho_g\) is clearly injective,
  we thus conclude that
  the group \(\left\{ \rho\in \mathrm{Aut}_T(\mathbb{P}^1_T) \middle| \rho(C)=C\right\}\)
  is of infinite cardinality.
  This implies that \(C=\{0,\infty\}\),
  i.e., that the open immersion \(G \subset \mathbb{G}_{m,T}\)
  is an isomorphism of \(T\)-schemes.

  Next, we prove that the isomorphism
  \(G \xrightarrow{\sim} \mathbb{G}_{m,T}\)
  is compatible with both group scheme structures over \(T\).
  Let us identify \(\Gamma(G,\mathcal{O}_G)\) with \(k[t,1/t]\)
  by means of the isomorphism \(G \xrightarrow{\sim} \mathbb{G}_{m,T}\).
  Then the morphism \(m\) is determined by the element
  \(m^{\#}(t)\in k[t,1/t]\otimes_k k[t,1/t]\).
  Moreover, since \(t\in k[t,1/t]\) is invertible,
  \(m^{\#}(t)\in k[t,1/t]\otimes_k k[t,1/t]\) is invertible.
  Since \(m^{\#}(t)\) is invertible in
  \(k[t,1/t]\otimes_k k[t,1/t]\subset k[t,1/t]\otimes_k k(t)\),
  there exist an element \(f\in k[t,1/t]\setminus \{0\}\) and an integer \(a\)
  such that \(m^{\#}(t) = t^a \otimes f\).
  Since \(m\circ (1_T\times \id_G) = \id_G\),
  it holds that \(f=t\).
  Since \(m\circ (\id_G\times 1_T) = \id_G\), it holds that \(a=1\).
  Thus it holds that \(m^{\#}(t) = t\otimes t\).
  This implies that
  the isomorphism \(G \xrightarrow{\sim} \mathbb{G}_{m,T}\)
  is compatible with both group scheme structures over \(T\).
  This completes the proof of assertion
  \ref{enumi: any G_m lem: ring str A^1 unique} in the case where \(T\) is isomorphic to the spectrum of a separably closed field.

  In the general case, since \(T\) is reduced, and \(G\subset \mathbb{G}_{m,T}\) is an open subscheme, by applying assertion \ref{enumi: any G_m lem: ring str A^1 unique} in the case where \(T\) is isomorphic to the spectrum of a separably closed field to the geometric fibers of \(G\to T\), we conclude that \(|G|=|\mathbb{G}_{m,T}|\), hence, in particular, assertion \ref{enumi: any G_m lem: ring str A^1 unique} hold.
  This completes the proof of assertion \ref{enumi: any G_m lem: ring str A^1 unique}.

  Next, we prove assertion
  \ref{enumi: any A^1 lem: ring str A^1 unique} in the case where \(T\) is isomorphic to the spectrum of a separably closed field \(k\).
  Write \(A \dfn \mathbb{A}^1_T\).
  Let us identify \(\Gamma(A,\mathcal{O}_A)\) with \(k[t]\)
  by means of the isomorphism \(A \xrightarrow{\sim} \mathbb{A}^1_T\).
  Let \((a:A\times_T A \to A, e_a:T\to A, m:A\times_T A\to A, e_m:T\to A)\)
  be a ring scheme structure on \(A\) over \(T\),
  where \(a\) is the addition morphism,
  \(e_a = 0_T\) is the identity section of the additive structure,
  \(m\) is the multiplication morphism, and
  \(e_m = 1_T\) is the identity section of the multiplicative structure.
  By \cite[Theorem 6.1]{AR},
  the group variety \(A^{\times}\) of units of \(A\)
  is an open subscheme of \(A\setminus \{0\} = \mathbb{G}_{m,T}\).
  Hence, by \ref{enumi: any G_m lem: ring str A^1 unique},
  the open immersion \(A^{\times}\to \mathbb{G}_{m,T}\) is an isomorphism
  of group schemes over \(T\).
  This implies that \(m^{\#}(t) = t\otimes t\).
  Write \(\Delta:A\times_T A\times_T A\to A\times_T A\times_T A\times_T A\)
  for the morphism
  such that for any \(T'\)-valued points \(\alpha_1,\alpha_2,\alpha_3\in A(T')\),
  \(\Delta \circ (\alpha_1,\alpha_2,\alpha_3) = (\alpha_1,\alpha_2,\alpha_1,\alpha_3)\).
  Then, by the distribution rule, it holds that
  \(m\circ (\id\times a) = a\circ (m\times m)\circ \Delta\).
  Hence, if we write
  \(a^{\#}(t) = \sum_{i,j\geq 0} a_{ij}t^i\otimes t^j\), where \(a_{ij}\in k\),
  then it holds that
  \begin{align*}
    t\otimes a^{\#}(t) &= (\id\times a)^{\#}(m^{\#}(t))  \\
    &= \Delta^{\#}((m^{\#}\otimes m^{\#})(a^{\#}(t)))  \\
    &= \Delta^{\#}\left((m^{\#}\otimes m^{\#})\left(\sum_{i,j\geq 0} a_{ij}t^i\otimes t^j \right)\right)  \\
    &= \Delta^{\#}\left(\sum_{i,j\geq 0} a_{ij}t^i\otimes t^i \otimes t^j \otimes t^j \right)  \\
    &= \sum_{i,j\geq 0}a_{ij}t^{i+j}\otimes t^i\otimes t^j.
  \end{align*}
  This implies that
  \(a^{\#}(t) = t\otimes a_{10} + a_{01}\otimes t\).
  Since \(a\circ (\id_G \times 0_T) = a\circ (0_T\times \id_G) = \id_G\),
  it holds that \(a^{\#}(t) = t\otimes 1 + 1 \otimes t\).
  This implies that
  \(A\) is isomorphic as a ring scheme over \(T\) to \(\mathbb{A}^1_T\).


  In the general case,
  since \(T\) is reduced,
  \(\mathbb{A}^1_T\) is also reduced.
  Thus,
  by applying assertion
  \ref{enumi: any A^1 lem: ring str A^1 unique} in the case where \(T\) is isomorphic to the spectrum of a field to
  the generic geometric fibers of \(\mathbb{A}^1_T\to T\),
  we conclude that assertion
  \ref{enumi: any A^1 lem: ring str A^1 unique} holds.
  This completes the proof of \autoref{lem: ring str A^1 unique}.
\end{proof}

The following lemma was motivated by
\cite[Proposition 2.3]{WW}.

\begin{lem}\label{lem: A1 is abs min}
  Let \(V\) be the spectrum of a DVR.
  Write \(\eta\in V\) for the generic point.
  Let \(f:X\to V\) be a flat separated ring scheme of finite type over \(V\)
  such that the generic fiber \(X_{\eta}\) of \(f\) is
  isomorphic to the scheme \(\mathbb{A}^1_{k(\eta)}\)
  equipped with its natural ring scheme structure.
  Then the following assertions hold:
  \begin{enumerate}
    \item \label{enumi: lem: A1 is abs min affine}
    \(X\) is affine.
    \item \label{enumi: lem: A1 is abs min exists morph}
    There exists a unique morphism of ring schemes \(h:X \to \mathbb{A}^1_V\) over \(V\)
    such that \(h_\eta\) is an isomorphism of ring schemes over \(\Spec(k(\eta))\).
  \end{enumerate}
\end{lem}

\begin{proof}
  Since the generic fiber of \(f\) is affine,
  it follows from \cite[Proposition 2.3.1]{Anan} that
  \(X\) is affine, i.e.,
  assertion \ref{enumi: lem: A1 is abs min affine} holds.
  In the remainder of the proof of \autoref{lem: A1 is abs min},
  we prove assertion \ref{enumi: lem: A1 is abs min exists morph}.
  Write \(R\dfn \Gamma(V,\mathcal{O}_V)\),
  \(K\) for the field of fractions of \(R\), and
  \(A\dfn \Gamma(X,\mathcal{O}_X)\).
  Since \(A\) is flat over \(R\), and
  \(X_{\eta}\) is isomorphic as a ring scheme over \(K\) to \(\mathbb{A}^1_K\),
  we may regard \(A\) as a subring of \(K[x]\),
  where \(x\) is the standard coordinate of \(\mathbb{A}^1_K\).
  Thus, to prove assertion \ref{enumi: lem: A1 is abs min exists morph},
  it suffices to prove that \(x\in A\).

  Write
  \(a^{\#}:K[x]\to K[x]\otimes_K K[x]\)
  for the ring homomorphism that defines the \textit{additive} structure of
  the ring scheme \(X_{\eta}\),
  \(m^{\#}:K[x]\to K[x]\otimes_K K[x]\)
  for the ring homomorphism that defines the \textit{multiplicative} structure of
  the ring scheme \(X_{\eta}\),
  \(e_a^{\#}:K[x]\to K\)
  for the ring homomorphism that defines the
  \textit{identity section} for the \textit{additive} structure
  of the ring scheme \(X_{\eta}\), and
  \(e_m^{\#}:K[x]\to K\)
  for the ring homomorphism that defines the
  \textit{identity section} for the \textit{multiplicative} structure
  of the ring scheme \(X_{\eta}\).
  Thus the following equalities hold:
  \[
  a^{\#}(x) = x\otimes 1 + 1 \otimes x, \ \
  m^{\#}(x) = x\otimes x, \ \
  e_a^{\#}(x) = 0, \ \  e_m^{\#}(x) = 1.
  \]
  Since \(A\) is flat over \(R\),
  we may regard \(A\otimes_R A\) as a subring of \(K[x]\otimes_K K[x]\).
  Moreover, since \(X_{\eta}\) is isomorphic to \(\mathbb{A}^1_K\),
  it follows from \autoref{lem: ring str A^1 unique} that
  we may regard
  the ring scheme structure of \(X\) is given by
  the restrictions \(a^{\#}|_A,m^{\#}|_A,e_a^{\#}|_A,e_m^{\#}|_A\).
  Thus, in particular,
  the ring homomorphism \(m^{\#}|_A:A\to K[x]\otimes_K K[x]\)
  factors uniquely through the subring
  \(A\otimes_R A\subset K[x]\otimes_K K[x]\).
  To prove assertion \ref{enumi: lem: A1 is abs min exists morph},
  it suffices to prove that \(x\in A\).

  Write \(M \dfn (K\cdot 1 + K\cdot x) \cap A \subset A\).
  Since \(e_a^{\#}|_A\) defines the identity section
  for the additive structure,
  it holds that \(A\cap (K\cdot 1) = R\cdot 1\), and
  \[R\cdot 1 = M\cap (K\cdot 1) \subset M \subset R\cdot 1 + K\cdot x.\]
  Since \(e_m^{\#}|_A\) defines the identity section
  for the multiplicative structure,
  it holds that
  \(M\subset R\cdot 1 + R\cdot x\).
  Hence there exists a non-zero ideal \(I\subset R\) such that \(M = R\cdot 1 + I\cdot x\) (where we note that since \(K\otimes_R A \cong K[x]\), \(I\neq 0\)).
  Next, observe that \(M\) and \(A/M\) are \(R\)-submodules of \(K\)-vector spaces,
  hence, in particular, \(R\)-flat.
  Thus we obtain a commutative diagram of \(R\)-flat modules
  in which the horizontal and vertical sequences are exact:
  \[
  \begin{tikzpicture}[auto]
    \node (A1) at (2.2,3) {\(0\)};
    \node (A2) at (5.4,3) {\(0\)};
    \node (A3) at (8.7,3) {\(0\)};
    \node (A) at (0,2) {\(0\)};
    \node (A') at (0,1) {\(0\)};
    \node (A'') at (0,0) {\(0\)};
    \node (B) at (2.2,2) {\(M\otimes_RM\)};
    \node (B') at (2.2,1) {\(M\otimes_RA\)};
    \node (B'') at (2.2,0) {\(M\otimes_R(A/M)\)};
    \node (C) at (5.4,2) {\(A\otimes_RM\)};
    \node (C') at (5.4,1) {\(A\otimes_RA\)};
    \node (C'') at (5.4,0) {\(A\otimes_R(A/M)\)};
    \node (D) at (8.7,2) {\((A/M)\otimes_RM\)};
    \node (D') at (8.7,1) {\((A/M)\otimes_RA\)};
    \node (D'') at (8.7,0) {\((A/M)\otimes_R(A/M)\)};
    \node (E) at (11.2,2) {\(0\)};
    \node (E') at (11.2,1) {\(0\)};
    \node (E'') at (11.2,0) {\(0\)};
    \node (E1) at (2.2,-1) {\(0\)};
    \node (E2) at (5.4,-1) {\(0\)};
    \node (E3) at (8.7,-1) {\(0\)};
    \draw[->] (A1) to (B);
    \draw[->] (A2) to (C);
    \draw[->] (A3) to (D);
    \draw[->] (B'') to (E1);
    \draw[->] (C'') to (E2);
    \draw[->] (D'') to (E3);
    \draw[->] (B) to (B');
    \draw[->] (B') to (B'');
    \draw[->] (C) to (C');
    \draw[->] (C') -- node[swap]  {\(\scriptstyle p\)} (C'');
    \draw[->] (D) to (D');
    \draw[->] (D') to (D'');
    \draw[->] (A) to (B);
    \draw[->] (B) to (C);
    \draw[->] (C) to (D);
    \draw[->] (D) to (E);
    \draw[->] (A') to (B');
    \draw[->] (B') to (C');
    \draw[->] (C') -- node  {\(\scriptstyle q\)} (D');
    \draw[->] (D') to (E');
    \draw[->] (A'') to (B'');
    \draw[->] (B'') to (C'');
    \draw[->] (C'') to (D'');
    \draw[->] (D'') to (E'');
  \end{tikzpicture}
  \]
  This implies that
  \(M\otimes_R M = (A\otimes_R M)\cap (M\otimes_R A)\)
  as \(R\)-submodules of \(A\otimes_RA \subset K[x]\otimes_KK[x]\).
  Let \(t\in M\) be an element.
  Then, since \(M = R\cdot 1 + I\cdot x\),
  there exist elements \(r\in R\) and \(s\in I\) such that \(t = r+sx\).
  Moreover, it holds that
  \[
  A\otimes_R A \ni m^{\#}(t) = m^{\#}(r+sx) = r+s\cdot (x\otimes x)
  = r + (sx) \otimes x = r + x\otimes (sx).
  \]
  Hence it holds that
  \((p\otimes_R \id_K)(m^{\#}(r+sx)) = 0\), and
  \((q\otimes_R \id_K)(m^{\#}(r+sx)) = 0\).
  Thus \(m^{\#}|_M : M \to K[x]\otimes_K K[x]\) factors through
  \((A\otimes_R M)\cap (M\otimes_R A) = M\otimes_RM\).
  Since \(M = R\cdot 1 + I\cdot x\),
  we conclude that \(I\cdot (x\otimes x)\subset I^2\cdot (x\otimes x)\).
  This implies that \(I\subset I^2\).
  Hence it holds that \(I = R\).
  Thus it holds that \(x\in M \subset A\).
  This completes the proof of \autoref{lem: A1 is abs min}.
\end{proof}

\begin{prop}\label{lem: A1 over DVR}
  Let \(V\) be the spectrum of a DVR and
  \(f:X\to V\) a flat separated ring scheme of finite type over \(V\).
  Then \(f\) is isomorphic as a ring scheme over \(V\)
  to the projection \(\mathbb{A}^1_V\to V\)
  if and only if the following conditions hold:
  \begin{enumerate}
    \item \label{enumi: A1 connected fiber}
    Each fiber of \(f\) is one-dimensional and connected.
    \item \label{enumi: A1 gen fiber}
    The generic fiber of \(f\) is isomorphic to \(\mathbb{A}^1_K\)
    as a scheme over \(K\). 
  \end{enumerate}
\end{prop}

\begin{proof}
  Necessity follows immediately.
  In the remainder of the proof of \autoref{lem: A1 over DVR},
  we prove sufficiency.
  Assume that \(f:X\to V\) satisfies conditions
  \ref{enumi: A1 connected fiber} and \ref{enumi: A1 gen fiber}.
  Write \(v\in V\) for the closed point,
  \(\eta\in V\) for the generic point,
  \(k\dfn k(v)\), and \(K\dfn k(\eta)\).
  By \autoref{lem: ring str A^1 unique} \ref{enumi: any A^1 lem: ring str A^1 unique},
  we may assume without loss of generality that
  the generic fiber of \(f\) is isomorphic as a ring scheme over \(K\) to \(\mathbb{A}^1_K\).

  By \autoref{lem: A1 is abs min} \ref{enumi: lem: A1 is abs min exists morph},
  there exists a unique morphism of ring schemes
  \(h:X\to \mathbb{A}^1_V\) over \(V\)
  such that \(h_{\eta}\) is an isomorphism of ring schemes over \(\Spec(K)\).
  Then \(h_v\) is a morphism of ring schemes over \(\Spec(k)\).
  Since \(f\) satisfies condition \ref{enumi: A1 connected fiber},
  the scheme-theoretic image of \(h_v\) is
  a connected closed subscheme of \(\mathbb{A}^1_k\).
  Moreover, since \(h_v\) is a morphism of ring schemes over \(\Spec(k)\),
  \(0,1\in \mathbb{A}^1_k\) are contained in the image of \(h_v\).
  Hence the scheme theoretic image of \(h_v\) is \(\mathbb{A}^1_k\).
  Thus, by \cite[Corollary 2.3.3.3]{Anan}, \(h_v\) is faithfully flat.
  Since \(h_{\eta}\) is an isomorphism,
  it follows from
  \cite[\href{https://stacks.math.columbia.edu/tag/039D}{Tag 039D}]{stacks-project}
  that \(h\) is faithfully flat.
  Since \(X\) is separated of finite type over \(V\), and
  \(\mathbb{A}^1_V\) is quasi-separated over \(V\),
  it follows from
  \cite[\href{https://stacks.math.columbia.edu/tag/01KV}{Tag 01KV}]{stacks-project},
  \cite[\href{https://stacks.math.columbia.edu/tag/03GI}{Tag 03GI}]{stacks-project}, and
  \cite[\href{https://stacks.math.columbia.edu/tag/01T8}{Tag 01T8}]{stacks-project}
  that \(h\) is separated and of finite type.
  By \autoref{lem: A1 is abs min} \ref{enumi: lem: A1 is abs min affine},
  \(X\) is affine.
  Write \(N\dfn X\times_{h,\mathbb{A}^1_V,0_{\mathbb{A}^1_V}}V\).
  Since the morphism of affine ring schemes \(h\)
  is faithfully flat, separated, and of finite type,
  \(N\) is a flat separated commutative affine group scheme
  of finite type over \(V\).
  Since the natural morphism
  \(N_{\eta} \xrightarrow{\sim} \Spec(K)\)
  is an isomorphism, and
  \(N\) is flat over \(V\),
  it holds that
  \(\Gamma(N,\mathcal{O}_N) \subset K\).
  Since \(V\) is the spectrum of a DVR,
  and the \(V\)-scheme \(N\) has a section \(V\to N\),
  it holds that \(N \xrightarrow{\sim} V\).
  Since \(h\) is a morphism of ring schemes over \(V\), this implies that
  \(h\) is a monomorphism in \(\SchU\).
  Since \(h\) is faithfully flat and of finite type,
  it follows from
  \cite[\href{https://stacks.math.columbia.edu/tag/025G}{Tag 025G}]{stacks-project}
  that \(h\) is an open immersion.
  Thus \(h\) is an isomorphism.
  This completes the proof of \autoref{lem: A1 over DVR}.
\end{proof}

\begin{rem}
  Let \(R\) be a mixed characteristic DVR and
  \(\pi\in R\) a uniformizer of \(R\).
  Write \(k\dfn R/(\pi)\), and
  \(p\) for the characteristic of \(k\).
  Then the \(R\)-scheme
  \(X\dfn \Spec(R[x,y]/(\pi y-x^{p^2}+x^p)) = \Spec(R[x,(x^{p^2}-x^p)/\pi])\)
  has the following properties:
  \begin{itemize}
    \item
    \(X\) is a flat affine ring scheme of finite type over \(R\).
    \item
    The generic fiber of \(X\to \Spec(R)\) is the ring scheme \(\mathbb{A}^1_K\).
    \item
    The special fiber of \(X\to \Spec(R)\) is not connected.
  \end{itemize}
  Hence, in \autoref{lem: A1 over DVR},
  if one do not assume the connectedness of the special fiber,
  then it may not hold that \(X\cong\mathbb{A}^1_R\).
  This example is obtained by forming ``N\'{e}ron's blow-up''
  (cf. \cite[\href{https://stacks.math.columbia.edu/tag/0BJ1}{Tag 0BJ1}]{stacks-project})
  of \(\mathbb{A}^1_R\) along
  the closed ring subscheme \(\Spec(k[x]/(x^{p^2}-x^p))\) over \(k\)
  of the special fiber \(\mathbb{A}^1_k\).
\end{rem}

\begin{lem}\label{lem: geom red}
  Let \(V\) be an affine scheme;
  \(X,Y,Z\) integral schemes;
  \(f:X\to V\), \(g:Y\to V\), and \(h:Z\to V\) morphisms.
  Write \(R\dfn \Gamma(V,\mathcal{O}_V)\).
  Assume that \(R\) is either a DVR or a field, and
  that the generic fiber of \(f,g\), and \(h\) are geometrically reduced and non-empty.
  Then \(X\times_VY\times_VZ\) is reduced.
\end{lem}

\begin{proof}
  Write \(\eta\) for the generic point of \(V\).
  Since \(X_{\eta},Y_{\eta},Z_{\eta}\) are geometrically reduced over \(\Spec(k(\eta))\),
  it follows from
  \cite[\href{https://stacks.math.columbia.edu/tag/035Z}{Tag 035Z}]{stacks-project}
  that \((X\times_VY\times_VZ)_{\eta} \cong
  X_{\eta}\times_{\Spec(k(\eta))}Y_{\eta}\times_{\Spec(k(\eta))}Z_{\eta}\)
  is reduced.
  Since \(X,Y,Z\) are integral, and
  the generic fibers of \(f,g,h\) are non-empty,
  it follows from \cite[Chapter III, Proposition 9.7]{Ha} that
  \(f,g,h\) are flat.
  Hence \(X\times_V Y\times_V Z\) is flat over \(V\).
  Since the scheme-theoretic image of the natural morphism
  \(\Spec(k(\eta)) \to V\) is \(V\),
  it follows from
  \cite[\href{https://stacks.math.columbia.edu/tag/081I}{Tag 081I}]{stacks-project}
  that the scheme theoretic image of
  \((X \times_V Y \times_V Z)_{\eta} \to X\times_V Y \times_V Z\) is equal to
  \(X \times_V Y \times_V Z\).
  Since \((X\times_V Y \times_V Z)_{\eta}\) is reduced,
  \(X\times_V Y \times_V Z\) is also reduced.
  This completes the proof of \autoref{lem: geom red}.
\end{proof}

\begin{lem}\label{cor: group over DVR}
  Let \(S\) be a quasi-separated scheme.
  Let \(V\) be an object of \(\Schb{S}\) which is affine,
  \(X\) an object of \(\Schb{S}\) which is
  quasi-compact (over \(\mathbb{Z}\)) and integral, and
  \(f:X\to V\) a morphism in \(\Schb{S}\).
  Write \(R\dfn \Gamma(V,\mathcal{O}_V)\).
  Assume that \(R\) is either a DVR or a field, and
  that the generic fiber of \(f\) is geometrically reduced.
  Then there is a natural bijective correspondence between
  the group scheme structures over \(V\) on \(X\) and
  the group object structures on the object \(f:X\to V\) of \((\Schb{S})_{/V}\).
  Similarly, there is a natural bijective correspondence between
  the ring scheme structures over \(V\) on \(X\) and
  the ring object structures on the object \(f:X\to V\) of \((\Schb{S})_{/V}\).
\end{lem}

\begin{proof}
  \autoref{cor: group over DVR} follows immediately from \autoref{lem: geom red},
  together with \autoref{lem: fiber product bbullet}
  \ref{enumi: fiber product bbullet not red}
  \ref{enumi: fiber product bbullet red} (where we note that if the generic fiber of \(f\) is empty, there does not exist a section of \(f\), hence \autoref{cor: group over DVR} follows immediately).
\end{proof}


Next, to give a category-theoretic characterization of morphisms
which are isomorphic to the natural projection \(\mathbb{P}^1_S\to S\),
we study certain algebraic spaces that parametrize closed immersions of
a certain type.

\begin{defi}\label{defi: Cl f}
  Let \(T\) be a scheme,
  \(X\) a proper flat \(T\)-scheme of finite presentation,
  \(Y\) a \(T\)-scheme of finite presentation,
  \(B\subset X\) a closed subscheme which is flat over \(T\),
  \(f:B\to Y\) a morphism of \(T\)-schemes, and
  \(U\) a \(T\)-scheme.
  We define the following sets:
  \begin{align*}
    \mathrm{Cl}(X,Y)(U) &\dfn
    \left\{ i:X_U\to Y_U \mid \text{\(i\) is a closed immersion} \right\}, \\
    \mathrm{Cl}_f(X,Y)(U) &\dfn
    \left\{ i:X_U\to Y_U \mid \text{\(i\) is a closed immersion such that \(i|_{B_U} = f_U\)} \right\}.
  \end{align*}
  Since base-change preserves the required property,
  we obtain functors
  \begin{align*}
    \mathrm{Cl}(X,Y) &:\Sch{T}^{\mathrm{op}}\to \mathsf{Set}, \\
    \mathrm{Cl}_f(X,Y) &:\Sch{T}^{\mathrm{op}}\to \mathsf{Set}.
  \end{align*}
\end{defi}

\begin{lem}\label{lem: representable by AS}
  Let \(T\) be a scheme,
  \(X\) a proper flat \(T\)-scheme of finite presentation,
  \(Y\) a separated \(T\)-scheme of finite presentation,
  \(B\) a closed subscheme of \(X\)
  which is flat and of finite presentation over \(T\), and
  \(f:B\to Y\) a closed immersion of \(T\)-schemes.
  Then the functors \(\mathrm{Cl}(X,Y)\) and \(\mathrm{Cl}_f(X,Y)\) are
  represented by algebraic spaces which are separated and
  locally of finite presentation over \(T\).
\end{lem}

\begin{proof}
  First, we prove that the functor \(\mathrm{Cl}(X,Y)\) is
  represented by an algebraic space which is separated and
  locally of finite presentation over \(T\).
  Write
  \[\mathrm{Hilb}_{Y/T}:\Sch{T}^{\mathrm{op}}\to \mathsf{Set}\]
  for the Hilbert functor
  (cf. \cite[\href{https://stacks.math.columbia.edu/tag/0CZY}{Tag 0CZY}]{stacks-project}),
  i.e., parametrizing closed subschemes of \(Y\)
  that are proper, flat, and of finite presentation over \(T\).
  Since \(Y\) is separated and of finite presentation over \(T\),
  it follows from
  \cite[\href{https://stacks.math.columbia.edu/tag/0D01}{Tag 0D01}]{stacks-project} and
  \cite[\href{https://stacks.math.columbia.edu/tag/0DM7}{Tag 0DM7}]{stacks-project} that
  \(\mathrm{Hilb}_{Y/T}\) is represented by an algebraic space
  which is separated and locally of finite presentation over \(T\).
  Let \(U_0\) be a \(T\)-scheme and
  \(i_{U_0}:X_{U_0}\to Y_{U_0}\) a closed immersion.
  By forming the scheme-theoretic image of \(X_{U_0}\) in \(Y_{U_0}\) via \(i_{U_0}\),
  we obtain a map
  \(\varphi(U_0) : \mathrm{Cl}(X,Y)(U_0) \to \mathrm{Hilb}_{Y/T}(U_0)\),
  which is clearly functorial with respect to \(U_0\).
  Hence we obtain a morphism of functors
  \(\varphi: \mathrm{Cl}(X,Y) \to \mathrm{Hilb}_{Y/T}\).
  Thus, by
  \cite[\href{https://stacks.math.columbia.edu/tag/02YS}{Tag 02YS}]{stacks-project},
  \cite[\href{https://stacks.math.columbia.edu/tag/03XQ}{Tag 03XQ}]{stacks-project}, and
  \cite[\href{https://stacks.math.columbia.edu/tag/03KQ}{Tag 03KQ}]{stacks-project},
  to prove that the functor \(\mathrm{Cl}(X,Y)\) is represented by
  an algebraic space which is separated and locally of finite presentation over \(T\),
  it suffices to prove that
  \(\varphi\) is separated, locally of finite presentation, and represented by an algebraic space (cf. \cite[\href{https://stacks.math.columbia.edu/tag/02YQ}{Tag 02YQ}]{stacks-project}).

  Let \(U_1\) be a \(T\)-scheme and
  \(U_1\to \mathrm{Hilb}_{Y/T}\) a morphism of algebraic spaces over \(T\).
  Then it follows from Yoneda's lemma that
  the morphism \(U_1\to \mathrm{Hilb}_{Y/T}\) corresponds to
  a closed subscheme \(Z\subset Y_{U_1}\) such that
  the composite \(Z\subset Y_{U_1}\to U_1\) is
  proper, flat, and of finite presentation.
  Hence it holds that
  \(U_1\times_{\mathrm{Hilb}_{Y/T},\varphi} \mathrm{Cl}(X,Y) \cong
  \mathrm{Isom}_{U_1}(X_{U_1},Z)\),
  where \(\mathrm{Isom}_{U_1}(X_{U_1},Z)\subset \mathrm{Mor}_{U_1}(X_{U_1},Z)\)
  is the open sub-algebraic space of the algebraic space \(\mathrm{Mor}_{U_1}(X_{U_1},Z)\)
  determined by the property of being an isomorphism
  (cf. \cite[\href{https://stacks.math.columbia.edu/tag/0D1C}{Tag 0D1C}]{stacks-project},
  \cite[\href{https://stacks.math.columbia.edu/tag/0DPP}{Tag 0DPP}]{stacks-project}).
  By \cite[\href{https://stacks.math.columbia.edu/tag/0DPN}{Tag 0DPN}]{stacks-project} and
  \cite[\href{https://stacks.math.columbia.edu/tag/0DPP}{Tag 0DPP}]{stacks-project},
  \(\mathrm{Isom}_{U_1}(X_{U_1},Z)\) is represented by an algebraic space
  which is separated and locally of finite presentation over \(U_1\).
  Thus \(\varphi\) is separated, locally of finite presentation, and
  represented by an algebraic space.
  This completes the proof of the assertion that the functor \(\mathrm{Cl}(X,Y)\) is
  represented by an algebraic space which is separated and
  locally of finite presentation over \(T\).

  Next, we prove that the functor \(\mathrm{Cl}_f(X,Y)\) is represented by
  an algebraic space which is separated and locally of finite presentation over \(T\).
  Let \(U_2\) be a \(T\)-scheme and
  \(i_{U_2}:X_{U_2}\to Y_{U_2}\) a closed immersion.
  By composing \(i_{U_2}\) with
  the closed immersion \(B_{U_2}\hookrightarrow X_{U_2}\),
  we obtain a morphism of algebraic spaces
  \(\psi : \mathrm{Cl}(X,Y) \to \mathrm{Cl}(B,Y)\) over \(T\).
  By \cite[\href{https://stacks.math.columbia.edu/tag/05WT}{Tag 05WT}]{stacks-project}
  and \cite[\href{https://stacks.math.columbia.edu/tag/04ZI}{Tag 04ZI}]{stacks-project},
  \(\psi\) is separated and locally of finite presentation.
  Write \(\tilde{f}:T\to \mathrm{Cl}(B,Y)\) for the morphism
  corresponding to the element \([f:B\to Y] \in \mathrm{Cl}(B,Y)(T)\).
  Then it holds that
  \(\mathrm{Cl}_f(X,Y) \cong \mathrm{Cl}(X,Y) \times_{\mathrm{Cl}(B,Y),\tilde{f}} T\).
  This implies that \(\mathrm{Cl}_f(X,Y)\) is represented by
  an algebraic space which is separated and locally of finite presentation over \(T\).
  This completes the proof of \autoref{lem: representable by AS}.
\end{proof}

\begin{lem}\label{lem: DVR exists}
  Let \(X\) be a stalkwise Noetherian scheme and
  \(x, \xi\in X\) points such that \(\xi \rsa x\).
  Then there exist a scheme \(V\) and
  a morphism \(f:V\to X\) such that the following conditions hold:
  \begin{enumerate}
    \item \(V\) is isomorphic to the spectrum of a discrete valuation ring.
    \item \(f(v) = x\), where \(v\in V\) is the unique closed point.
    \item \(f(\eta) = \xi\), where \(\eta\in V\) is the unique generic point.
    \item The natural morphism \(k(\xi) \to k(\eta)\) is an isomorphism.
  \end{enumerate}
\end{lem}

\begin{proof}
  Write \(\mathfrak{p}\dfn \ker(\mathcal{O}_{X,x}\to \mathcal{O}_{X,\xi})\) and
  \(A\dfn \mathcal{O}_{X,x}/\mathfrak{p}\).
  Then \(A\) is a subring of \(k(\xi)\).
  Since \(A\) is Noetherian, it follows from \cite[\href{https://stacks.math.columbia.edu/tag/00PH}{Tag 00PH}]{stacks-project} that
  there exists a DVR \(R\subset k(\xi)\) which dominates \(A\).
  Then it follows immediately that
  \(V \dfn \Spec(R)\) and
  the composite
  \[f: \Spec(R) \to \Spec(A) \to X\]
  satisfy required properties.
  This completes the proof of \autoref{lem: DVR exists}.
\end{proof}

In the following \autoref{lem: val criteria st Noeth}, we show that if a base scheme is stalkwise Noetherian, then the valuative criterion of properness holds using only discrete valuation rings.

\begin{lem}\label{lem: val criteria st Noeth}
  Let \(Y\) be a stalkwise Noetherian scheme and
  \(f:X\to Y\) a separated algebraic space of finite presentation over \(Y\).
  Assume that
  for any commutative diagram
  \[
  \begin{CD}
    \Spec(K) @> \eta >> X \\
    @V i VV @VV f V \\
    \Spec(R) @> p >> Y
  \end{CD}
  \]
  such that \(R\) is a DVR, \(K\) is the field of fractions of \(R\), and
  \(i\) is the natural morphism,
  there exists a morphism \(q:\Spec(R) \to X\) such that
  \(\eta = q\circ i\), and \(p = f\circ q\).
  Then \(f\) is proper.
\end{lem}

\begin{proof}
  Let
  \[
  \begin{CD}
    \Spec(K) @> \eta >> X \\
    @V i VV @VV f V \\
    \Spec(R) @> p >> Y
  \end{CD}
  \]
  be a commutative diagram
  such that \(R\) is a (not necessary discrete!) valuation ring,
  \(K\) is the field of fractions of \(R\), and
  \(i\) is the natural morphism.
  Write \(r\in \Spec(R)\) for the unique closed point,
  \(y\dfn p(r)\),
  \(B\dfn \mathcal{O}_{Y,y}\),
  \(j:\Spec(B) \to Y\) for the natural morphism,
  \(X_B\dfn X\times_{f,Y,j}B\),
  \(j_X:X_B \to X\) for the natural projection,
  \(p':\Spec(R)\to \Spec(B)\) for the local morphism induced by
  the local homomorphism \(B\to R\), and
  \(\eta':\Spec(K) \to X_B\) for the unique morphism such that
  \(j_X\circ \eta' = \eta\) and \(f_B\circ \eta' = p'\circ i\):
  \[
  \begin{CD}
    \Spec(K) @> \eta' >> X_B @>{j_X}>> X \\
    @V i VV @VV f_B V @VV f V \\
    \Spec(R) @> p' >> \Spec(B) @>{j}>> Y.
  \end{CD}
  \]
  Since \(Y\) is stalkwise Noetherian,
  \(B\) is Noetherian.
  Hence we conclude, by considering the case where \(R\) is a DVR, from
  \cite[\href{https://stacks.math.columbia.edu/tag/0CMF}{Tag 0CMF}]{stacks-project}
  that \(f_B\) is proper.
  This, in turn, implies,
  in the case where \(R\) is \textit{arbitrary} (i.e., not necessarily discrete),
  by \cite[\href{https://stacks.math.columbia.edu/tag/0A40}{Tag 0A40}]{stacks-project},
  that there exists a morphism \(q':\Spec(R) \to X_B\) such that
  \(\eta' = q'\circ i\), and \(p' = f_B\circ q'\).
  Thus, by
  \cite[\href{https://stacks.math.columbia.edu/tag/0A40}{Tag 0A40}]{stacks-project},
  \(f\) is proper.
  This completes the proof of \autoref{lem: val criteria st Noeth}.
\end{proof}


\begin{lem}\label{lem: homeo Noeth}
  Let \(S\) be a locally Noetherian scheme and
  \(f:X\to S\) an algebraic space over \(S\).
  Assume that the following conditions hold:
  \begin{enumerate}
    \item \label{enumi: lfp lem: homeo Noeth}
    \(f\) is locally of finite presentation.
    \item \label{enumi: val lem: homeo Noeth}
    For any morphism \(g:\Spec(R) \to S\), where \(R\) is a DVR or a field,
    there exists a unique morphism \(h:\Spec(R) \to X\)
    such that \(g = f\circ h\).
  \end{enumerate}
  Then the continuous map between underlying topological spaces
  \(|f|:|X|\to |S|\) is a homeomorphism
  (cf. \cite[\href{https://stacks.math.columbia.edu/tag/03BY}{Tag 03BY}]{stacks-project},
  \cite[\href{https://stacks.math.columbia.edu/tag/03BX}{Tag 03BX}]{stacks-project}).
  In particular,
  \(f\) is quasi-compact
  (cf. \cite[\href{https://stacks.math.columbia.edu/tag/03E4}{Tag 03E4}]{stacks-project}).
\end{lem}

\begin{proof}
  To prove \autoref{lem: homeo Noeth},
  we may assume without loss of generality that
  \(S\) is Noetherian.
  By condition \ref{enumi: val lem: homeo Noeth},
  the map between underlying sets \(|f|:|X|\to |S|\) is bijective.
  Hence to prove \autoref{lem: homeo Noeth},
  it suffices to prove that
  for any \'{e}tale morphism \(p:U\to X\) such that \(U\) is a quasi-compact scheme,
  \(\im(|f\circ p|)\subset |S|\) is an open subset.

  Let \(p:U\to X\) be an \'{e}tale morphism
  such that \(U\) is a quasi-compact scheme.
  Then, by condition \ref{enumi: lfp lem: homeo Noeth},
  \(f\circ p\) is of finite presentation.
  Hence, by \cite[\href{https://stacks.math.columbia.edu/tag/054J}{Tag 054J}]{stacks-project},
  \(\im(|f\circ p|)\subset |S|\) is a constructible subset.
  Thus, by \cite[\href{https://stacks.math.columbia.edu/tag/0542}{Tag 0542}]{stacks-project},
  to prove that \(\im(|f\circ p|)\subset |S|\) is open,
  it suffices to prove that
  \(\im(|f\circ p|)\subset |S|\) is stable under generization.

  Let \(u\in U\) be a point and
  \(\eta \in S\) a point such that
  \(\eta \rsa f(p(u))\).
  Then, by \autoref{lem: DVR exists},
  there exists
  a morphism \(g:\Spec(R) \to S\) such that
  \(R\) is a DVR, and \(\im(|g|) = \{\eta,f(p(u))\}\).
  By condition \ref{enumi: val lem: homeo Noeth},
  there exists a unique morphism \(h:\Spec(R) \to X\) such that
  \(g = f\circ h\).
  Write \(v\in \Spec(R)\) for the closed point,
  \(\xi\in \Spec(R)\) for the generic point,
  \(U'\) for the base-change of \(p:U\to X\) by \(h:\Spec(R) \to X\),
  \(q:U'\to \Spec(R)\) for the natural projection morphism, and
  \(h':U'\to U\) for the natural projection morphism:
  \[
  \begin{CD}
    U' @>{h'}>> U \\
    @V{q}VV @VV{p}V \\
    \Spec(R) @>{h}>> X.
  \end{CD}
  \]
  Since \(g = f\circ h\), and
  \(|f|:|X| \to |S|\) is a bijection,
  it holds that \(h(v) = p(u)\), hence that \(U'\neq \emptyset\).
  Since \(p:U\to X\) is \'{e}tale,
  \(q:U'\to \Spec(R)\) is \'{e}tale.
  Hence there exists a point \(\xi'\in U'\) such that
  \(q(\xi') = \xi\).
  Thus it holds that
  \[
  \eta = g(\xi) = f(h(q(\xi'))) = f(p(h'(\xi'))) \in \im(|f\circ p|).
  \]
  This implies that \(\im(|f\circ p|)\subset |S|\) is stable under generization,
  i.e., \(\im(|f\circ p|)\subset |S|\) is open.
  This completes the proof of \autoref{lem: homeo Noeth}.
\end{proof}


\begin{defi}\label{defi: P1 like}
  Let \(S\) be a quasi-separated stalkwise Noetherian scheme;
  \(f:X\to S\) an object of \(\Schb{S}\);
  \(s_0,s_1,s_\infty:S\to X\) morphisms in \(\Schb{S}\).
  Suppose that \(\bbbullet = \bbullet \cup \left\{ \red\right\}\).
  Then we shall say that the collection of data \((X,s_0,s_1,s_\infty)\) is
  \textit{\(\mathbb{P}^1\)-like} in \(\Schb{S}\)
  if the following conditions hold:
  \begin{enumerate}
    \item \label{enumi: P1-like is red}
    \(X\) is reduced.
    \item \label{enumi: P1-like Noether proper}
    \(f\) is proper and of finite presentation.
    \item \label{enumi: P1-like Noether connected}
    For any object \(T\to S\) in \(\Schb{S}\)
    such that \(T\) is isomorphic to the spectrum of a field,
    \(X\times_S^{\bbullet} T\) is isomorphic as a \(T\)-scheme to \(\mathbb{P}^1_T\).
    \item \label{enumi: P1-like Noether section}
    For any \(i,j\in \{0,1,\infty\}\) such that \(i\neq j\),
    it holds that \(S\times_{s_i,X,s_j}^\bbullet S = \emptyset\).
    \item \label{enumi: P1-like Noether ring}
    For any \(i,j,k\in \{0,1,\infty\}\) such that
    \(\left\{i,j,k\right\} = \left\{0,1,\infty\right\}\),
    there exist an open immersion \(\iota: U\hookrightarrow X\) in \(\Schb{S}\),
    morphisms \(t_j: S\to U\), \(t_k:S\to U\) in \(\Schb{S}\), and
    a ring object structure on \(U\) in \(\Schc{S}\)
    such that
    \(t_j\) is the additive identity section,
    \(t_k\) is the multiplicative identity section,
    \(s_j = \iota\circ t_j\),
    \(s_k = \iota\circ t_k\), and
    \(\im(\mathsf{Sp}_{\bbullet/S}(\iota)) =
    \mathsf{Sp}_{\bbullet/S}(X) \setminus \im(\mathsf{Sp}_{\bbullet/S}(s_i))\)
    (cf. \autoref{reconstruction: Top}).
  \end{enumerate}
  Thus the property that
  the collection of data \((X,s_0,s_1,s_\infty)\) is \(\mathbb{P}^1\)-like in \(\Schb{S}\)
  is defined completely in terms of properties that
  may be characterized category-theoretically
  (cf. \autoref{lem: Spec field},
  \autoref{cor: red is cat},
  \autoref{cor: open imm is cat},
  \autoref{reconstruction: Top},
  \autoref{cor: proper is cat},
  \autoref{lem: P1 over field})
  from the data \((\Schb{S},X,s_0,s_1,s_\infty)\).
\end{defi}

%

\begin{prop}\label{prop: P1 over Noether}
  Let \(S\) be a locally Noetherian
  (hence quasi-separated ---
  cf. \cite[\href{https://stacks.math.columbia.edu/tag/01OY}{Tag 01OY}]{stacks-project})
  normal scheme. 
  Let \(X\) be an object of \(\Schb{S}\).
  Then \(X\) is isomorphic to \(\mathbb{P}^1_S\) as an \(S\)-scheme
  if and only if the following conditions hold:
  \begin{enumerate}
    \item \label{enumi: P^1 is P^1 like}
    There exist morphisms \(s_0,s_1,s_\infty:S\to X\) in \(\Schb{S}\)
    such that
    the collection of data \((X,s_0,s_1,s_\infty)\)
    is \(\mathbb{P}^1\)-like in \(\Schb{S}\).
    \item \label{enumi: P^1 univ}
    For any \(\mathbb{P}^1\)-like collection of data
    \((Y,t_0,t_1,t_\infty)\) in \(\Schb{S}\),
    there exists a unique closed immersion \(h:X\to Y\) in \(\Schb{S}\)
    such that for each \(i\in \{0,1,\infty\}\), \(h\circ s_i = t_i\).
  \end{enumerate}
  In particular,
  the property that \(X\) is isomorphic to \(\mathbb{P}^1_S\) as an \(S\)-scheme
  may be characterized category-theoretically
  (cf. \autoref{prop: closed imm is cat}, \autoref{defi: P1 like})
  from the data \((\Schb{S},X)\).
\end{prop}

\begin{proof}
  First, we prove necessity.
  Assume that \(X\cong \mathbb{P}^1_S\).
  For any \(i \in \{0,1,\infty\}\),
  write \(s_i:S\to X\) for the morphism in \(\Schb{S}\)
  obtained by base-changing the sections
  \(0,1,\infty: \Spec(\mathbb{Z}) \to \mathbb{P}^1_{\mathbb{Z}}\).
  Then it follows immediately from
  \autoref{lem: fiber product bbullet}
  \ref{enumi: fiber product bbullet not red} \ref{enumi: fiber product bbullet red}
  that the collection of data
  \((X,s_0,s_1,s_\infty)\) satisfies
  conditions
  \ref{enumi: P1-like is red},
  \ref{enumi: P1-like Noether proper},
  \ref{enumi: P1-like Noether connected}, and
  \ref{enumi: P1-like Noether section} of \autoref{defi: P1 like}.
  Moreover, since for any \(i\in \{0,1,\infty\}\),
  \(\mathbb{A}^1_S \cong X\setminus s_i\),
  it follows immediately that
  the collection of data \((X,s_0,s_1,s_\infty)\) satisfies
  condition \ref{enumi: P1-like Noether ring} of \autoref{defi: P1 like}.
  Hence \(X\) satisfies condition
  \ref{enumi: P^1 is P^1 like} of \autoref{prop: P1 over Noether}.

  Next, we verify that \(X\) satisfies condition
  \ref{enumi: P^1 univ} of \autoref{prop: P1 over Noether}.
  Let \((Y, t_0,t_1,t_\infty)\) be
  a \(\mathbb{P}^1\)-like collection of data in \(\Schb{S}\).
  Write
  \begin{itemize}
    \item \(B\dfn S\coprod S\coprod S\),
    \item
    \(s \dfn s_0\coprod s_1\coprod s_\infty: B \to X\)
    for the closed immersion,
    \item
    \(t \dfn t_0\coprod t_1\coprod t_\infty: B \to Y\)
    for the closed immersion,
    \item \(C\dfn \mathrm{Cl}_t(X,Y)\),
    where we regard \(B\) as a closed subscheme of \(X\) by means of \(s\), and
    we take ``\(T\)'' to be \(S\)
    (cf. \autoref{defi: Cl f}, \autoref{lem: representable by AS}), and
    \item \(p:C\to S\) for the structure morphism.
  \end{itemize}
  Then, to verify that \(X\) satisfies condition
  \ref{enumi: P^1 univ} of \autoref{prop: P1 over Noether},
  it suffices to prove that
  the set of \(S\)-valued points \(C(S)\) is of cardinality \(1\).
  Since \(S\) is normal, the injection \(C_{\red}(S) \hookrightarrow C(S)\) is surjective.
  Hence, in particular,
  to prove that \(C(S)\) is of cardinality \(1\),
  it suffices to prove that
  the composite \(C_{\red} \hookrightarrow C \xrightarrow{p} S\) is an isomorphism.
  By \autoref{lem: representable by AS},
  \(p\) is a morphism of algebraic spaces
  which is separated and locally of finite presentation.
  Let \(T\to S\) be a morphism such that
  \(T\) is isomorphic to the spectrum of a field.
  Then, by \autoref{lem: affine over qsep is qcpt}
  \ref{enumi: lem: aff over qsep is qcpt qc} \ref{enumi: lem: aff over qsep is qcpt sep},
  \(T\to S\) belongs to \(\Schb{S}\).
  Moreover, by \autoref{lem: fiber product bbullet}
  \ref{enumi: fiber product bbullet not red} \ref{enumi: fiber product bbullet red} and
  \autoref{defi: P1 like} \ref{enumi: P1-like Noether connected},
  \(C(T)\) is of cardinality \(1\).
  In particular, by allowing \(T\) to vary,
  we conclude that \(\dim C_T = 0\),
  where we write \(C_T\dfn C\times_S T\).
  Since the natural projection \(p_T: C_T\to T\) is locally of finite presentation,
  it follows from
  \cite[\href{https://stacks.math.columbia.edu/tag/06LZ}{Tag 06LZ}]{stacks-project}
  that \(C_T\) is a scheme.
  Hence, since \(C(T)\) is of cardinality \(1\) for arbitrary \(T\),
  we conclude that for arbitrary \(T\),
  the composite \(C_{T,\red} \to C_T \xrightarrow{p_T} T\) is an isomorphism.

  Next, we prove the following assertion:
  \begin{enumerate}[label=(\fnsymbol*),start=2]
    \item \label{enumi: prop: P1 over Noether DVR valued point}
    For any morphism \(\Spec(R) \to S\) such that \(R\) is a DVR, 
    the set of \(R\)-valued points \(C(R)\) is of cardinality \(1\). 
  \end{enumerate}
  (Here, we observe that it follows formally,
  by considering the ring of formal power series in one variable over a field,
  from \ref{enumi: prop: P1 over Noether DVR valued point} that
  \ref{enumi: prop: P1 over Noether DVR valued point} continues to hold
  even if the condition ``\(R\) is a DVR'' is replaced by
  the condition ``\(R\) is a field''.)
  Let \(\Spec(R) \to S\) be a morphism of schemes such that \(R\) is a DVR. 
  Write \(V\dfn \Spec(R)\). 
  Then, by \autoref{lem: affine over qsep is qcpt} 
  \ref{enumi: lem: aff over qsep is qcpt qc} 
  \ref{enumi: lem: aff over qsep is qcpt sep}, 
  \(V\to S\) belongs to \(\Schb{S}\). 
  Since \(X\cong \mathbb{P}^1_S\),
  to prove assertion \ref{enumi: prop: P1 over Noether DVR valued point},
  it suffices to prove that \(Y_{V,\red}\cong \mathbb{P}^1_V\).
  Suppose that \(\bbbullet = \bbullet \cup \left\{ \red\right\}\).
  Write \(\eta\) for the generic point of \(V\) and
  \(v\) for the closed point of \(V\).
  By \autoref{lem: fiber product bbullet}
  \ref{enumi: fiber product bbullet red},
  the fiber product \(Z\dfn Y\times^{\bbbullet}_S V\) exists in \(\Schc{S}\)
  and is naturally isomorphic to \(Y_{V,\red}\). 
  By \autoref{defi: P1 like} \ref{enumi: P1-like Noether proper},
  \(Y\) is proper over \(S\).
  Hence the projection morphism \(Z\to V\) is proper.
  By \autoref{defi: P1 like} \ref{enumi: P1-like Noether connected},
  it holds that \(Z\times^{\bbbullet}_V \Spec(k(\eta)) \cong \mathbb{P}^1_{k(\eta)}\), and
  \(Z\times^{\bbbullet}_V \Spec(k(v))\cong \mathbb{P}^1_{k(v)}\).
  Write \(\overline{Z}_{\eta}\subset Z\) for the schematic closure of
  the generic fiber of the natural morphism \(Z\to V\).
  Thus, since \(V\) is the spectrum of a DVR,
  the natural morphism \(\overline{Z}_{\eta}\to V\) is proper and flat.
  Hence, by
  \cite[\href{https://stacks.math.columbia.edu/tag/0D4J}{Tag 0D4J}]{stacks-project},
  the dimension of the special fiber of
  the natural morphism \(\overline{Z}_{\eta}\to V\) is one.
  Since
  \(Z\times^{\bbbullet}_V \Spec(k(v))\cong \mathbb{P}^1_{k(v)}\),
  the morphism of underlying topological spaces
  \(|\overline{Z}_{\eta}| \to |Z|\) is a homeomorphism.
  Thus, since \(Z\) is reduced,
  the closed immersion \(\overline{Z}_{\eta} \xrightarrow{\sim} Z\)
  is an isomorphism.
  In particular, \(Z\) is integral.

  Since \(V\) is reduced, and \(Z\cong Y_{V,\red}\),
  for each \(i\in \{0,1,\infty\}\),
  there exists a unique section \(s_{i,V}':V\to Z\) of the projection \(Z\to V\)
  such that the composite of \(s_{i,V}'\) with the closed immersion \(Z\subset Y_V\)
  is \(s_{i,V}:V\to Y_V\).
  Next, observe that by \autoref{lem: fiber product bbullet}
  \ref{enumi: fiber product bbullet red},
  for each \(i\in \{0,1,\infty\}\), the natural morphism
  \(Z\setminus \im(s_{i,V}') \xrightarrow{\sim}%
  (Y\setminus \im(s_i))\times_S^\bbbullet V\)
  is an isomorphism.
  Hence, by \autoref{defi: P1 like}
  \ref{enumi: P1-like Noether section}
  \ref{enumi: P1-like Noether ring},
  for any \(i,j,k\in \{0,1,\infty\}\) such that
  \(\{i,j,k\} = \{0,1,\infty\}\),
  \(Z\setminus \im(s_{i,V}')\) has a ring object structure in \((\Schc{S})_{/V}\)
  such that (the morphism \(V\to Z\setminus \im(s_{i,V}')\) determined by)
  \(s_{j,V}'\) is the multiplicative identity section, and
  (the morphism \(V\to Z\setminus \im(s_{i,V}')\) determined by)
  \(s_{k,V}'\) is the additive identity section.
  Since \(Z\times^{\bbbullet}_V k(\eta) \cong \mathbb{P}^1_{k(\eta)}\),
  for each \(i\in \{0,1,\infty\}\),
  it holds that
  \((Z\setminus \im(s_{i,V}'))_\eta \cong \mathbb{A}^1_{k(\eta)}\)
  (as \(k(\eta)\)-schemes).
  In particular,
  for each \(i\in \{0,1,\infty\}\),
  \((Z\setminus \im(s_{i,V}'))_\eta\) is geometrically reduced over \(k(\eta)\).
  Moreover, since \(Z\to V\) is proper,
  \(Z\) is Noetherian.
  In particular, for any \(i\in \{0,1,\infty\}\),
  \(Z\setminus \im(s_{i,V}')\) is quasi-compact over \(\mathbb{Z}\).
  Hence, since \(Z\) is integral,
  it follows from \autoref{cor: group over DVR} that
  for any \(i,j,k\in \{0,1,\infty\}\) such that
  \(\{i,j,k\} = \{0,1,\infty\}\),
  the ring object structure on \(Z\setminus \im(s_{i,V}')\) in \((\Schc{S})_{/V}\)
  induces a ring scheme structure over \(V\) on \(Z\setminus \im(s_{i,V}')\)
  such that (the morphism \(V\to Z\setminus \im(s_{i,V}')\) determined by)
  \(s_{j,V}'\) is the multiplicative identity section, and
  (the morphism \(V\to Z\setminus \im(s_{i,V}')\) determined by)
  \(s_{k,V}'\) is the additive identity section.

  For each \(i\in \{0,1,\infty\}\),
  since \(Z\setminus \im(s_{i,V}')\) is a flat separated ring scheme
  of finite type over \(V\),
  \((Z\setminus \im(s_{i,V}'))_{v,\red} \cong \mathbb{A}^1_{k(v)}\)
  (as \(k(v)\)-schemes)
  is one-dimensional and connected, and
  \((Z\setminus \im(s_{i,V}'))_\eta \cong \mathbb{A}^1_{k(\eta)}\)
  (as ring schemes over \(k(\eta)\)),
  it follows from \autoref{lem: A1 over DVR} that
  the ring scheme \(Z\setminus \im(s_{i,V}')\) over \(V\)
  is isomorphic as a ring scheme over \(V\)
  to the projection \(\mathbb{A}^1_V\to V\).
  Moreover, for any \(i,j,k\in \{0,1,\infty\}\)
  such that \(\{i,j,k\} = \{0,1,\infty\}\),
  since \(s_{j,V}'\) is the multiplicative identity section
  of the ring scheme \(Z\setminus \im(s_{i,V}')\) over \(V\), and
  \(s_{k,V}'\) is the additive identity section
  of the ring scheme \(Z\setminus \im(s_{i,V}')\) over \(V\),
  by observing that \(Z\) is naturally isomorphic to the
  \(V\)-scheme obtained by gluing together
  \(Z\setminus \im(s_{i,V}')\) and \(Z\setminus \im(s_{k,V}')\)
  along \(Z\setminus (\im(s_{i,V}')\cup \im(s_{k,V}'))\),
  we conclude that \(Z\cong \mathbb{P}^1_V\).
  Thus, in particular, \(C(V)\) is of cardinality \(1\),
  i.e., assertion
  \ref{enumi: prop: P1 over Noether DVR valued point} holds.

  By \autoref{lem: homeo Noeth} and
  assertion \ref{enumi: prop: P1 over Noether DVR valued point},
  \(p:C\to S\) is a (necessarily quasi-compact) homeomorphism.
  Hence, by \autoref{lem: representable by AS},
  \autoref{lem: val criteria st Noeth}, and
  assertion \ref{enumi: prop: P1 over Noether DVR valued point},
  \(p:C\to S\) is proper.
  Moreover, by
  \cite[\href{https://stacks.math.columbia.edu/tag/0A4X}{Tag 0A4X}]{stacks-project} and
  assertion \ref{enumi: prop: P1 over Noether DVR valued point},
  \(p:C\to S\) is finite.
  In particular, since \(S\) is a scheme,
  \(C\) is a finite \(S\)-scheme (cf. \cite[\href{https://stacks.math.columbia.edu/tag/03ZP}{Tag 03ZP}]{stacks-project}),
  hence may be written in the form \(\Spec_S(p_*\mathcal{O}_C)\).
  By \autoref{defi: P1 like} \ref{enumi: P1-like Noether connected},
  for any generic point \(\eta\in S\),
  it holds that \(C_{\eta} \cong \Spec(k(\eta))\).
  Since
  \begin{itemize}
    \item \(S\) is normal,
    \item \(\mathcal{O}_S, p_*\mathcal{O}_{C_{\red}}\) are subsheaves
    of the sheaf of rational functions \(\mathcal{K}_S\) on \(S\), and
    \item
    the morphism of sheaves of rings
    \(\mathcal{O}_S\to p_*\mathcal{O}_{C_{\red}}\)
    is finite and compatible with
    the inclusions \(\mathcal{O}_S, p_*\mathcal{O}_{C_{\red}} \subset \mathcal{K}_S\),
  \end{itemize}
  the composite
  \(C_{\red}\to C \xrightarrow{p} S\) is an isomorphism.
  Thus \(C(S)\) is of cardinality \(1\), i.e.,
  \(X\) satisfies condition \ref{enumi: P^1 univ} of \autoref{prop: P1 over Noether}.
  This completes the proof of necessity.

  Next, we prove sufficiency.
  Assume that \(X\) satisfies conditions
  \ref{enumi: P^1 is P^1 like} and \ref{enumi: P^1 univ}
  of \autoref{prop: P1 over Noether}.
  Write \(Y\dfn \mathbb{P}^1_S\) and
  for each \(i\in \{0,1,\infty\}\),
  \(t_i:S\to Y\) for the morphism in \(\Schb{S}\)
  obtained by base-change of the sections
  \(0,1,\infty: \Spec(\mathbb{Z}) \to \mathbb{P}^1_{\mathbb{Z}}\).
  By the necessity portion of \autoref{prop: P1 over Noether},
  \((Y,t_0,t_1,t_\infty)\) satisfies condition
  \ref{enumi: P^1 univ} of \autoref{prop: P1 over Noether}.
  In particular, since \(X\) satisfies condition
  \ref{enumi: P^1 is P^1 like} of \autoref{prop: P1 over Noether},
  there exists a unique closed immersion
  \(h_1:Y\to X\) in \(\Schb{S}\) such that
  for any \(i\in \{0,1,\infty\}\), \(h_1\circ t_i = s_i\).
  By the necessity portion of \autoref{prop: P1 over Noether},
  \((Y,t_0,t_1,t_\infty)\) satisfies condition
  \ref{enumi: P^1 is P^1 like} of \autoref{prop: P1 over Noether}.
  In particular,
  since \(X\) satisfies condition \ref{enumi: P^1 univ} of \autoref{prop: P1 over Noether},
  there exists a unique closed immersion
  \(h_2:X\to Y\) in \(\Schb{S}\) such that
  for any \(i\in \{0,1,\infty\}\), \(h_2\circ s_i = t_i\).
  Then it holds that \(h_1\circ h_2\circ s_i = s_i\), and
  \(h_2\circ h_1 \circ t_i = t_i\).
  Since \((X,s_0,s_1,s_{\infty})\) satisfies conditions
  \ref{enumi: P^1 is P^1 like} \ref{enumi: P^1 univ} of \autoref{prop: P1 over Noether},
  it follows from
  the uniqueness portion of condition \ref{enumi: P^1 univ} of \autoref{prop: P1 over Noether}
  that \(h_1\circ h_2 = \id_X\).
  Since \((Y,t_0,t_1,t_{\infty})\) satisfies conditions
  \ref{enumi: P^1 is P^1 like} \ref{enumi: P^1 univ} of \autoref{prop: P1 over Noether},
  it follows from
  the uniqueness portion of condition \ref{enumi: P^1 univ} of \autoref{prop: P1 over Noether} that
  \(h_2\circ h_1 = \id_Y\).
  This implies that \(X\cong Y\).
  This completes the proof of \autoref{prop: P1 over Noether}.
\end{proof}

\begin{cor}\label{cor: A1 over dugger}
  Let \(S\) be a locally Noetherian normal scheme.
  Let \(X\) be a ring object of \(\Schb{S}\).
  Write \(0_X:S\to X, 1_X:S\to X\) for the additive and multiplicative identity sections.
  Then \(X\) is isomorphic as a ring scheme over \(S\) to \(\mathbb{A}^1_S\)
  if and only if \(S\btimes_{0_X,X,1_X} S = \emptyset\), and
  there exist an open immersion
  \(\iota:X \to \mathbb{P}^1_S\) in \(\Schb{S}\) and
  sections \(s_0,s_1,s_{\infty}:S\to \mathbb{P}^1_S\)
  such that the collection of data \((\mathbb{P}^1_S,s_0,s_1,s_{\infty})\) is \(\mathbb{P}^1\)-like, and
  \(\im(\mathsf{Sp}_{\bbullet/S}(\iota)) =
  \mathsf{Sp}_{\bbullet/S}(\mathbb{P}^1_S) \setminus \im(\mathsf{Sp}_{\bbullet/S}(s_{\infty}))\).
  In particular, the property that
  \(X\) is isomorphic as a ring scheme over \(S\) to \(\mathbb{A}^1_S\)
  may be characterized category-theoretically
  (cf.
  \autoref{cor: open imm is cat},
  \autoref{reconstruction: Top},
  \autoref{prop: P1 over Noether})
  from the data \((\Schb{S}, X)\).
\end{cor}

\begin{proof}
  Necessity follows immediately from \autoref{lem: fiber product bbullet}
  \ref{enumi: fiber product bbullet not red}
  \ref{enumi: fiber product bbullet red},
  \autoref{reconstruction: Top}, and \autoref{prop: P1 over Noether}.
  Next, we prove sufficiency.
  Assume that \(S\btimes_{0_X,X,1_X} S = \emptyset\), and, moreover, that
  there exist an open immersion
  \(\iota:X \to \mathbb{P}^1_S\) in \(\Schb{S}\) and
  sections \(s_0,s_1,s_{\infty}:S\to \mathbb{P}^1_S\)
  such that the collection of data \((\mathbb{P}^1_S,s_0,s_1,s_{\infty})\) is \(\mathbb{P}^1\)-like, and
  \(\im(\mathsf{Sp}_{\bbullet/S}(\iota)) =
  \mathsf{Sp}_{\bbullet/S}(\mathbb{P}^1_S) \setminus \im(\mathsf{Sp}_{\bbullet/S}(s_{\infty}))\).
  Then \(X\cong \mathbb{A}^1_S\).
  Hence, by \autoref{lem: fiber product bbullet}
  \ref{enumi: fiber product bbullet not red}
  \ref{enumi: fiber product bbullet red},
  the ring object structure on \(X\) determines
  a ring scheme structure over \(S\) on \(X \cong \mathbb{A}^1_S\).
  Since \(S\btimes_{0_X,X,1_X} S = \emptyset\),
  it follows from \autoref{lem: fiber product bbullet}
  \ref{enumi: fiber product bbullet not red}
  \ref{enumi: fiber product bbullet red}
  that \(\im(0_X) \cap \im (1_X) = \emptyset\).
  Hence there exists an isomorphism of \(S\)-schemes \(f:X\xrightarrow{\sim} \mathbb{A}^1_S\)
  such that \(0_S = f\circ 0_X, 1_S = f\circ 1_X\).
  Thus, by \autoref{lem: ring str A^1 unique}
  \ref{enumi: any A^1 lem: ring str A^1 unique},
  \(f\) is an isomorphism of ring schemes over \(S\).
  This completes the proof of \autoref{cor: A1 over dugger}.
\end{proof}

Next,
we consider a category-theoretic reconstruction of
the underlying schemes of objects of \(\Schb{S}\).


\newcommand{\spb}[1]{{#1}^{\mathsf{Sp}}_{\bbullet}}
\newcommand{\spc}[1]{{#1}^{\mathsf{Sp}}_{\bbbullet}}
\newcommand{\topol}[1]{{#1}^{|\cdot|}}
\newcommand{\topb}[1]{{#1}^{|\cdot|}_{\bbullet}}
\newcommand{\topc}[1]{{#1}^{|\cdot|}_{\bbbullet}}
\newcommand{\Rspb}{\mathsf{RSp}_{\bbullet/S}}
\newcommand{\Rspc}{\mathsf{RSp}_{\bbbullet/T}}
\newcommand{\RspbA}{\mathsf{RSp}_{\bbullet/S}^A}
\newcommand{\RspcA}{\mathsf{RSp}_{\bbbullet/T}^{A_{\bbbullet}}}
\newcommand{\RspbB}{\mathsf{RSp}_{\bbullet/S}^B}
\newcommand{\RspbulA}{\mathsf{RSp}_{\bbullet/S}^{A_{\bbullet}}}
\newcommand{\RspbulB}{\mathsf{RSp}_{\bbullet/S}^{B_{\bbullet}}}
\newcommand{\RspculA}{\mathsf{RSp}_{\bbbullet/T}^{A_{\bbbullet}}}
\newcommand{\RspculB}{\mathsf{RSp}_{\bbbullet/T}^{B_{\bbbullet}}}

\begin{defi}\label{defi: rec sch}
  Let \(S\) be a locally Noetherian normal scheme.
  \begin{enumerate}
    \item \label{defi: rec sch enumi: mcA}
    We shall write
    \(\mathcal{A}^1_{\bbullet/S}\)
    for the (\vsm) connected groupoid of ring objects and
    isomorphisms of ring objects of \(\Schb{S}\)
    such that \(\mathbb{A}^1_S\in \mathcal{A}^1_{\bbullet/S}\).
    Note that, by \autoref{lem: ring str A^1 unique}
    \ref{enumi: any A^1 lem: ring str A^1 unique},
    for any objects \(A,B \in \mathcal{A}^1_{\bbullet/S}\),
    \(\Hom_{\mathcal{A}^1_{\bbullet/S}}(A,B)\) is of cardinality one.
    \item \label{defi: rec sch enumi: OpIm}
    Let \(X\in \Schb{S}\) be an object.
    Write \(\mathsf{OpIm}_{\bbullet/S}(X)\subset (\Schb{S})_{/X}\)
    for the full subcategory
    determined by quasi-compact open immersions \(U\to X\).
    Then the category \(\mathsf{OpIm}_{\bbullet/S}(X)\)
    is defined completely in terms of properties that may be
    characterized category-theoretically
    (cf. \autoref{cor: open imm is cat}, \autoref{cor: qc is cat} \ref{enumi: qc morph is cat})
    from the data \((\Schb{S},X)\).
    \item \label{defi: rec sch enumi: OpIm varphi}
    Let \(X\in \Schb{S}\) be an object.
    Write \(\mathsf{Op}_{\bbullet/S}^{\mathsf{cat}}(X)\) for the category
    whose objects are elements of \(\mathsf{Op}_{\bbullet/S}(X)\)
    (cf. \autoref{reconstruction: Top}), and
    whose morphisms are the inclusion morphisms.
    We define the following functor:
    \begin{align*}
      \varphi_X : \mathsf{OpIm}_{\bbullet/S}(X) &\to \mathsf{Op}_{\bbullet/S}^{\mathsf{cat}}(X) \\
      [i:U\to X] &\mapsto \im(\mathsf{Sp}_{\bbullet/S}(i)).
    \end{align*}
    Note that, by the definition of the homeomorphism
    \(\eta_X: \mathsf{Sp}_{\bbullet/S}(X) \xrightarrow{\sim} |X|\)
    (cf. \autoref{reconstruction: Top}),
    for any object \([i:U\to X]\in \mathsf{OpIm}_{\bbullet/S}(X)\),
    \(\varphi_X(i:U\to X) = \eta_X^{-1}(\im(|i|))\).
    \item \label{defi: rec sch enumi: Index cat}
    Let \(X\in \Schb{S}\) be an object and
    \(\spb{U}\subset \mathsf{Sp}_{\bbullet/S}(X)\) an open subset.
    We shall write
    \(\mathcal{I}_{\bbullet/S}(X;\spb{U})\subset \mathsf{OpIm}_{\bbullet/S}(X)^{\mathrm{op}}\)
    for the full subcategory determined by the objects
    \([i:U'\to X]\in \mathcal{I}_{\bbullet/S}(X;\spb{U})\) such that
    \(\varphi_X(i:U'\to X)\subset \spb{U}\).
    \item \label{defi: rec sch enumi: OpIm iota}
    Let \(X\in \Schb{S}\) be an object.
    We shall write \(\iota_X: \mathsf{OpIm}_{\bbullet/S}(X) \to \Schb{S}\)
    for the forgetful functor.
    \item \label{defi: rec sch enumi: OpIm f -1}
    Let \(f:X\to Y\) be a morphism in \(\Schb{S}\).
    By \autoref{lem: fiber product bbullet}
    \ref{enumi: fiber product bbullet not red}
    \ref{enumi: fiber product bbullet red},
    for any object \([i:V\to Y]\in \mathsf{OpIm}_{\bbullet/S}(Y)\),
    the natural projection
    \(i_{\bbullet,X}: V\btimes_Y X\to X\)
    is an object of \(\mathsf{OpIm}_{\bbullet/S}(X)\).
    Hence we obtain a functor
    \(f^{-1}:\mathsf{OpIm}_{\bbullet/S}(Y) \to \mathsf{OpIm}_{\bbullet/S}(X)\).
  \end{enumerate}
\end{defi}

\begin{defi}\label{defi: defi: rec sch 2}
  Let \(S\) be a scheme.
  \begin{enumerate}
    \item \label{defi: rec sch enumi: U Sch}
    Let \(\bigstar \dfn \SchU\) or \(\SchV\).
    We define
    \begin{align*}
      U_{\bbullet/S}^{\bigstar}: \Schb{S} &\to \bigstar \\
      X &\mapsto X.
    \end{align*}
    \item \label{defi: rec sch enumi: inclusion}
    Write \(i^{\sfSch}_{\univ{U}\in\univ{V}}: \SchU \to \SchV\)
    for the natural inclusion functor.
    Then the equality
    \(U_{\bbullet/S}^{\SchV} = i^{\sfSch}_{\univ{U}\in\univ{V}}\circ U_{\bbullet/S}^{\SchU}\) holds.
    \item
    For any morphism of schemes \(f:X\to Y\),
    any open subset \(\topol{V}\subset |Y|\), and
    any open subset \(\topol{U}\subset |f|^{-1}(\topol{V})\),
    we shall write
    \[
    f^{\#}(\topol{V},\topol{U}): \mathcal{O}_Y(\topol{V}) \to \mathcal{O}_X(\topol{U})
    \]
    for the morphism of rings determined by the composite of
    the morphism of rings
    \(\mathcal{O}_Y(\topol{V}) \to \mathcal{O}_X(|f|^{-1}(\topol{V}))\)
    induced by \(f\) and the restriction morphism
    \(\mathcal{O}_X(|f|^{-1}(\topol{V}))\to \mathcal{O}_X(\topol{U})\).
  \end{enumerate}
\end{defi}


\begin{reconstruction}[\(\RspbA(X)\)]
  \label{reconstruction: Tch A obj}
  Let \(S\) be a locally Noetherian normal scheme.
  Let \(A\in \mathcal{A}^1_{\bbullet/S}\) a ring object in \(\Schb{S}\)
  such that \(A\cong \mathbb{A}^1_S\) as ring objects in \(\Schb{S}\).

  (i)
  For any object \(X\in \Schb{S}\), we define a functor
  \[
  \check{\mathcal{O}}_{\bbullet/S,X}^A \dfn \Hom_{\Schb{S}}(\iota_X(-),A):
  \mathsf{OpIm}_{\bbullet/S}^{\mathrm{op}}(X) \to \mathsf{Ring}.
  \]
  Then, since \(A\) is isomorphic as a ring object in \(\Schb{S}\) to \(\mathbb{A}^1_S\),
  and \(\mathbb{A}^1_S\) represents the global section functor
  \(\Schb{S} \to \mathsf{Ring}\),
  there exists a unique assignment
  \[
  (X,i) \mapsto [\check{\eta}_X^{A,\#}(i):\mathcal{O}_X(\im(|i|)) \xrightarrow{\sim}
  \check{\mathcal{O}}_{\bbullet/S,X}^A(i)],
  \]
  where \(X\) ranges over the objects of \(\Schb{S}\), and
  \(i\) ranges over the objects of \(\mathsf{OpIm}_{\bbullet/S}(X)\),
  such that \(\check{\eta}_X^{A,\#}(i)\) is an isomorphism of rings, and
  for any diagram
  \[
  \begin{CD}
    U @>{i}>> X \\
    @V{g}VV @VV{f}V \\
    V @>{j}>> Y
  \end{CD}
  \]
  in \(\Schb{S}\),
  if \(i,j\) are quasi-compact open immersions,
  then the diagram of rings
  \begin{equation}
    \label{eq: check O}
    \begin{CD}
      \mathcal{O}_Y(\im(|j|))
      @>{\check{\eta}_Y^{A,\#}(j)}>{\sim}>
      \check{\mathcal{O}}_{\bbullet/S,Y}^A(j) \\
      @V{f^{\#}(\im(|j|), \im(|i|))}VV @VV{(-)\circ g}V \\
      \mathcal{O}_X(\im(|i|))
      @>{\check{\eta}_X^{A,\#}(i)}>{\sim}>
      \check{\mathcal{O}}_{\bbullet/S,X}^A(i)
    \end{CD}
    \tag{\(\dagger\)}
  \end{equation}
  commutes.

  (ii)
  Let \(f:X\to Y\) be a morphism of \(\Schb{S}\).
  For any open subset \(\spb{U}\subset \mathsf{Sp}_{\bbullet/S}(X)\),
  we define
  \begin{align*}
    \mathcal{O}_{\bbullet/S,X}^A(\spb{U}) &\dfn
    \lim_{i\in \mathcal{I}_{\bbullet/S}(X;\spb{U})} \check{\mathcal{O}}_{\bbullet/S,X}^A(i).
  \end{align*}
  For any open subset \(\topol{U} \subset |X|\),
  we define
  \[
  \eta_X^{A,\#}(\topol{U}): \mathcal{O}_X(\topol{U}) \xrightarrow{\sim}
  \mathcal{O}_{\bbullet/S,X}^A(\eta_X^{-1}(\topol{U}))
  \]
  (cf. \autoref{reconstruction: Top})
  to be the composite of the following two isomorphisms
  \[
  \mathcal{O}_X(\topol{U}) \xrightarrow{\sim}
  \lim_{i\in I} \mathcal{O}_X(\im(|i|))
  \xrightarrow[\sim]{\lim_{i\in I} \check{\eta}_X^{A,\#}(i)}
  \lim_{i\in I} \check{\mathcal{O}}_{\bbullet/S,X}^A(i)
  = \mathcal{O}_{\bbullet/S,X}^A(\eta_X^{-1}(\topol{U})),
  \]
  where \(I \dfn \mathcal{I}_{\bbullet/S}(X,\eta_X^{-1}(\topol{U}))\).
  For any open subsets \(\spb{V} \subset \mathsf{Sp}_{\bbullet/S}(Y)\) and
  \(\spb{U} \subset \mathsf{Sp}_{\bbullet/S}(f)^{-1}(\spb{V})\),
  we define
  \begin{align*}
    \mathcal{O}_{\bbullet/S,f}^A(\spb{V},\spb{U}) &\dfn
    \lim [\check{\mathcal{O}}_{\bbullet/S,Y}^A(j)\to \check{\mathcal{O}}_{\bbullet/S,X}^A(i)] \\
    &: \mathcal{O}_{\bbullet/S,Y}^A(\spb{V}) \to \mathcal{O}_{\bbullet/S,X}^A(\spb{U}),
  \end{align*}
  where the limit is indexed by the opposite category associated to
  the category whose objects are commutative diagrams
  \begin{equation}
    \begin{CD}
      U_1 @>{i_1}>> X \\
      @V{g_1}VV @VV{f}V \\
      V_1 @>{j_1}>> Y
    \end{CD}
    \tag{\(i_1:U_1\to X, j_1:V_1 \to Y, g_1\)}
  \end{equation}
  such that \(\varphi_X(i_1)\subset \spb{U}\) and
  \(\varphi_Y(j_1)\subset \spb{V}\), and
  whose morphisms
  \((i_1:U_1\to X, j_1:V_1 \to Y, g_1)\to (i_2:U_2\to X, j_2:V_1 \to Y, g_2)\)
  are pairs of morphisms
  \((s:U_1\to U_2, t:V_1\to V_2)\) in \(\Schb{S}\) such that
  \(i_1 = i_2\circ s, j_1 = j_2 \circ t\), and
  the diagram
  \[
  \begin{CD}
    U_1 @>{s}>> U_2 \\
    @V{g_1}VV @VV{g_2}V \\
    V_1 @>{t}>> V_2
  \end{CD}
  \]
  commutes.
  Then, since the diagram \eqref{eq: check O} commutes,
  for any subset \(\topol{V} \subset |Y|\) and
  any open subset \(\topol{U} \subset |f|^{-1}(\topol{V})\),
  the diagram of rings
  \begin{equation}
    \label{eq: O diagram}
    \begin{CD}
      \mathcal{O}_Y(\topol{V})
      @>{\eta_Y^{A,\#}(\topol{V})}>{\sim}>
      \mathcal{O}_{\bbullet/S,Y}^A(\eta_Y^{-1}(\topol{V})) \\
      @V{f^{\#}(\topol{V},\topol{U})}VV
      @VV{\mathcal{O}_{\bbullet/S,f}^A(\eta_Y^{-1}(\topol{V}),\eta_Y^{-1}(\topol{U}))}V \\
      \mathcal{O}_X(\topol{U})
      @>{\eta_X^{A,\#}(\topol{U})}>{\sim}>
      \mathcal{O}_{\bbullet/S,X}^A(\eta_X^{-1}(\topol{U}))
    \end{CD}
    \tag{\(\ddagger\)}
  \end{equation}
  commutes.
  For any open subsets \(\spb{U}\subset \spb{U'} \subset \mathsf{Sp}_{\bbullet/S}(X)\),
  we define
  \[
  (-)|_{\spb{U}} \dfn \mathcal{O}_{\bbullet/S,\id_X}^A(\spb{U'},\spb{U})
  : \mathcal{O}_{\bbullet/S,X}^A(\spb{U'}) \to \mathcal{O}_{\bbullet/S,X}^A(\spb{U}).
  \]

  (iii) Let \(X\in \Schb{S}\) be an object.
  Since the diagram \eqref{eq: O diagram} commutes,
  it follows from the definitions of
  \(\mathcal{O}_{\bbullet/S,X}^A(\spb{U})\) and \((-)|_{\spb{U}}\) that
  the rings \(\mathcal{O}_{\bbullet/S,X}^A(\spb{U})\) and
  the ring homomorphisms \((-)|_{\spb{U}}\) form
  a presheaf
  \(\mathcal{O}_{\bbullet/S,X}^A: \mathsf{Op}_{\bbullet/S}^{\mathsf{cat}}(X)^{\op} \to \mathsf{Ring}\)
  on \(\mathsf{Sp}_{\bbullet/S}(X)\), and
  the family of isomorphisms of rings
  \[
  (\eta_X^{A,\#}(\topol{U}): \mathcal{O}_X(\topol{U}) \xrightarrow{\sim}
  \mathcal{O}_{\bbullet/S,X}^A(\eta_X^{-1}(\topol{U})))_{\topol{U}\in \mathsf{Open}(|X|)},
  \]
  where \(\mathsf{Open}(|X|)\) is the set of the open sets of \(|X|\),
  determines an isomorphism of presheaves
  \(\eta_X^{A,\#}: \mathcal{O}_X \xrightarrow{\sim} \eta_{X,*}\mathcal{O}_{\bbullet/S,X}^A\)
  on \(|X|\).
  Since \(\mathcal{O}_X\) is a sheaf on \(|X|\), and
  \(\eta_X: \mathsf{Sp}_{\bbullet/S}(X) \xrightarrow{\sim} |X|\) is a homeomorphism,
  the presheaf \(\mathcal{O}_{\bbullet/S,X}^A\)
  on \(\mathsf{Sp}_{\bbullet/S}(X)\) is a sheaf.
  Hence, the pair
  \(\RspbA(X) \dfn
  (\mathsf{Sp}_{\bbullet/S}(X), \mathcal{O}_{\bbullet/S,X}^A)\)
  is a (\vsm) ringed space, and the pair
  \(\eta_X^A \dfn (\eta_X,\eta_X^{A,\#})\)
  may be regarded as an isomorphism of ringed spaces
  \(\eta_X^A : \RspbA(X) \xrightarrow{\sim} X\).
  In particular, \(\RspbA(X)\) is a (\vsm) scheme.
\end{reconstruction}


\begin{reconstruction}[\(\RspbA(f)\)]
  \label{reconstruction: Tch A morph}
  Let \(S\) be a locally Noetherian normal scheme.
  Let \(f:X\to Y\) be a morphism in \(\Schb{S}\), and
  \(A\in \mathcal{A}^1_{\bbullet/S}\) a ring object in \(\Schb{S}\)
  such that \(A\cong \mathbb{A}^1_S\) as ring objects in \(\Schb{S}\).
  For any open subset \(\spb{V} \subset \mathsf{Sp}_{\bbullet/S}(Y)\),
  we define
  \begin{align*}
    \mathcal{O}_{\bbullet/S,f}^{A,\#}(\spb{V}) &\dfn
    \mathcal{O}_{\bbullet/S,f}^A(\spb{V},\mathsf{Sp}_{\bbullet/S}(f)^{-1}(\spb{V})) \\
    &:\mathcal{O}_{\bbullet/S,Y}^A(\spb{V}) \to
    \mathcal{O}_{\bbullet/S,X}^A(\mathsf{Sp}_{\bbullet/S}(f)^{-1}(\spb{V})).
  \end{align*}
  Since the diagram \autoref{reconstruction: Tch A obj} \eqref{eq: O diagram} commutes,
  the family of morphisms
  \[
  \mathcal{O}_{\bbullet/S,f}^{A,\#}\dfn \left(\mathcal{O}_{\bbullet/S,f}^{A,\#}(\spb{V}):
  \mathcal{O}_{\bbullet/S,Y}^A(\spb{V}) \to
  \mathcal{O}_{\bbullet/S,X}^A(\mathsf{Sp}_{\bbullet/S}(f)^{-1}(\spb{V}))
  \right)_{\spb{V}\in \mathsf{Op}_{\bbullet/S}(Y)}
  \]
  determines a morphism of sheaves
  \[\mathcal{O}_{\bbullet/S,f}^{A,\#}: \mathcal{O}_{\bbullet/S,Y}^A\to
  \mathsf{Sp}_{\bbullet/S}(f)_*\mathcal{O}_{\bbullet/S,X}^A\]
  on \(\mathsf{Sp}_{\bbullet/S}(Y)\).
  Hence the pair
  \[\RspbA(f) \dfn (\mathsf{Sp}_{\bbullet/S}(f), \mathcal{O}_{\bbullet/S,f}^{A,\#})\]
  determines a morphism of (\vsm) schemes
  \[\RspbA(f):\RspbA(X)\to \RspbA(Y).\]
  Moreover, since the diagram \autoref{reconstruction: Tch A obj} \eqref{eq: O diagram} commutes,
  for any object \(A\in \mathcal{A}^1_{\bbullet/S}\),
  the diagram of (\vsm) schemes
  \begin{equation}
    \label{eq eta f}
    \begin{CD}
      \RspbA(X) @>{\eta_X^A}>{\sim}> X \\
      @V{\RspbA(f)}VV @VV{f}V \\
      \RspbA(Y) @>{\eta_Y^A}>{\sim}> Y
    \end{CD}
    \tag{\(\ddagger\)}
  \end{equation}
  commutes.
  This implies that
  \(\RspbA: \Schb{S} \to \SchV\)
  is a functor, and the family of isomorphisms of (\vsm) schemes
  \(\eta^A \dfn (\eta_X^A:\RspbA(X)
  \xrightarrow{\sim} X)_{X\in \Schb{S}}\)
  determines an isomorphism of functors
  \(\eta^A: \RspbA\xrightarrow{\sim} U_{\bbullet/S}^{\SchV} =
  i^{\sfSch}_{\univ{U}\in\univ{V}}\circ U_{\bbullet/S}^{\SchU}\).
\end{reconstruction}


\begin{reconstruction}[\(\eta^{A,B}\)]
  \label{reconstruction: eta A to B}
  Let \(S\) be a locally Noetherian normal scheme.

  (i) Let \(X\in \Schb{S}\) be an object and
  \([A\xrightarrow{\sim} B]\in \mathcal{A}^1_{\bbullet/S}\)
  be an isomorphism of ring objects in \(\Schb{S}\).
  For any object \([i:U\to X]\in \mathsf{OpIm}_{\bbullet/S}(X)\), write
  \[
  \check{\eta}_X^{A,B,\#}(i):
  \check{\mathcal{O}}_{\bbullet/S,X}^A(i) \xrightarrow{\sim}
  \check{\mathcal{O}}_{\bbullet/S,X}^B(i)
  \]
  (cf. \autoref{reconstruction: Tch A obj} (i))
  for the isomorphism of rings obtained by composing with
  the isomorphism of ring objects \(A\xrightarrow{\sim} B\).
  Then it follows from the definitions of
  \(\check{\mathcal{O}}_{\bbullet/S,X}^*(i)\) and
  \(\check{\eta}_X^{A,B,\#}(i)\) that
  for any diagram
  \[
  \begin{CD}
    U @>{i}>> X \\
    @V{g}VV @VV{f}V \\
    V @>{j}>> Y,
  \end{CD}
  \]
  if \(i,j\) are quasi-compact immersions,
  then the diagrams of rings
  \begin{equation}
    \label{eq: check eta A to B}
    \begin{CD}
      \check{\mathcal{O}}_{\bbullet/S,Y}^A(j)
      @>{\check{\eta}_Y^{A,B,\#}(j)}>{\sim}>
      \check{\mathcal{O}}_{\bbullet/S,Y}^B(j)  \\
      @V{(-)\circ g}VV @VV{(-)\circ g}V \\
      \check{\mathcal{O}}_{\bbullet/S,X}^A(i)
      @>{\check{\eta}_X^{A,B,\#}(i)}>{\sim}>
      \check{\mathcal{O}}_{\bbullet/S,X}^B(i).
    \end{CD}
    \tag{\(\dagger\)}
  \end{equation}
  commutes.

  (ii) Let \(X\in \Schb{S}\) be an object,
  \(\spb{U}\subset \mathsf{Sp}_{\bbullet/S}(X)\) an open subset, and
  \([A\xrightarrow{\sim} B]\in \mathcal{A}^1_{\bbullet/S}\)
  an isomorphism of ring objects in \(\Schb{S}\).
  Write
  \[
  \eta_X^{A,B,\#}(\spb{U}):
  \mathcal{O}_{\bbullet/S,X}^A(\spb{U})
  \xrightarrow{\sim} \mathcal{O}_{\bbullet/S,X}^B(\spb{U})
  \]
  for the isomorphism defined by
  \[
  \eta_X^{A,B,\#}(\spb{U}) \dfn
  \lim_{i\in \mathcal{I}_{\bbullet/S}(X;\spb{U})} \check{\eta}_X^{A,B,\#}(i).
  \]
  Since the diagram \eqref{eq: check eta A to B} commutes,
  for any morphism \(f:X\to Y\) in \(\Schb{S}\),
  any open subset \(\spb{V} \subset \mathsf{Sp}_{\bbullet/S}(Y)\), and
  any open subset \(\spb{U} \subset \mathsf{Sp}_{\bbullet/S}(X)\) such that
  \(\spb{U} \subset \mathsf{Sp}_{\bbullet/S}(f)^{-1}(\spb{V})\),
  the diagram of rings
  \begin{equation}
    \label{eq: eta A to B}
    \begin{CD}
      \mathcal{O}_{\bbullet/S,Y}^A(\spb{V})
      @>{\eta_Y^{A,B,\#}(\spb{V})}>{\sim}>
      \mathcal{O}_{\bbullet/S,Y}^B(\spb{V}) \\
      @V{\mathcal{O}_{\bbullet/S,f}^A(\spb{V},\spb{U})}VV
      @VV{\mathcal{O}_{\bbullet/S,f}^B(\spb{V},\spb{U})}V \\
      \mathcal{O}_{\bbullet/S,X}^A(\spb{U})
      @>{\eta_X^{A,B,\#}(\spb{U})}>{\sim}>
      \mathcal{O}_{\bbullet/S,X}^B(\spb{U})
    \end{CD}
    \tag{\(\ddagger\)}
  \end{equation}
  commutes.

  (iii) Let \(X\in \Schb{S}\) be an object.
  Then the commutativity of the diagram \eqref{eq: eta A to B}
  implies that the family of isomorphisms
  \[
  \eta_X^{A,B,\#}\dfn \left(\eta_X^{A,B,\#}(\spb{U}) : \mathcal{O}_{\bbullet/S,X}^A(\spb{U})
  \xrightarrow{\sim} \mathcal{O}_{\bbullet/S,X}^B(\spb{U})
  \right)_{\spb{U}\in \mathsf{Op}_{\bbullet/S}(X)}
  \]
  determines an isomorphism
  \(\eta_X^{A,B,\#}: \mathcal{O}_{\bbullet/S,X}^A
  \xrightarrow{\sim} \mathcal{O}_{\bbullet/S,X}^B\)
  of sheaves of rings on \(\mathsf{Sp}_{\bbullet/S}(X)\).
  In particular, the pair
  \(\eta_X^{A,B} \dfn (\id_{\mathsf{Sp}_{\bbullet/S}(X)}, \eta_X^{A,B,\#})\)
  determines an isomorphism of schemes
  \(\eta_X^{A,B}: \RspbA(X)\xrightarrow{\sim} \RspbB(X)\).
  Moreover, the commutativity of the diagram \eqref{eq: eta A to B}
  implies that for any morphism \(f:X\to Y\) in \(\Schb{S}\),
  the diagram of (\vsm) schemes
  \begin{equation}
    \label{eq: eta A to B sch}
    \begin{CD}
      \RspbA(X)
      @>{\eta_X^{A,B}}>{\sim}>
      \RspbB(X)  \\
      @V{\RspbA(f)}VV
      @VV{\RspbB(f)}V \\
      \RspbA(Y)
      @>{\eta_Y^{A,B}}>{\sim}>
      \RspbB(Y)
    \end{CD}
    \tag{\(\dagger\dagger\)}
  \end{equation}
  commutes.
  Thus, the family of isomorphisms
  \[\eta^{A,B}\dfn \left(\eta_X^{A,B}: \RspbA(X)\xrightarrow{\sim} \RspbB(X)\right)_{X\in\Schb{S}}\]
  determines an isomorphism of functors
  \(\eta^{A,B}:\RspbA \xrightarrow{\sim} \RspbB\).
  Moreover, it follows immediately from the definition of \(\eta_X^{A,B,\#}\)
  that for any (iso)morphisms
  \(A\xrightarrow{\sim} B, B\xrightarrow{\sim} C\) in \(\mathcal{A}^1_{\bbullet/S}\),
  \(\eta^A = \eta^B \circ \eta^{A,B}\), and
  \(\eta^{A,C} = \eta^{B,C}\circ \eta^{A,B}\).
\end{reconstruction}


\begin{reconstruction}[\(\Rspb\)]
  \label{reconstruction: Sch}
  Let \(S\) be a locally Noetherian normal scheme.
  By \autoref{reconstruction: Tch A morph} and
  \autoref{reconstruction: eta A to B},
  we obtain a functor
  \begin{align*}
    \mathcal{A}^1_{\bbullet/S} &\to \Hom(\Schb{S},\SchV), \\
    A &\mapsto \RspbA, \\
    [A\xrightarrow{\sim} B] &\mapsto \eta^{A,B},
  \end{align*}
  where \(\Hom(\Schb{S},\SchV)\) denotes the (\vsm) category of functors \(\Schb{S}\to \SchV\) and
  morphisms of functors.
  Define
  \[
  \Rspb\dfn \colim_{A\in \mathcal{A}^1_{\bbullet/S}} \RspbA \
  : \ \Schb{S}\to \SchV.
  \]
  Since \(\mathcal{A}^1_{\bbullet/S}\) is a connected (\vsm) groupoid,
  the colimit \(\Rspb\) exists, and
  the natural morphism
  \(\RspbA \xrightarrow{\sim} \Rspb\)
  is an isomorphism.
  By \autoref{reconstruction: eta A to B},
  the family of isomorphisms of functors
  \(\eta^A: \RspbA \xrightarrow{\sim} U_{\bbullet/S}^{\SchV}\)
  determines an isomorphism of functors
  \(\eta: \Rspb \xrightarrow{\sim}
  U_{\bbullet/S}^{\SchV} =
  i^{\sfSch}_{\univ{U}\in\univ{V}}\circ U_{\bbullet/S}^{\SchU}\).
\end{reconstruction}

\newcommand{\bbul}[1]{{#1}_{\bbullet}}
\newcommand{\bbbul}[1]{{#1}_{\bbbullet}}

\begin{lem}\label{lem: Tch equiv commute}
  Let \(S,T\) be locally Noetherian normal schemes and
  \(F:\Schb{S} \xrightarrow{\sim} \Schc{T}\) an equivalence.
  Then the following assertions hold:
  \begin{enumerate}
    \item \label{lem: Tch equiv commute: FA is A1}
    Let \(\bbul{A}\in \mathcal{A}^1_{\bbullet/S}\).
    Write \(\bbbul{A}\dfn F(\bbul{A})\) for the ring object of \(\Schc{T}\)
    determined by \(\bbul{A}\).
    Then \(\bbbul{A} \in \mathcal{A}^1_{\bbbullet/T}\).
    \item \label{lem: Tch equiv commute: TchA comm}
    Let \(\bbul{A}\in \mathcal{A}^1_{\bbullet/S}\).
    Write \(\bbbul{A}\dfn F(A)\) for the ring object of \(\Schc{T}\)
    determined by \(\bbul{A}\).
    Then \(S,T,F,\bbul{A}\) determine an isomorphism \(\rho^{A}\) between
    the two composite functors of the following diagram:
    \[
    \begin{CD}
      \Schb{S} @>{F}>> \Schc{T} \\
      @V{\RspbulA}VV
      @VV{\RspculA}V \\
      \SchV @= \SchV.
    \end{CD}
    \]
    \item \label{lem: Tch equiv commute: Tch comm}
    \(S,T,F\) determine an isomorphism \(\rho^{\mathsf{Rsp}}\) between
    the two composite functors of the following diagram:
    \[
    \begin{CD}
      \Schb{S} @>F>> \Schc{T} \\
      @V{\Rspb}VV @VV{\Rspc}V \\
      \SchV @= \SchV.
    \end{CD}
    \]
  \end{enumerate}
\end{lem}

\begin{proof}
  Assertion \ref{lem: Tch equiv commute: FA is A1}
  follows immediately from \autoref{cor: A1 over dugger}.

  Next, we prove assertion \ref{lem: Tch equiv commute: TchA comm}.
  Let \(F^{-1}\) be a quasi-inverse of \(F\).
  Since \(F^{-1}\) is a left adjoint of \(F\), and \(\bbbul{A} = F(\bbul{A})\),
  there exists an isomorphism of functors
  \begin{equation}
    \label{check A*}
    \check{A}_*: \Hom_{\Schc{T}}(*,\bbbul{A}) \xrightarrow{\sim} \Hom_{\Schb{S}}(F^{-1}(*),\bbul{A})
    : \Schc{T} \to \mathsf{Ring}.
    \tag{\(\bigstar\)}
  \end{equation}
  For any object \(X\in \Schb{S}\) and any object
  \([\bbbul{i}:\bbbul{U}\to F(X)]\in \mathsf{OpIm}_{\bbbullet/T}(F(X))\),
  we write \(\bbul{U} \dfn F^{-1}(\bbbul{U})\),
  \(\bbul{i}:\bbul{U} \to X\) for the morphism corresponding to
  \(\bbbul{i}:\bbbul{U} \to F(X)\)
  (where we note that by \autoref{cor: open imm is cat} and
  \autoref{cor: qc is cat} \ref{enumi: qc morph is cat},
  for any quasi-compact open immersion \(\bbbul{U}\to F(X)\) in \(\Schc{T}\),
  the corresponding morphism \(F^{-1}(\bbbul{U}) \to X\) in \(\Schb{S}\)
  is also a quasi-compact open immersion), and
  \[\check{A}(\bbbul{i}) \dfn \check{A}_{\bbbul{U}}:
  \check{\mathcal{O}}_{\bbbullet/T,F(X)}^{\bbbul{A}}(\bbbul{i})\to
  \check{\mathcal{O}}_{\bbullet/S,X}^{\bbul{A}}(\bbul{i})\]
  (cf. \autoref{reconstruction: Tch A obj} (i)).
  Then, for any morphism \(f:X\to Y\) in \(\Schb{S}\) and
  any diagram
  \[
  \begin{CD}
    \bbbul{U} @>{\bbbul{i}}>> F(X) \\
    @V{g}VV @VV{F(f)}V \\
    \bbbul{V} @>{\bbbul{j}}>> F(Y)
  \end{CD}
  \]
  such that \(\bbbul{i}, \bbbul{j}\) are quasi-compact open immersions,
  the diagram of rings
  \begin{equation}
    \label{eq: Tch equiv commutes}
    \begin{CD}
      \check{\mathcal{O}}_{\bbbullet/T,F(Y)}^{\bbbul{A}}(\bbbul{j})
      @>{\check{A}(\bbbul{j})}>{\sim}>
      \check{\mathcal{O}}_{\bbullet/S,Y}^{\bbul{A}}(\bbul{j}) \\
      @V{(-)\circ g}VV @VV{(-)\circ F^{-1}(g)}V \\
      \check{\mathcal{O}}_{\bbbullet/T,F(X)}^{\bbbul{A}}(\bbbul{i})
      @>{\check{A}(\bbbul{i})}>{\sim}>
      \check{\mathcal{O}}_{\bbullet/S,X}^{\bbul{A}}(\bbul{i})
    \end{CD}
    \tag{\(\dagger\)}
  \end{equation}
  commutes.

  Let \(X\in \Schb{S}\) be an object.
  Since \(F^{-1}\) is a left adjoint of \(F\),
  \(F^{-1}\) induces an equivalence of categories
  \[
  \mathsf{OpIm}_{\bbbullet/T}(F(X)) \xrightarrow{\sim}
  \mathsf{OpIm}_{\bbullet/S}(X).
  \]
  By the definition of the homeomorphism \(\rho_X\)
  (cf. the proof of \autoref{lem: Sp equiv commute}) associated to the \(S,T,F\) under consideration,
  for any open subset \(\spc{U}\subset \mathsf{Sp}_{\bbbullet/T}(F(X))\),
  \(F^{-1}\) induces an equivalence of categories
  \[
  F_{X;\spc{U}}^{\bbbullet\to \bbullet}: \mathcal{I}_{\bbbullet/T}(F(X);\spc{U})
  \xrightarrow{\sim} \mathcal{I}_{\bbullet/S}(X;\rho_X^{-1}(\spc{U})).
  \]
  For any open subset \(\spc{U}\subset \mathsf{Sp}_{\bbbullet/T}(F(X))\),
  if we write \(\spb{U} \dfn \rho_X^{-1}(\spc{U})\),
  \(\bbbul{I} \dfn \mathcal{I}_{\bbbullet/T}(F(X);\spc{U})\), and
  \(\bbul{I} \dfn \mathcal{I}_{\bbullet/S}(X;\spb{U})\),
  then we define an isomorphism of rings
  \[
  \rho_X^{A,\#}(\spc{U}): \mathcal{O}_{\bbbullet/T,F(X)}^{\bbbul{A}}(\spc{U})
  \xrightarrow{\sim} \mathcal{O}_{\bbullet/S,X}^{\bbul{A}}(\spb{U})
  \]
  (cf. \autoref{reconstruction: Tch A obj} (ii))
  as the composite of
  \[
  \lim_{\bbbul{i}\in \bbbul{I}} \check{A}(\bbbul{i}):
  \mathcal{O}_{\bbbullet/T,F(X)}^{\bbbul{A}}(\spc{U})
  = \lim_{\bbbul{i}\in \bbbul{I}} \check{\mathcal{O}}_{\bbbullet/T,F(X)}^{\bbbul{A}}(\bbbul{i})
  \xrightarrow{\sim} \lim_{\bbbul{i}\in \bbbul{I}}
  \check{\mathcal{O}}_{\bbullet/S,X}^{\bbul{A}}(F_{X;\spc{U}}^{\bbbullet\to \bbullet}(\bbbul{i}))
  \]
  and the inverse of the natural isomorphism
  \[
  \mathcal{O}_{\bbullet/S,X}^{\bbul{A}}(\spb{U})
  = \lim_{\bbul{i}\in \bbul{I}} \check{\mathcal{O}}_{\bbullet/S,X}^{\bbul{A}}(\bbul{i})
  \xrightarrow{\sim} \lim_{\bbbul{i}\in \bbbul{I}}
  \check{\mathcal{O}}_{\bbullet/S,X}^{\bbul{A}}(F_{X;\spc{U}}^{\bbbullet\to \bbullet}(\bbbul{i}))
  \]
  induced by \(F_{X;\spc{U}}^{\bbbullet\to \bbullet}\).
  Since the diagram \eqref{eq: Tch equiv commutes} commutes,
  for any morphism \(f:X\to Y\) in \(\Schb{S}\),
  any open subset \(\spc{V}\subset \mathsf{Sp}_{\bbbullet/T}(F(Y))\), and
  any open subset
  \(\spc{U}\subset \mathsf{Sp}_{\bbbullet/T}(F(f))^{-1}(\spc{V})\),
  if we write \(\spb{U} \dfn \rho_X^{-1}(\spc{U})\) and
  \(\spb{V} \dfn \rho_Y^{-1}(\spc{V})\),
  then the diagram of rings
  \begin{equation}
    \label{eq: b to bb TchA comm}
    \begin{CD}
      \mathcal{O}_{\bbbullet/T,F(Y)}^{\bbbul{A}}(\spc{V})
      @>{\rho_Y^{A,\#}(\spc{V})}>{\sim}>
      \mathcal{O}_{\bbullet/S,Y}^{\bbul{A}}(\spb{V}) \\
      @V{\mathcal{O}_{\bbbullet/T,F(f)}^{\bbbul{A}}(\spc{V},\spc{U})}VV
      @VV{\mathcal{O}_{\bbullet/S,f}^{\bbul{A}}(\spb{V},\spb{U})}V \\
      \mathcal{O}_{\bbbullet/T,F(X)}^{\bbbul{A}}(\spc{U})
      @>{\rho_X^{A,\#}(\spc{U})}>{\sim}>
      \mathcal{O}_{\bbullet/S,X}^{\bbul{A}}(\spb{U})
    \end{CD}
    \tag{\(\ddagger\)}
  \end{equation}
  (cf. \autoref{reconstruction: Tch A obj} (ii))
  commutes.
  Hence the family of isomorphisms of rings
  \[
  \rho_X^{A,\#} \dfn (\rho_X^{A,\#}(\spc{U})%
  )_{\spc{U}\in \mathsf{Op}_{\bbbullet/T}(F(X))}
  \]
  (cf. \autoref{reconstruction: Top})
  determines an isomorphism of sheaves of rings
  \(\rho_X^{A,\#}: \mathcal{O}_{\bbbullet/T,F(X)}^{\bbbul{A}}
  \xrightarrow{\sim} \rho_{X,*} \mathcal{O}_{\bbullet/S,X}^{\bbul{A}}\)
  (cf. \autoref{reconstruction: Tch A obj} (iii)).
  In particular,
  the pair
  \(\rho_X^A\dfn (\rho_X,\rho_X^{A,\#})\)
  determines an isomorphism of schemes
  \(\rho_X^A: \RspbulA(X) \xrightarrow{\sim} \RspculA(F(X))\)
  (cf. \autoref{reconstruction: Tch A obj} (iii)).

  Let \(f:X\to Y\) be a morphism in \(\Schb{S}\).
  Since the diagram \eqref{eq: b to bb TchA comm} commutes,
  the following diagram of sheaves of rings
  on \(\RspculA(F(Y))\) commutes:
  \[
  \begin{tikzpicture}[auto]
    \node (A) at (0,2.5) {\(\mathcal{O}_{\bbbullet/T,F(Y)}^{\bbbul{A}}\)};
    \node (A') at (0,0) {\(\mathsf{Sp}_{\bbbullet/T}(F(f))_*\mathcal{O}_{\bbbullet/T,F(X)}^{\bbbul{A}}\)};
    \node (B) at (7,2.5) {\(\rho_{Y,*} \mathcal{O}_{\bbullet/S,Y}^{\bbul{A}}\)};
    \node (B') at (7,1) {\(\rho_{Y,*}  \mathsf{Sp}_{\bbullet/S}(f)_*\mathcal{O}_{\bbullet/S,X}^{\bbul{A}}\)};
    \node (C) at (7,0.5) {\rotatebox{90}{\(=\)}};
    \node (B'') at (7,0) {\(\mathsf{Sp}_{\bbbullet/T}(F(f))_*\rho_{X,*} \mathcal{O}_{\bbullet/S,X}^{\bbul{A}}\).};
    \draw[->] (A) -- node[swap]  {\(\scriptstyle \mathcal{O}_{\bbbullet/T,F(f)}^{\bbbul{A},\#}\)} (A');
    \draw[->] (B) -- node  {\(\scriptstyle \rho_{Y,*}\mathcal{O}_{\bbullet/S,f}^{\bbul{A},\#}\)} (B');
    \draw[->] (A) -- node  {\(\scriptstyle \rho_Y^{A,\#}\)} (B);
    \draw[->] (A') -- node  {\(\scriptstyle \mathsf{Sp}_{\bbbullet/T}(F(f))_*\rho_X^{A,\#}\)} (B'');
  \end{tikzpicture}
  \]
  (cf. \autoref{lem: Sp equiv commute}, \autoref{reconstruction: Tch A morph}).
  Hence the diagram of schemes
  \[
  \begin{CD}
    \RspbulA(X) @>{\rho_X^A}>{\sim}> \RspcA(F(X)) \\
    @V{\RspbulA(f)}VV
    @VV{\RspcA(F(f))}V \\
    \RspbulA(Y) @>{\rho_Y^A}>{\sim}> \RspcA(F(Y))
  \end{CD}
  \]
  commutes.
  Thus the family of isomorphisms
  \[
  \rho^A \dfn \left(\rho_X^A: \RspbulA(X)\xrightarrow{\sim} \RspcA(F(X))\right)_{X\in \Schb{S}}
  \]
  determines an isomorphism of functors
  \(\rho^A: \RspbulA\xrightarrow{\sim} \RspcA\circ F\).
  This completes the proof of assertion \ref{lem: Tch equiv commute: TchA comm}.


  Finally, we prove assertion \ref{lem: Tch equiv commute: Tch comm}.
  Let \(p:\bbul{A} \to \bbul{B}\) be an isomorphism in \(\mathcal{A}^1_{\bbullet/S}\).
  Write \(\bbbul{A}\dfn F(\bbul{A})\) and \(\bbbul{B}\dfn F(\bbul{B})\).
  Then, by \ref{lem: Tch equiv commute: FA is A1},
  \(\bbbul{A}, \bbbul{B} \in \mathcal{A}^1_{\bbbullet/T}\).
  Moreover, for any quasi-inverse \(F^{-1}\) of \(F\) and any object \(*\in \Schc{T}\),
  the following diagram commutes:
  \[
  \begin{CD}
    \Hom_{\Schc{T}}(*,\bbbul{A}) @>{\check{A}_*}>{\sim}> \Hom_{\Schb{S}}(F^{-1}(*),\bbul{A}) \\
    @V{F(p)\circ (-)}VV @VV{p\circ (-)}V \\
    \Hom_{\Schc{T}}(*,\bbbul{B}) @>{\check{B}_*}>{\sim}> \Hom_{\Schb{S}}(F^{-1}(*),\bbul{B}), \\
  \end{CD}
  \]
  (cf. \eqref{check A*}).
  This commutativity implies that
  the functors ``\(\eta^{A,B}\)'' of \autoref{reconstruction: eta A to B}
  are compatible with passage from ``\(\bbullet\)'' to ``\(\bbbullet\)''.
  Thus, by applying \autoref{reconstruction: Sch},
  we obtain an isomorphism
  \[
  \rho^{\mathsf{Rsp}} \dfn \colim_{\bbul{A}\in \mathcal{A}^1_{\bbullet/S}} \rho^A :
  \Rspb \xrightarrow{\sim} \Rspc \circ F.
  \]
  This completes the proof of \autoref{lem: Tch equiv commute}.
\end{proof}

\begin{cor}\label{cor: underlying sch equiv commute}
  Let \(S,T\) be locally Noetherian normal schemes and
  \(F:\Schb{S} \xrightarrow{\sim} \Schc{T}\) an equivalence.
  Then the following diagram commutes (up to natural isomorphism):
  \[
  \begin{CD}
    \Schb{S} @>F>> \Schc{T} \\
    @V{U_{\bbullet/S}^{\SchU}}VV @VV{U_{\bbbullet/T}^{\SchU}}V \\
    \SchU @= \SchU.
  \end{CD}
  \]
\end{cor}

\

\begin{proof}
  By \autoref{reconstruction: Sch} and
  \autoref{lem: Tch equiv commute} \ref{lem: Tch equiv commute: Tch comm},
  we obtain the following natural isomorphisms:
  \begin{align*}
    i^{\sfSch}_{\univ{U}\in\univ{V}}\circ U_{\bbullet/S}^{\SchU}
    \xleftarrow{\sim} \Rspb
    \xrightarrow[\sim]{\rho^{\mathsf{Rsp}}} \Rspc \circ F
    \xrightarrow{\sim} i^{\sfSch}_{\univ{U}\in\univ{V}}\circ U_{\bbbullet/T}^{\SchU} \circ F,
  \end{align*}
  where the first and the third isomorphisms are
  the ``\(\bbullet\)'' and ``\(\bbbullet\)'' versions of
  the isomorphism ``\(\eta\)'' of \autoref{reconstruction: Sch}.
  Since \(i^{\sfSch}_{\univ{U}\in\univ{V}}\) is fully faithful,
  the composite isomorphism of the above display determines
  a natural isomorphism
  \(U_{\bbullet/S}^{\SchU} \xrightarrow{\sim} U_{\bbbullet/T}^{\SchU} \circ F\).
\end{proof}

In the remainder of this section,
we discuss some properties of \(\Schb{S}\) related to our reconstructions.

The following proposition concerns the rigidity of various forgetful functors.


\begin{prop}\label{lem: forgetful functor is rigid}
  Let \(S\) be a quasi-separated scheme.
  Then the following assertions hold:
  \begin{enumerate}
    \item \label{enumi: forgetful functor set}
    \(\Aut(U_{\bbullet/S}^{\SetU}) = \left\{ \id_{U_{\bbullet/S}^{\SetU}}\right\}\)
    (cf. \autoref{reconstruction: Set} \eqref{equation: underlying set functor}).
    \item \label{enumi: forgetful functor top}
    \(\Aut(U_{\bbullet/S}^{\TopU}) = \left\{ \id_{U_{\bbullet/S}^{\TopU}}\right\}\)
    (cf. \autoref{reconstruction: Top} \eqref{equation: underlying top functor}).
    \item \label{enumi: forgetful functor sch}
    \(\Aut(U_{\bbullet/S}^{\SchU}) = \left\{ \id_{U_{\bbullet/S}^{\SchU}}\right\}\)
    (cf. \autoref{defi: defi: rec sch 2} \ref{defi: rec sch enumi: U Sch}).
  \end{enumerate}
\end{prop}

\begin{proof}
  First, we prove assertion \ref{enumi: forgetful functor set}.
  Let \(\alpha : U_{\bbullet/S}^{\sfSet} \xrightarrow{\sim} U_{\bbullet/S}^{\sfSet}\)
  be an isomorphism.
  Then, for any object \(X\in \Schb{S}\) and any point \(x\in X\),
  if we write \(f:\Spec(k(x)) \to X\) for the morphism of \(\Schb{S}\)
  determined by the point \(x\in X\),
  then the diagram
  \[
  \begin{CD}
    \{x\} = U_{\bbullet/S}^{\SetU}(\Spec(k(x)))
    @>{U_{\bbullet/S}^{\SetU}(f)}>>
    U_{\bbullet/S}^{\SetU}(X) = |X| \\
    @V{\alpha_{\Spec(k(x))}}VV @V{\alpha_X}VV \\
    \{x\} = U_{\bbullet/S}^{\SetU}(\Spec(k(x)))
    @>{U_{\bbullet/S}^{\SetU}(f)}>>
    U_{\bbullet/S}^{\SetU}(X) = |X|
  \end{CD}
  \]
  commutes.
  Since \(U_{\bbullet/S}^{\SetU}(f)(x) = x\),
  it holds that \(\alpha_X(x) = x\).
  Hence for any object \(X\in \Schb{S}\),
  it holds that \(\alpha_X = \id_{U_{\bbullet/S}^{\sfSet}(X)}\).
  This completes the proof of assertion \ref{enumi: forgetful functor set}.

  Assertion \ref{enumi: forgetful functor top} follows immediately
  from assertion \ref{enumi: forgetful functor set}.

  Finally, we prove \ref{enumi: forgetful functor sch}.
  Let \(\alpha: U_{\bbullet/S}^{\SchU} \xrightarrow{\sim} U_{\bbullet/S}^{\SchU}\) be an isomorphism.
  Then \(\alpha\) is a family of isomorphisms of schemes
  \((\alpha_X : X \xrightarrow{\sim} X)_{X\in \Schb{S}}\)
  such that for any morphism \(f:X\to Y\) of \(\Schb{S}\),
  the following diagram of schemes commutes:
  \[
  \begin{CD}
    X @>{f}>> Y  \\
    @V{\alpha_X}VV @VV{\alpha_Y}V \\
    X @>{f}>> Y.
  \end{CD}
  \]
  By assertion \ref{enumi: forgetful functor top},
  each morphism \(\alpha_X\) induces the identity morphism
  on the underlying topological space \(U_{\bbullet/S}^{\TopU}(X)\) of \(X\).
  Hence to prove \ref{enumi: forgetful functor sch},
  it suffices to prove that for any \(X\in \Schb{S}\),
  the isomorphism of sheaves of rings
  \(\alpha_X^{\#}:\mathcal{O}_X \xrightarrow{\sim} \mathcal{O}_X\)
  is the identity morphism.
  Let \(U\) be an affine open subset of \(X\) and
  \(s\in \mathcal{O}_X(U)\) a section.
  Write \(i:U\to X\) for the open immersion.
  To prove that \(\alpha_X^{\#} = \id_{\mathcal{O}_X}\),
  it suffices to prove that \(\alpha_X^{\#}(U)(s) = s\).
  By \autoref{cor: coprod exists} \ref{enumi: cor: coprod exists aff},
  \(U\) belongs to \(\Schb{S}\).
  Since \(i\circ \alpha_U = \alpha_X \circ i\),
  it holds that \(\alpha_U = \alpha_X |_U\).
  Write \(p:\mathbb{A}^1_X\to X\) for the natural projection and
  \(\tilde{s}: U\to \mathbb{A}^1_X\)
  for the morphism over \(X\) in \(\Schb{S}\)
  corresponding to \(s\) such that \(p\circ \tilde{s} = i\).
  If \(\tilde{t}:U\to \mathbb{A}^1_X\) is the morphism corresponding to
  the section \(t=\alpha_X^{\#}(U)(s)\in \mathcal{O}_X(U)\),
  then the following diagram of schemes commutes:
  \[
  \begin{tikzpicture}[auto]
    \node (A) at (0,1.5) {\(U\)};
    \node (A') at (0,0) {\(U\)};
    \node (B) at (4,1.5) {\(\mathbb{A}^1_X\)};
    \node (B') at (4,0) {\(\mathbb{A}^1_X\)};
    \node (C) at (8,1.5) {\(X\)};
    \node (C') at (8,0) {\(X\).};
    \draw[->] (A) -- node[swap]  {\(\scriptstyle \alpha_U = \alpha_X|_U\)} (A');
    \draw[->] (B) -- node[swap]  {\(\scriptstyle \alpha_X\times \mathbb{A}^1_{\mathbb{Z}}\)} (B');
    \draw[->] (C) -- node[swap]  {\(\scriptstyle \alpha_X\)} (C');
    \draw[->] (A) -- node  {\(\scriptstyle \cong\)} (A');
    \draw[->] (B) -- node  {\(\scriptstyle \cong\)} (B');
    \draw[->] (C) -- node  {\(\scriptstyle \cong\)} (C');
    \draw[->] (A) -- node  {\(\scriptstyle \tilde{t}\)} (B);
    \draw[->] (A') -- node  {\(\scriptstyle \tilde{s}\)} (B');
    \draw[->] (B) -- node  {\(\scriptstyle p\)} (C);
    \draw[->] (B') -- node  {\(\scriptstyle p\)} (C');
    \draw[->] (A) to[bend left=20] node  {\(\scriptstyle i\)} (C);
    \draw[->] (A') to[bend right=20] node  {\(\scriptstyle i\)} (C');
  \end{tikzpicture}
  \]
  Since \(\tilde{s} \circ \alpha_U = \alpha_{\mathbb{A}^1_X}\circ \tilde{s}\),
  to prove that \(t = s\),
  it suffices to prove that
  \(\alpha_{\mathbb{A}^1_X} = \id_{\mathbb{A}^1_{\mathbb{Z}}}\times \alpha_X\).
  Moreover, since \(\mathbb{A}^1_X\subset \mathbb{P}^1_X\) is an open subscheme,
  it suffices to prove that
  \(\alpha_{\mathbb{P}^1_X} = \id_{\mathbb{P}^1_{\mathbb{Z}}}\times \alpha_X\).
  Write \(0_X,1_X,\infty_X:X\to \mathbb{P}^1_X\) for the morphisms
  obtained by base-changing the sections \(0,1,\infty: \Spec(\mathbb{Z}) \to \mathbb{P}^1_{\mathbb{Z}}\).
  Then for any \(\iota\in \{0_X,1_X,\infty_X\}\), the following diagram commutes:
  \[
  \begin{tikzpicture}[auto]
    \node (A) at (0,2) {\(X\)};
    \node (A') at (0,0) {\(X\)};
    \node (B) at (4,2) {\(\mathbb{P}^1_X\)};
    \node (B') at (4,0) {\(\mathbb{P}^1_X\)};
    \node (C) at (8,2) {\(X\)};
    \node (C') at (8,0) {\(X\).};
    \draw[->] (A) -- node[swap]  {\(\scriptstyle \alpha _X\)} (A');
    \draw[->] (B) -- node  {\(\scriptstyle \alpha _{\mathbb{P}^1_X}\)} (B');
    \draw[->] (C) -- node  {\(\scriptstyle \alpha _X\)} (C');
    \draw[->] (A) -- node  {\(\scriptstyle \iota\)} (B);
    \draw[->] (A') -- node  {\(\scriptstyle \iota\)} (B');
    \draw[->] (B) -- node  {\(\scriptstyle p\)} (C);
    \draw[->] (B') -- node  {\(\scriptstyle p\)} (C');
    \draw[->] (A) to[bend left=20] node  {\(\scriptstyle \id _X\)} (C);
    \draw[->] (A') to[bend right=20] node[swap]  {\(\scriptstyle \id _X\)} (C');
  \end{tikzpicture}
  \]
  This commutativity implies that
  \(\alpha _{\mathbb{P}^1_X}^{-1}\circ (\id _{\mathbb{P}^1_\mathbb{Z}}\times \alpha _X):
  \mathbb{P}^1_X \xrightarrow{\sim} \mathbb{P}^1_X\)
  is an automorphism of \(\mathbb{P}^1_X\) over \(X\) which preserves \(0_X,1_X,\infty_X\),
  hence that
  \(\alpha_{\mathbb{P}^1_X}^{-1}\circ (\id_{\mathbb{P}^1_\mathbb{Z}}\times \alpha_X) = \id_{\mathbb{P}^1_X}\).
  Thus \(\alpha_{\mathbb{P}^1_X} = \id_{\mathbb{P}^1_\mathbb{Z}}\times \alpha_X\).
  This completes the proof of \autoref{lem: forgetful functor is rigid}.
\end{proof}

The following proposition concerns the rigidity of the composite of \(U_{\bbullet/S}^{\SchU}\)
(cf. \autoref{defi: defi: rec sch 2} \ref{defi: rec sch enumi: U Sch})
with
the functor \(f^*\) determined by base-change by \(f\).
For a morphism \(f:S\to T\) of \(\mathsf{Sch}_{\bbullet/\mathbb{Z}}\),
we shall write
\[
\Aut_T(f) \dfn \left\{\psi:S\xrightarrow{\sim} S \mid f = f\circ\psi\right\}
\]
for the group of automorphisms of \(S\) over \(T\).


\begin{prop}\label{lem: pull-back functor is rigid}
  Let \(T\) be a quasi-separated scheme and
  \(f:S\to T\) a morphism in \(\Schb{T}\).
  Assume that \(\id_T: T \to T\) belongs to \(\Schb{T}\).
  Write \(f^* \dfn (-)\btimes_T S: \Schb{T} \to \Schb{S}\) for the functor determined by the
  operation of base-change, via \(f\), from \(T\) to \(S\)
  (cf. \cite[\href{https://stacks.math.columbia.edu/tag/03GI}{Tag 03GI}]{stacks-project},
  \autoref{lem: fiber product bbullet}
  \ref{enumi: fiber product bbullet not qcpt and red}
  \ref{enumi: fiber product bbullet not qcpt and in red}
  \ref{enumi: fiber product bbullet not red}
  \ref{enumi: fiber product bbullet red}).
  \begin{enumerate}
    \item \label{autfS to autf*}
    For any automorphism \(\psi\in \Aut_T(f)\),
    the family of automorphisms
    \[\alpha^{\psi} \dfn (\id_Y\btimes_T \psi: Y\btimes_T S\xrightarrow{\sim} Y\btimes_T S)_{Y\in \Schb{T}}\]
    determines an automorphism
    \(\alpha^{\psi}: U_{\bbullet/S}^{\SchU}\circ f^*\xrightarrow{\sim}U_{\bbullet/S}^{\SchU}\circ f^*\).
    Moreover, \((\alpha^{\psi})_T = \psi\).
    \item \label{alpha A1 comes from alpha T}
    For any automorphism \(\alpha: U_{\bbullet/S}^{\SchU}\circ f^*\xrightarrow{\sim}U_{\bbullet/S}^{\SchU}\circ f^*\),
    \(\alpha_{\mathbb{A}^1_T} = \id_{\mathbb{A}^1_{\mathbb{Z}}}\times \alpha_T\).
    \item \label{alpha T is in AutfS}
    For any automorphism \(\alpha: U_{\bbullet/S}^{\SchU}\circ f^*\xrightarrow{\sim}U_{\bbullet/S}^{\SchU}\circ f^*\),
    \(\alpha_T: S\xrightarrow{\sim} S\) belongs to \(\Aut_T(f)\).
    \item \label{alpha comes from alpha T}
    For any automorphism \(\alpha: U_{\bbullet/S}^{\SchU}\circ f^*\xrightarrow{\sim}U_{\bbullet/S}^{\SchU}\circ f^*\) and
    any object \(Y\in \Schb{T}\),
    \(\alpha_Y = \id_Y\btimes_T \alpha_T\).
    \item \label{conclude}
    The morphism of groups
    \begin{align*}
      (-)_T: \Aut(U_{\bbullet/S}^{\SchU}\circ f^*) &\to \Aut(S)  \\
      \alpha &\mapsto \alpha_T
    \end{align*}
    is injective, and its image is equal to \(\Aut_T(f) \subset \Aut(S)\).
    In particular, \(\Aut(U_{\bbullet/S}^{\SchU}\circ f^*)\cong \Aut_T(f)\).
  \end{enumerate}
\end{prop}

\begin{proof}
  Assertion \ref{autfS to autf*} follows immediately from the definitions.

  Next, we prove assertion \ref{alpha A1 comes from alpha T}. 
  Let \(\alpha:U_{\bbullet/S}^{\SchU}\circ f^*\xrightarrow{\sim}U_{\bbullet/S}^{\SchU}\circ f^*\) be an automorphism.
  Thus \(\alpha\) is a family of automorphisms
  \((\alpha_Y:Y\btimes_TS \xrightarrow{\sim} Y\btimes_TS)_{Y\in \Schb{T}}\)
  such that for any morphism \(g:Y\to Y'\) in \(\Schb{T}\)
  the following diagram in \(\SchU\) commutes:
  \begin{equation}
    \label{f* diag}
    \begin{CD}
      Y\btimes_TS @>{g\btimes_T \id_S}>> Y'\btimes_TS \\
      @V{\alpha_Y}V{\cong}V @V{\alpha_{Y'}}V{\cong}V \\
      Y\btimes_TS @>{g\btimes_T \id_S}>> Y'\btimes_TS.
    \end{CD}
    \tag{\(\dagger\)}
  \end{equation}
  Let \(Y\in \Schb{T}\) be an object.
  Write \(X\dfn Y\btimes_T S\);
  \(0_X,1_X,\infty_X: X\to \mathbb{P}^1_X\) for the morphisms
  obtained by base-changing the sections \(0,1,\infty: \Spec(\mathbb{Z}) \to \mathbb{P}^1_{\mathbb{Z}}\);
  \(0_Y,1_Y,\infty_Y: Y\to \mathbb{P}^1_Y\) for the morphisms
  obtained by base-changing the sections \(0,1,\infty: \Spec(\mathbb{Z}) \to \mathbb{P}^1_{\mathbb{Z}}\);
  \(p_X:\mathbb{P}^1_X\to X\) and \(p_Y:\mathbb{P}^1_Y\to Y\) for the natural projections.
  Then, since \(\mathbb{P}^1_X\) is naturally isomorphic to \(\mathbb{P}^1_Y \btimes_T S\),
  for each \(i\in \{0,1,\infty\}\),
  it holds that \(i_X = i_Y\btimes_T \id_S\) and \(p_X = p_Y\btimes_T \id_S\),
  and the following diagram in \(\SchU\) commutes:
  \[
  \begin{tikzpicture}[auto]
    \node (A) at (0,1.5) {\(X\)};
    \node (A') at (0,0) {\(X\)};
    \node (B) at (5,1.5) {\(\mathbb{P}^1_X\)};
    \node (B') at (5,0) {\(\mathbb{P}^1_X\)};
    \node (C) at (10,1.5) {\(X\)};
    \node (C') at (10,0) {\(X\).};
    \draw[->] (A) -- node[swap]  {\(\scriptstyle \alpha_Y\)} (A');
    \draw[->] (B) -- node  {\(\scriptstyle \alpha_{\mathbb{P}^1_Y}\)} (B');
    \draw[->] (C) -- node  {\(\scriptstyle \alpha_Y\)} (C');
    \draw[->] (A) -- node  {\(\scriptstyle i_X = i_Y\btimes_T \id_S\)} (B);
    \draw[->] (A') -- node  {\(\scriptstyle i_X = i_Y\btimes_T \id_S\)} (B');
    \draw[->] (B) -- node  {\(\scriptstyle p_X = p_Y\btimes_T \id_S\)} (C);
    \draw[->] (B') -- node  {\(\scriptstyle p_X = p_Y\btimes_T \id_S\)} (C');
    \draw[->] (A) to[bend left=20] node  {\(\scriptstyle \id_X\)} (C);
    \draw[->] (A') to[bend right=20] node  {\(\scriptstyle \id_X\)} (C');
  \end{tikzpicture}
  \]
  Hence \(\alpha_{\mathbb{P}^1_Y}^{-1}\circ (\id_{\mathbb{P}^1_\mathbb{Z}}\times \alpha_Y):
  \mathbb{P}^1_X \xrightarrow{\sim} \mathbb{P}^1_X\)
  is an automorphism of \(\mathbb{P}^1_X\) over \(X\) which preserves \(0_X,1_X,\infty_X\).
  This implies that
  \(\alpha_{\mathbb{P}^1_Y}^{-1}\circ (\id_{\mathbb{P}^1_\mathbb{Z}}\times \alpha_Y) = \id_{\mathbb{P}^1_X}\).
  Thus \(\alpha_{\mathbb{P}^1_Y} = \id_{\mathbb{P}^1_\mathbb{Z}}\times \alpha_Y\).
  In particular, \(\alpha_{\mathbb{P}^1_T} = \id_{\mathbb{P}^1_\mathbb{Z}}\times \alpha_T\).
  Since \(\mathbb{A}^1_T\subset \mathbb{P}^1_T\) is an open subscheme,
  the commutativity of \eqref{f* diag} implies that
  \(\alpha_{\mathbb{A}^1_T} = \id_{\mathbb{A}^1_{\mathbb{Z}}}\times \alpha_T\).
  This completes the proof of assertion \ref{alpha A1 comes from alpha T}. 

  Next, we prove assertion \ref{alpha T is in AutfS}.
  Let \(\alpha:U_{\bbullet/S}^{\SchU}\circ f^*\xrightarrow{\sim}U_{\bbullet/S}^{\SchU}\circ f^*\) be an automorphism and
  \(q\in T\) a point.
  Write \(Y \dfn \Spec(k(q))\) and \(i:Y\to T\) for the natual morphism.
  Then, by the commutativity of \eqref{f* diag},
  \(\alpha_T\circ (i\btimes_T \id_S) = (i\btimes_T \id_S) \circ \alpha_Y\).
  Hence, since \(i\btimes_T \id_S: Y\btimes_T S \to S\) and
  \(\id_Y \btimes_T f : Y\btimes_T S \to Y\) are the natural projections,
  for any \(p\in f^{-1}(q)\),
  \begin{align*}
    |f|(|\alpha_T|(|i\btimes_T \id_S|(p)))
    &= |f|(|i\btimes_T \id_S|(|\alpha_Y|(p))) \\
    &= |i|(|\id_Y \btimes_T f|(|\alpha_Y|(p))) \\
    &\in \im(|i|) = \{q\}.
  \end{align*}
  This implies that \(|f\circ \alpha_T| = |f|\).
  Hence, to prove assertion \ref{alpha T is in AutfS},
  it suffices to prove that
  \(f^{\#} = f^{\#}\circ f_*(\alpha_T^{\#}):\mathcal{O}_T \to f_*\mathcal{O}_S\).

  Let \(V\subset T\) be an affine open subscheme and
  \(t\in \mathcal{O}_T(V)\) a section.
  Then, by \autoref{cor: coprod exists} \ref{enumi: cor: coprod exists aff}, \(V\in \Schb{T}\).
  Write \(f_V^{\#}:\mathcal{O}_T(V)\to \mathcal{O}_S(f^{-1}(V))\)
  for the morphism of rings of sections over \(V\subset T\) and
  \(\alpha_{T,f^{-1}(V)}^{\#}: \mathcal{O}_S(f^{-1}(V))\xrightarrow{\sim} \mathcal{O}_S(f^{-1}(V))\)
  for the morphism of rings of sections over \(f^{-1}(V) = V\btimes_T S\subset S\).
  Then \(t\in \mathcal{O}_T(V)\) corresponds to a morphism of \(T\)-schemes \(\tilde{t}:V\to \mathbb{A}^1_T\) in \(\Schb{T}\),
  \(f_V^{\#}(t) \in \mathcal{O}_S(f^{-1}(V))\) corresponds to the morphism of \(S\)-schemes
  \(\tilde{t}\btimes_T \id_S:f^{-1}(V) \to \mathbb{A}^1_S\) in \(\Schb{S}\), and
  \(\alpha_{T,f^{-1}(V)}^{\#}(f_V^{\#}(t))\in \mathcal{O}_S(f^{-1}(V))\) corresponds to
  the morphism of \(S\)-schemes
  \(\tilde{u}: f^{-1}(V) \to \mathbb{A}^1_S\) in \(\Schb{S}\)
  such that
  all small square in the following diagram in \(\SchU\) are Cartesian: 
  \[
  \begin{tikzpicture}[auto]
    \node (A) at (0,3) {\(f^{-1}(V)\)};
    \node (A') at (0,1.5) {\(f^{-1}(V)\)};
    \node (A'') at (0,0) {\(V\)};
    \node (B) at (4,3) {\(\mathbb{A}^1_S\)};
    \node (B') at (4,1.5) {\(\mathbb{A}^1_S\)};
    \node (B'') at (4,0) {\(\mathbb{A}^1_T\)};
    \node (C) at (8,3) {\(S\)};
    \node (C') at (8,1.5) {\(S\)};
    \node (C'') at (8,0) {\(T.\)};
    \draw[->] (A) --  (A');
    \draw[->] (A') -- node[swap]  {\(\scriptstyle \id_V\btimes_T f\)} (A'');
    \draw[->] (B) -- node  {\(\scriptstyle \id_{\mathbb{A}^1_{\mathbb{Z}}} \times \alpha_T\)} (B');
    \draw[->] (B') -- node  {\(\scriptstyle \id_{\mathbb{A}^1_{\mathbb{Z}}} \times f\)} (B'');
    \draw[->] (C) -- node  {\(\scriptstyle \alpha_T\)} (C');
    \draw[->] (C') -- node  {\(\scriptstyle f\)} (C'');
    \draw[->] (A) -- node  {\(\scriptstyle \tilde{u}\)} (B);
    \draw[->] (A') -- node  {\(\scriptstyle \tilde{t}\btimes_T \id_S\)} (B');
    \draw[->] (A'') -- node  {\(\scriptstyle \tilde{t}\)} (B'');
    \draw[->] (B) --  (C);
    \draw[->] (B') --  (C');
    \draw[->] (B'') --  (C'');
  \end{tikzpicture}
  \]
  By assertion \ref{alpha A1 comes from alpha T},
  \(\id_{\mathbb{A}^1_{\mathbb{Z}}} \times \alpha_T = \alpha_{\mathbb{A}^1_T}\).
  Moreover, since the diagram the diagram \eqref{f* diag} and
  the upper left-hand square of the above diagram are Cartesian,
  it holds that \(\tilde{u} = \tilde{t}\btimes_T \id_S\).
  This implies that \(\alpha_{T,f^{-1}(V)}^{\#}(f_V^{\#}(t)) = f_V^{\#}(t)\).
  This completes the proof of assertion \ref{alpha T is in AutfS}.

  Next, we prove assertion \ref{alpha comes from alpha T}. 
  By assertion \ref{alpha A1 comes from alpha T},
  \(\id_{\mathbb{A}^1_{\mathbb{Z}}} \times \alpha_T = \alpha_{\mathbb{A}^1_T}\). 
  Hence, by the commutativity of \eqref{f* diag} and \autoref{cor: coprod exists} \ref{enumi: cor: coprod exists aff},
  for any affine open subscheme \(V\subset T\),
  \begin{equation*}
    \label{eq: auto 1}
    \alpha_{\mathbb{A}^1_V} = \id_{\mathbb{A}^1_{\mathbb{Z}}}\times (\alpha_T|_{f^{-1}(V)}).
    \tag{\(\ddagger\)}
  \end{equation*}
  Moreover, by the commutativity of \eqref{f* diag} and equation \eqref{eq: auto 1},
  for any \(n\geq 1\) and any affine open subscheme \(V\subset T\),
  \begin{equation*}
    \label{eq: auto n}
    \alpha_{\mathbb{A}^n_V} = \id_{\mathbb{A}^n_{\mathbb{Z}}}\times (\alpha_T|_{f^{-1}(V)}).
    \tag{\(\dagger\dagger\)}
  \end{equation*}
  Thus, by the commutativity of \eqref{f* diag}, equation \eqref{eq: auto n}, and
  assertion \ref{alpha T is in AutfS}, 
  for any affine open subscheme \(\Spec(B) \cong V \subset T\) and
  any (\usm) set \(I\),
  if we write \(B \dfn \Gamma(V,\mathcal{O}_V)\) and \(N\dfn B^{\oplus I}\), then
  \begin{equation*}
    \label{eq: auto I}
    \alpha_{\Spec(\mathrm{Sym}_B(N))}
    = \id_{\Spec(\mathrm{Sym}_{\mathbb{Z}}(\mathbb{Z}^{\oplus I}))} \times (\alpha_T|_{f^{-1}(V)})
    = \id_{\Spec(\mathrm{Sym}_B(N))} \btimes_T \alpha_T.
    \tag{\(\ddagger\ddagger\)}
  \end{equation*}
  Note that for any affine open subscheme \(V\subset T\) and
  any affine scheme \(\Spec(A) \to V\),
  if we write \(B \dfn \Gamma(V,\mathcal{O}_V)\),
  then \(\Spec(A)\) may be regarded as a closed subscheme of \(\Spec(\mathrm{Sym}_B(B^{\oplus A}))\).
  Hence, by the commutativity of \eqref{f* diag}, equation \eqref{eq: auto I},
  and \autoref{cor: coprod exists} \ref{enumi: cor: coprod exists aff},
  for any affine open subscheme \(V\subset T\) and
  any affine scheme \(W\) over \(V\) (in \(\Schb{T}\)),
  \(\alpha_W = \id_W\btimes_T \alpha_T\).
  Thus, by the commutativity of \eqref{f* diag}, for any object \(Y\in \Schb{T}\),
  \(\alpha_Y = \id_Y\btimes_T \alpha_T\).
  This completes the proof of assertion \ref{alpha comes from alpha T}.

  Assertion \ref{conclude} follows immediately from
  assertions \ref{autfS to autf*}, \ref{alpha comes from alpha T}, and \ref{alpha T is in AutfS}.
  This completes the proof of \autoref{lem: pull-back functor is rigid}.
\end{proof}

\begin{cor}\label{cor: isom of pb is triv}
  Let \(T\) be a quasi-separated scheme and
  \(f:S\to T\) a morphism in \(\Schb{T}\).
  Assume that \(\id_T: T \to T\) belongs to \(\Schb{T}\).
  Write \(f^* \dfn (-)\btimes_T S: \Schb{T} \to \Schb{S}\) for the functor determined by the
  operation of base-change, via \(f\), from \(T\) to \(S\)
  (cf. \cite[\href{https://stacks.math.columbia.edu/tag/03GI}{Tag 03GI}]{stacks-project},
  \autoref{lem: fiber product bbullet}
  \ref{enumi: fiber product bbullet not qcpt and red}
  \ref{enumi: fiber product bbullet not qcpt and in red}
  \ref{enumi: fiber product bbullet not red}
  \ref{enumi: fiber product bbullet red}).
  Then \(\Aut(f^*) = \{\id_{f^*}\}\).
\end{cor}

\begin{proof}
  Let \(\alpha:f^*\xrightarrow{\sim} f^*\) be an isomorphism.
  Then, by \autoref{lem: pull-back functor is rigid} \ref{alpha comes from alpha T},
  for any object \(Y\in \Schb{T}\),
  it holds that \(\alpha_Y = \id_Y\btimes_T \alpha_T\).
  Since \(\alpha_T: S\xrightarrow{\sim} S\) is an isomorphism in \(\Schb{S}\),
  it holds that \(\alpha_T = \id_S\).
  This implies that \(\alpha = \id_{f^*}\).
  This completes the proof of \autoref{cor: isom of pb is triv}.
\end{proof}

Finally, we prove the main result of the present paper.



\begin{thm}\label{thm: isom cat equiv}
  Let \(S\) and \(T\) be locally Noetherian normal schemes and
  \(\bbullet \subset \rqqs\) a subset.
  Then the natural functor
  \begin{align*}
    \Isom (S,T) &\to \ISOM (\Schb{T},\Schb{S}) \\
    f &\mapsto f^*
  \end{align*}
  is an equivalence of (\vsm) categories.
\end{thm}

\begin{proof}
  First, we observe that since \(\Isom(S,T)\) is a discrete category,
  the functor \(f \mapsto f^*\) is faithful.

  Next, we verify that the functor \(f\mapsto f^*\) is full.
  Let \(f,g: S \xrightarrow{\sim} T\) be isomorphisms.
  Assume that there exists an isomorphism of functors \(\alpha: f^* \xrightarrow{\sim} g^*\).
  Write \(h \dfn g^{-1}\circ f : S\xrightarrow{\sim} S\) for the isomorphism of
  \(T\)-schemes from \(f:S\to T\) to \(g:S\to T\),
  \(\tilde{\alpha}:U_{\bbullet/S}^{\SchU}\circ f^*\xrightarrow{\sim} U_{\bbullet/S}^{\SchU}\circ g^*\)
  (cf. \autoref{defi: defi: rec sch 2} \ref{defi: rec sch enumi: U Sch})
  for the isomorphism of functors induced by \(\alpha: f^* \xrightarrow{\sim} g^*\), and
  \(\beta \dfn ( \id_Y\btimes_T h : Y\btimes_{T,f} S \xrightarrow{\sim} Y\btimes_{T,g} S )_{Y\in \Schb{T}}\)
  for the isomorphism of functors
  \(\beta: U_{\bbullet/S}^{\SchU}\circ f^*\xrightarrow{\sim} U_{\bbullet/S}^{\SchU}\circ g^*\).
  Then \(\tilde{\alpha}^{-1}\circ \beta \in \Aut(U_{\bbullet/S}^{\SchU}\circ f^*)\).
  Hence, by \autoref{lem: pull-back functor is rigid} \ref{alpha T is in AutfS},
  \(\alpha_T^{-1} \circ \beta_T \in \Aut_T(f)\).
  Since \(\beta_T = \id_T\btimes_T (g^{-1}\circ f)\),
  the isomorphism of \(S\)-schemes \(\alpha_T: T\btimes_{T,f} S \xrightarrow{\sim} T\btimes_{T,g} S\)
  is an isomorphism over \(T\).
  Hence the outer square in the following diagram commutes:
  \[
  \begin{tikzpicture}[auto]
    \node (A) at (0,1.5) {\(T\btimes_{T,f} S\)};
    \node (A') at (0,0) {\(T\btimes_{T,g} S\)};
    \node (B) at (5,1.5) {\(S\)};
    \node (B') at (5,0) {\(S\)};
    \node (C) at (10,1.5) {\(T\)};
    \node (C') at (10,0) {\(T\).};
    \node (D) at (2.5,0.75) {\(\circlearrowleft\)};
    \node (D) at (5,2.1) {\(\circlearrowleft\)};
    \node (D) at (4.4,-0.6) {\(\circlearrowright\)};
    \draw[->] (A) -- node[swap]  {\(\scriptstyle \alpha_T\)} (A');
    \draw[double equal sign distance] (B) to (B');
    \draw[double equal sign distance] (C) to (C');
    \draw[->] (A) -- node  {\(\scriptstyle \mathrm{pr}_S\)} (B);
    \draw[->] (A') -- node  {\(\scriptstyle \mathrm{pr}_S\)} (B');
    \draw[->] (B) -- node  {\(\scriptstyle f\)} (C);
    \draw[->] (B') -- node  {\(\scriptstyle g\)} (C');
    \draw[->] (A) to[bend left=20] node  {\(\scriptstyle \mathrm{pr}_T\)} (C);
    \draw[->] (A') to[bend right=20] node  {\(\scriptstyle \mathrm{pr}_T\)} (C');
  \end{tikzpicture}
  \]
  This implies that \(f = g\), hence that \(\alpha \in \Aut(f^*)\).
  Thus, by \autoref{cor: isom of pb is triv}, \(\alpha = \id_{f^*}\),
  which implies that the functor \(f\mapsto f^*\) is full.

  Finally, we verify that the functor \(f\mapsto f^*\) is essentially surjective.
  Let \(F:\Schb{T}\xrightarrow{\sim} \Schb{S}\) be an equivalence.
  By \autoref{cor: underlying sch equiv commute},
  there exists an isomorphism of functors
  \(\alpha: U_{\bbullet/S}^{\SchU}\circ F \xrightarrow{\sim} U_{\bbullet/T}^{\SchU}\).
  Let \(q_Y:Y\to T\) and \(q_Z:Z\to T\) be objects of \(\Schb{T}\) and
  \(q:Y\to Z\) a morphism in \(\Schb{T}\).
  Write \(\beta_T:F(T)\to S\) for the structure morphism in \(\Schb{S}\).
  Since \(F(T)\) is a terminal object of \(\Schb{S}\),
  the morphism \(\beta_T\) is an isomorphism.
  Write \(f\dfn \alpha_T \circ \beta_T^{-1}:S\xrightarrow{\sim} T\),
  \(\mathrm{pr}_Y:Y\btimes_{q_Y,T,f}S\xrightarrow{\sim} Y\) and
  \(\mathrm{pr}_Z:Z\btimes_{q_Z,T,f}S\xrightarrow{\sim} Z\)
  for the first projections,
  \(p_Y:Y\btimes_{q_Y,T,f}S\to S\) and
  \(p_Z:Z\btimes_{q_Z,T,f}S\to S\)
  for the second projections,
  \(\beta_Y \dfn \mathrm{pr}_Y^{-1}\circ \alpha_Y: F(Y) \xrightarrow{\sim} Y\btimes_{q_Y,T,f} S\),
  \(\beta_Z \dfn \mathrm{pr}_Z^{-1}\circ \alpha_Z: F(Z) \xrightarrow{\sim} Z\btimes_{q_Z,T,f} S\), and
  \(q':Y\btimes_{q_Y,T,f} S\to Z\btimes_{q_Z,T,f} S\) for the unique morphism
  such that \(p_Y = p_Z\circ q'\) and \(q\circ \mathrm{pr}_Y = \mathrm{pr}_Z\circ q'\).
  Then it follows from the definition of \(\alpha\) that the following diagram commutes:
  \[
  \begin{tikzpicture}[auto]
    \node (A) at (0,1.5) {\(F(Y)\)};
    \node (A') at (0,0) {\(F(Z)\)};
    \node (B) at (4,1.5) {\(Y\btimes_{q_Y,T,f} S\)};
    \node (B') at (4,0) {\(Z\btimes_{q_Z,T,f} S\)};
    \node (C) at (8,1.5) {\(Y\)};
    \node (C') at (8,0) {\(Z\)};
    \draw[->] (A) -- node[swap]  {\(\scriptstyle F(q)\)} (A');
    \draw[->] (B) -- node  {\(\scriptstyle q'\)} (B');
    \draw[->] (C) -- node  {\(\scriptstyle q\)} (C');
    \draw[->] (A) -- node  {\(\scriptstyle \beta_Y\)} (B);
    \draw[->] (A') -- node  {\(\scriptstyle \beta_Z\)} (B');
    \draw[->] (B) -- node  {\(\scriptstyle \mathrm{pr}_Y\)} (C);
    \draw[->] (B') -- node  {\(\scriptstyle \mathrm{pr}_Z\)} (C');
    \draw[->] (A) to[bend left=20] node  {\(\scriptstyle \alpha_Y\)} (C);
    \draw[->] (A') to[bend right=20] node  {\(\scriptstyle \alpha_Z\)} (C');
  \end{tikzpicture}
  \]
  Thus the family of isomorphisms
  \((\beta_Y: F(Y) \xrightarrow{\sim} Y\btimes_{q_Y,T,f}S)_{Y\in \Schb{T}}\)
  determines an isomorphism of functors \(\beta:F\xrightarrow{\sim} f^*\).
  This completes the proof of \autoref{thm: isom cat equiv}.
\end{proof}

\newcounter{num}

\setcounter{num}{4}


\begin{thebibliography}{9}
  \bibitem[Ana73]{Anan} 
  \textsc{S. Anantharaman},
  Sch\'emas en groupes, espaces homog\`enes et espaces alg\'ebriques sur une base de dimension 1,
  M\'em. Soc. Math. Fr. \textbf{33} (1973), 5-79.
  %
  \bibitem[Gre64]{AR} 
  \textsc{Martin. J. Greenberg},
  Algebraic Rings,
  Trans. Amer. Math. Soc. \textbf{111} (1964), 472-481.
  %
  \bibitem[EGA IV\(_3\)]{EGA IV3} 
  \textsc{A. Grothendieck},
  \'El\'ements de g\'eometrie alg\'ebrique. IV: \'Etude locale des sch\'emas et des morphismes de sch\'emas,
  Publ. Math. IH\'ES \textbf{28} (1966), 5-255.
  %
  \bibitem[Ha]{Ha} 
  \textsc{R. Hartshorne},
  Algebraic Geometry,
  Grad. Texts in Math. \textbf{52}
  Springer-Verlag,
  New York, 1977.
  %
  \bibitem[Mzk04]{Mzk04} 
  \textsc{S. Mochizuki},
  Categorical representation of locally Noetherian log schemes,
  Adv. Math. \textbf{188} (2004), 222-246.
  %
  \bibitem[vDdB19]{deBr19} 
  \textsc{R. van Dobben de Bruyn},
  Automorphisms of Categories of Schemes,
  Preprint, \href{https://arxiv.org/abs/1906.00921}{arXiv:1906.00921}.
  %
  \bibitem[Stacks]{stacks-project} 
  \textsc{The Stacks Project Authors},
  \href{https://stacks.math.columbia.edu/}{Stacks Project}.
  %
  \bibitem[WW80]{WW} 
  \textsc{W. C. Waterhouse and B. Weisfeiler},
  One dimensional Affine Group Schemes,
  J. Algebra \textbf{66} (1980), 550-568.
\end{thebibliography}
\end{document}